\documentclass[reqno,11pt]{amsart}
\usepackage{latexsym,epsfig,amssymb,amsmath,amsthm,color,url, bbm}
\usepackage[numbers,comma,sort&compress]{natbib}
\usepackage{enumitem}
\usepackage{mathtools}
\usepackage{hyperref}
\usepackage{derivative}
\usepackage{xcolor}
\usepackage{graphicx}
\usepackage{comment}
\makeatletter
\def\widecap{\mathpalette\wide@cap}
\def\wide@cap#1#2{\!\sbox\z@{$#1#2$}
     \mathop{\vbox{\m@th\ialign{##\crcr
\kern0.08em\capfill#1{0.8\wd\z@}\crcr\noalign{\nointerlineskip}
                    $\hss#1#2\hss$\crcr}}}\limits\kern-0.25em}
\def\capfill#1#2{$\m@th\sbox\tw@{$#1($}
  \hss\resizebox{#2}{\wd\tw@}{\rotatebox[origin=c]{90}{\upshape)}}\hss$}
\makeatletter

\allowdisplaybreaks 
\numberwithin{equation}{section}

\newtheorem{theorem}{Theorem}[section]
\newtheorem{corollary}[theorem]{Corollary}
\newtheorem{proposition}[theorem]{Proposition}
\newtheorem{lemma}[theorem]{Lemma}
\theoremstyle{definition}

\newtheorem{assum}[theorem]{Assumption}

\newtheorem{example}[theorem]{Example}
\newtheorem{open}[theorem]{Open Problem}
\theoremstyle{remark}
\newtheorem{remark}[theorem]{Remark}

\newcommand{\floor}[1]{\lfloor #1\rfloor}
\newcommand{\bb}{\mathbb}
\renewcommand{\l}{\left}
\renewcommand{\r}{\right}

\newcommand{\ang}[1]{\langle #1 \rangle}

\newcommand{\ntfloor}{{\lfloor nt \rfloor}}
\newcommand{\hp}{\widehat{p}}
\newcommand{\oeta}{\overline{\eta}}

\newcommand{\mT}{\boldsymbol{T}}
\newcommand{\omT}{\overline{\mT}}
\newcommand{\tmT}{\widetilde{\mT}}
\newcommand{\mW}{\boldsymbol{W}}
\newcommand{\omW}{\overline{\mW}}
\newcommand{\tmW}{\widetilde{\mW}}
\newcommand{\mQ}{\boldsymbol{Q}}
\newcommand{\omQ}{\overline{\mQ}}
\newcommand{\tmQ}{\widetilde{\mQ}}

\DeclareMathOperator{\tr}{trace}
\DeclareMathOperator{\diag}{diag}
\DeclareMathOperator{\Oh}{O}
\DeclareMathOperator{\oh}{o}

\begin{document}
\bibliographystyle{plainnat}

\title[Step reinforced Walk with regularly varying memory]{Functional limit theorems and phase transition\\ for step reinforced random walks\\ with regularly varying memory}
\date{}
\author[A. Majumdar]{Aritra Majumdar}
\address{Aritra Majumdar, \ Theoretical Statistics and Mathematics Unit, \ Indian Statistical Institute, \ Kolkata, India.}
\email{aritram425@gmail.com.}
\author[K. Maulik]{Krishanu Maulik }
\address{Krishanu Maulik, \ Theoretical Statistics and Mathematics Unit, \ Indian Statistical Institute, \ Kolkata, India.}
\email{krishanu@isical.ac.in.}


\begin{abstract}
    We study a class of generalized step reinforced random walks with regularly varying memory. At every step, the walker chooses a step from the past with probability proportional to a given regularly varying sequence, called the memory sequence. Then it either repeats the chosen step with probability $p$ or uses an innovation with probability $1-p$. 

Under finite variance assumption of the innovation steps, we identify a phase transition for the suitably normalized walk, from functional weak convergence to a Gaussian process to almost sure and $L^2$ convergence to a process, which is a random multiple of a power law function. The phase transition is based on the boundedness of the quadratic variation process of a martingale driving the growth of the walk.

We show that the traditional exponential timescale and time-dependent spatial normalization in the critical regime are not universal for all memory sequences. A linear timescale is universal in the critical regime, but there is no universal spatial normalization. The spatial normalization can be smaller, larger or of the same order as the traditional spatial normalization $\sqrt{n \log n}$, observed in the literature.

We also provide functional law of large numbers for the linearly normalized process, viewed at a linearly scaled time. The convergence is almost sure and in $L^1$ under the assumption of a finite first moment of the innovation steps. A stronger finite variance assumption gives us $L^2$ convergence.

 We use the functional law of large numbers and fluctuation results, obtained for the step reinforced walk to prove similar results for the corresponding barycenter process.
\end{abstract}

\subjclass[2020]{Primary 60K50; Secondary 60F05, 60F15, 60F17, 60G42, 62L20, 82B26.}

\keywords{step reinforced random walk, diffusive, superdiffusive, law of large numbers, functional limits, regular variation, Karamata's theorem, r.c.l.l.\ functions, Skorohod topology,  barycenter process}
{
\maketitle
}


\section{Introduction}
Usual random walks consider an i.i.d.\ sequence of \textit{innovation} random variables $\{\xi_n\}_{n\ge 1}$ as the step sizes. We assume the innovations to have finite mean, which will be taken as $0$ without loss of generality. However, such walks do not account for memory effect and anomalous growth. \textit{Step Reinforced Random Walks (SRRW)} was introduced to consider these issues. The dynamics of an SRRW is described as follows. The first step is the innovation step: $X_1 = \xi_1$. At each subsequent epoch $n+1$, with $n\ge1$, we select a past epoch $\beta_{n+1}$, randomly among $\{1, \ldots, n\}$. The variables $\{\beta_n\}_{n\ge2}$ are called the \textit{recollection epochs}. With probability $p$, the walker recalls and repeats the step taken at the epoch $\beta_{n+1}$; else it takes a fresh innovation step $\xi_{n+1}$. We denote the indicator for the successful recall as $\alpha_{n+1}$ and call $p$ as the \textit{memory parameter}. Then, the step at the $(n+1)$-th epoch is given by
\begin{equation}{\label{RW definition}}
    X_{n+1}=\alpha_{n+1}X_{\beta_{n+1}}+(1-\alpha_{n+1})\xi_{n+1}.  
\end{equation}
The corresponding walk is given by $S_0=0$ and $S_n = S_{n-1} + X_n$, for $n\ge 1$.  The random variables $\{\alpha_n\}_{n\ge2}$ are i.i.d.\ Bernoulli random variables with parameter $p$. We assume $\{\alpha_n\}_{n\geq 2}$, $\{\beta_n\}_{n\geq 2}$ and $\{\xi_n\}_{n\geq 1}$ to be independent of each other. 

Traditionally, the recollection epoch $\beta_{n+1}$ is chosen uniformly over the past epochs $\{1, 2, \cdots, n\}$. Bertenghi and Laulin~\cite{bertenghi2024universal} introduced a generalization, where the law of the chosen epoch $\beta_n$ is proportional to a regularly varying sequence $\{\mu_n\}$ of index $\gamma>-1$. The sequence $\{\mu_n\}_{n\ge 1}$ will be called the \textit{memory sequence}. The probability mass function of the recollection epoch $\beta_{n+1}$ is given by $\mathbb{P}(\beta_n=k)={\mu_k}/{\nu_{n}}$, for $1\leq k\leq n$,
where the sequence $\{\nu_n\}$ is given by
$\nu_n:=\sum_{k=1}^n\mu_k$. We call the corresponding random walk, whose increments are given by~\eqref{RW definition}, a \textit{Step Reinforced Random Walk Regularly Varying Memory (SRRW-RVM)}. 

For $p=0$, i.e.\ in the absence of any recollection, we recover the usual mean zero random walk with independent and identically distributed increments. For the other extreme value of the memory parameter $p=1$, the walk grows linearly with the random slope $X_1=\xi_1$, namely $S_n=n \xi_1$. Hence, we shall rule out these two trivialities and consider only the values of the memory parameter $p\in(0,1)$. The asymptotic behavior of the walk is easier to analyze when the innovations are symmetric Rademacher random variables. In that case, the SRRW-RVM will be referred to as the \textit{SRRW-R-RVM}. When the memory sequence is constant, the recollection epochs are uniformly distributed over the past epochs recovering the traditional SRRW. Further, by a suitable reparametrization of the memory parameter, one recovers the \textit{Elephant Random Walk (ERW)} introduced by Sch\"utz and Trimper~\cite{schutz2004}, from the SRRW-R-RVM. In~\cite{Laulin2022}, the author considered a particular choice of memory sequence given by
\begin{equation}{\label{eq: cont prod mem}}
\mu_n=\prod_{i=1}^{n-1}\left(1+{\gamma}/{i}\right)
\end{equation}, symmetric Rademacher innovations and reparametrized memory parameter corresponding to the ERW.

Bertenghi and Laulin~\cite{bertenghi2024universal} established a law of large numbers and functional central limit theorem for the walk when $\widehat p < p < p_c$ and $\gamma\ge0$. In this case, the walk is diffusive. The present article extends this analysis to the remaining parameter regimes and, most importantly, develops a general theory for the critical case \(p=p_c\). At criticality, the slowly varying part of the memory sequence determines whether the normalized walk has a Gaussian weak limit or an almost-sure random limit and determines the spatial normalization.

For the memory sequence given by~\eqref{eq: cont prod mem}~\cite{Laulin2022} established a phase transition at the critical value $p_c=(\gamma+1/2)/(\gamma+1)$ of the memory parameter, in the framework of ERW. For the subcritical regime $0\le p<p_c$, the walk is diffusive with a Gaussian process limit. The covariance kernel of the limiting process has a removable singularity at $\widehat p = \gamma/(\gamma+1)$, which was not treated separately in~\cite{Laulin2022}. For this special power-law memory, the critical regime was also studied in~\cite{Laulin2022}, including almost-sure upper-order behaviour and a functional limit under an accelerated timescale. Our critical results are different in scope: we treat arbitrary regularly varying memory sequences and show that its slowly varying component can determine both the critical normalization and the mode of convergence.

\subsection{Main contributions of the present work} 
The present article develops a law of large numbers and fluctuation theory for the SRRW-RVM introduced in~\cite{bertenghi2024universal}, covering all values of memory parameter $p\in[0,1]$ and all regularly varying memory sequences $\{\mu_n\}$ with index $\gamma>-1$. 

Under the additional assumption of finite second moment of the innovation sequence, we obtain the critical value of the memory parameter $p_c$. The critical regime provides a bridge between Gaussian fluctuation behavior in the subcritical phase and almost sure and $L^2$ random limiting behavior in the supercritical phase. However, in the critical regime, depending on the memory sequence, either behavior may occur. We define a martingale sequence
\begin{equation}
    \label{eq: def M}
    M_n=a_n\sum_{k=1}^nX_k\mu_k, \quad n\ge 1,
\end{equation}
where the sequence $\{a_n\}$ is regularly varying of index $-p(\gamma+1)$ (see Theorem~4 of~\cite{Bojanic1973}), given by $a_1=1$ and
\begin{equation} \label{contd prod}
a_n:=\prod_{i=1}^{n-1}\left(1+p\frac{\mu_{i+1}}{\nu_i}\right)^{-1},\quad  \text{for $n\ge2$}.
\end{equation}
The martingale sequence $\{M_n\}$ controls the asymptotic growth of the SRRW-RVM. The normalization $\{v_n^2 = \sum_{i=1}^n a_i^2 \mu_i^2\}$ for the asymptotic variance of this martingale sequence is regularly varying of index $2\gamma+1-2p(\gamma+1)$ and determines the phase transition between two modes of convergence, including the limiting behavior in the critical regime. Note that, by Karamata's theorem, $\{v_n\}$ has positive index of regular variation and is unbounded for the subcritical regime $p<p_c$, while it has negative index of regular variation and is bounded for the supercritical regime $p>p_c$. For unbounded $\{v_n\}$, the normalized SRRW-RVM converges weakly to a Gaussian process, while when $\{v_n\}$ is bounded, the normalized SRRW-RVM converges almost surely and in $L^2$ to a process which is a nondegenerate random scalar multiple of a deterministic power law function. In the critical regime, the limiting process is a scalar random multiple of the square root function irrespective of the sequence $\{v_n\}$ being bounded or not. Whenever $\{v_n\}$ is unbounded, the scalar random variable is Gaussian, while it need not be so when $\{v_n\}$ is bounded. The sequence $\{v_n\}$ is slowly varying in the critical regime and is a function of the slowly varying part of the memory sequence $\{\mu_n\}$. The phase transition is controlled not merely by the memory parameter \(p\), but by the asymptotic variance scale of the natural martingale. At the critical value \(p=p_c\), the slowly varying asymptotic variance scale, which is a function of the slowly varying component of the memory sequence, determines both the normalization and even the mode of convergence.

The results show that the critical regime is sensitive to the slowly varying component of the memory sequence. Thus, although the critical value \(p_c\) and the square-root form of the limiting process are universal within the regularly varying class, the critical spatial normalization and the mode of convergence are not. This provides a finer classification of critical step-reinforced random walks than is available for the previously studied power-law or uniform memories.

The exponential timescale traditionally used in the critical regime does not provide a universal nondegenerate scaling limit. However, we show that the linear timescale is universal. At linear timescale in the critical regime, the spatial normalization is time independent. Yet, there is no universal spatial normalization in the critical regime, but depends on the slowly varying asymptotic variance scale $\{v_n\}$.

All the limit results are in terms of the convergence of the normalized process in the space of r.c.l.l.\ functions. We provide complete and rigorous proof of the tightness-at-zero argument to establish convergence in $D([0,\infty))$. This requires an estimate for equicontinuity in probability at $0$. The details have been provided in Section~\ref{ERW}. For proving almost sure and $L^2$ convergence of the process, convergence of supremum over compact sets has also been shown in this article. This analysis is also missing in the ERW/SRRW literature. Using the functional convergence, further limits of continuous functionals can be obtained easily.

Following~\cite{Aguech2025, Hu2024}, we obtain the almost sure and $L^1$ laws of large numbers only under first moment assumptions for the SRRW-RVM as well, but we do so additionally in the process sense. This process law of large numbers also yields corresponding convergence results for continuous functionals.

 We also apply the functional versions of the laws of large numbers and the fluctuation results of the SRRW-RVM to it's associated barycenter process, using a transfer theorem developed in~\cite{Lo2018}. The process convergence results for the fluctuations of the barycenter process in $L^2$ require careful analysis, especially near zero, which has been provided in details.
\subsection{Notations and Conventions} \label{subsec: notation}
We close this section by introducing some notations used in this work. All vectors will be row vectors of appropriate dimensions, which will be clear from the context. The transpose of the row vector $\boldsymbol{x}$ will be denoted by $\boldsymbol{x}'$. We denote the vector of all zeros (ones) by $\boldsymbol{0}$ ($\boldsymbol{1}$). We denote a diagonal matrix with diagonal elements $\{a_1, \ldots, a_d\}$ as $\diag(a_1, \ldots, a_d)$. For a matrix $\boldsymbol A$, we denote its trace or the sum of the diagonal elements by $\tr(\boldsymbol{A})$.

All empty sums and empty products will be $0$ and $1$ respectively.
The indicator function of a set $A$ will be $\mathbb 1_A$. For nonnegative real sequences $\{a_n\}$ and $\{b_n\}$, we write $a_n\sim b_n$ if $\lim_{n\rightarrow \infty}{a_n}/{b_n}=1$.

We write $X{=}_{\rm d} Y$ when the random variables $X$ and $Y$ have same distribution. In this article, $Z$ will be a standard normal variable and $(\mathcal B(t), t\ge0)$ will be a standard Brownian motion. 

We also denote by $D(I)$, the space of all r.c.l.l.\ functions defined on an interval $I \subseteq [0,\infty)$ and taking values in $\mathbb{R}^d$. The dimension will be clear from the context. It is equipped with the Skorohod topology making it a complete separable metric space.  The corresponding metric for $I=[0,\infty)$ will be denoted as $d_J$. For more details, see~\cite[Chapter $3$, Section $16$]{Billingsley1999} and \cite[Chapters $3,12$]{Whitt2002}. The Borel sigma algebra (generated by the open sets under the Skorohod topology) will be denoted by $\mathcal{D}$. We shall also consider the norm $\|\cdot\|_\infty$ which gives the topology of uniform convergence on compact sets for functions defined on $[0,\infty)$.

Further, $D_0([0,\infty))$ will denote the subspace of $D([0,\infty))$ where the functions start at zero at $t=0$ and $\mathcal D_0$ will be the subset topology. 

We also occasionally work with the space of continuous $\bb R^d$ valued functions on $[0,\infty)$, denoted by $C([0,\infty))$, and equipped with the topology giving uniform convergence on compact sets. The topologies given by $d_J$ and $\|\cdot\|_\infty$ coincide on $C([0,\infty))$. We also write $C_0([0, \infty))$ as the space of all continuous functions starting at zero at time $t=0$.

We denote the convergence in finite dimensional distribution and convergence in law, as probability measures on the appropriate metric space, by $\to_{{\rm fdd}}$ and $\to_{{\rm w}}$ respectively. The convergence of the associated random variables will have the same notations as well. The corresponding convergence of the random variables taking values in appropriate metric spaces; in probability, almost surely, in $L^1$ and in $L^2$ will be denoted by $\to_{{\rm P}}$, $\to_{{\rm a.s.}}$, $\to_{\rm L^1}$ and $\to_{\rm L^2}$ respectively. For $L^1$ or $L^2$ convergence in $D([0,\infty))$ space, we shall clearly indicate the metric used. For two functions $f, g \in D([0, \infty))$, we have $d_J(f,g) \le \|f-g\|_\infty$ and $L^1$ or $L^2$ convergence in $(D([0,\infty), \|\cdot\|_\infty)$ will imply the same in $(D([0,\infty), d_J)$ as well.
\subsection{Outline of the rest of the article}
Section~\ref{sec: results} presents the main results of this article along with some related remarks. Section~\ref{as fluc} treats the fluctuations of the process when $\{v_n\}$ is bounded. The proofs here do not require the innovations to be symmetric Rademacher variables. In Section~\ref{sec: SLLN}, we prove the laws of large numbers. For simplicity of presentation and ease of understanding, we first prove these results for innovation sequences with finite variance. In Section~\ref{ERW}, we prove the weak process convergence of the normalized process. Again, for simplicity, we first prove these results in this section for symmetric Rademacher innovations. In Section~\ref{sec: cont}, we prove the continuity of the limiting processes in terms of the memory parameter and the index of regular variation of the memory sequence, while we study the law of large numbers and fluctuations of the barycenter process, associated with the SRRW-RVM in Section~\ref{sec: barycenter}. 
In Section~\ref{sec: examples}, we provide various examples under the critical regime to illustrate different possible normalizations and modes of convergence. 
Finally, in Section~\ref{sec: SRRW}, we complete the proofs of the results considered in Sections~\ref{sec: SLLN} and~\ref{ERW} for general innovation sequences.

\section{Main Results}
\label{sec: results}
In this section, we provide the main results and some remarks on their interpretation.

We first provide the functional law of large numbers. This result extends the earlier SLLN -- which proved almost sure convergence only, assuming finite second moment of the innovation sequence -- see Theorem~$3.2$ of~\cite{bertenghi2024universal}, where it was established only for the parameter regime $p\in \left( \widehat{p}, p_c \right)$ and $\gamma\geq 0$. 
Note that for $p=0$, we recover the SLLN for the usual random walk with independent and identically distributed increments $\{\xi_n\}$.
\begin{theorem}[Strong law of large numbers] \label{thm: SLLN}
    Let $\{S_n\}_{n\geq 0}$ be an SRRW-RVM with zero mean innovation sequence $\{\xi_n\}_{n\ge 1}$. 
    
     For all $p\in[0,1)$, the process ${S_{\lfloor n\cdot \rfloor}}/n$ converges to the zero process almost surely in $(D([0,\infty)),\mathcal D)$, as well as in $L^1$ in $(D([0,\infty)), \|\cdot\|_\infty)$ and hence also in $(D([0,\infty)), d_J)$. 
    
    For $p=1$, we have $(S_{\floor{nt}}/n: t\ge0) \to (\xi_1 t: t\ge0)$ almost surely in $(D([0,\infty)),\mathcal D)$ in $L^1$ in $(D([0,\infty)), \|\cdot\|_\infty)$ and hence in $(D([0,\infty)), d_J)$ as well. 

    Also, if the innovations have finite variance, the above convergence results hold in $(D([0,\infty)), \|\cdot\|_\infty)$ in $L^2$ and hence in $(D([0,\infty)), d_J)$ as well.
\end{theorem}

For the results in the remainder of the section, we assume finite second moments of the innovations, taken to be $1$ without loss of generality. The following result holds when $p\in(p_c, 1]$ and, for appropriate choices of $\{\mu_n\}$, also when $p=p_c$. 

\begin{theorem}[Almost sure and in \texorpdfstring{$L^2$}{L2} convergence for bounded \texorpdfstring{$\{v_n\}$}{{vn}} regime]
\label{thm: bdd quad var}
     Let $\{S_n\}_{n\geq 0}$ be an SRRW-RVM with the zero mean innovation sequence $\{\xi_n\}_{n\ge 1}$ satisfying $\mathbb{E}\xi_1^2=1$. If $\{v_n\}$ is bounded, then
     \begin{enumerate}[label=(\alph*)]
     \item The martingale sequence $\{M_n\}$, given in~\eqref{eq: def M}, has almost sure and $L^2$ limit $M_\infty$. \label{Mn limit}
     \item When the innovation sequence is symmetric Rademacher, the limiting random variable $M_\infty$ is platykurtic, that is, it has kurtosis strictly smaller than $3$ and is not Gaussian. \label{platykurtic}
     \item For the normalized SRRW-RVM process $a_n\mu_nS_{\floor {n\cdot}}$, we have, in $(D([0,\infty)),\mathcal{D})$,
     \[         \left(a_n\mu_nS_{\lfloor{nt}\rfloor}:t\geq 0\right) \to\left(\frac{p(\gamma+1)}{p(\gamma+1)-\gamma}t^{p(\gamma+1)-\gamma} M_{\infty}: t\geq 0\right)
     \]
      almost surely in $(D([0,\infty)),\mathcal{D})$ and in $L^2$ in $(D([0,\infty)), \|\cdot\|_\infty)$ as well as in $(D([0,\infty)), d_J)$.
     \label{supercrit proc}
     \end{enumerate}
     \end{theorem}
     \begin{remark} \label{rem: p=1}
    Theorem~\ref{thm: bdd quad var} also includes the case $p=1>p_c$. Then,~\eqref{contd prod} simplifies to $a_n = \mu_1/\nu_n$, and therefore 
    $a_n \mu_n \sim {(\gamma+1)\mu_1}/n$. Also, from~\eqref{eq: M mart}, we have $M_n = a_n \nu_n \xi_1 = \mu_1 \xi_1$ for all $n=1,2,\cdots, \infty$. One can then readily see that $(S_{\floor{nt}}/n : t\ge 0)\rightarrow (t\xi_1 : t\ge0)$ almost surely and in $L^2$, which agrees with Theorem~\ref{thm: SLLN} for $p=1$. 
\end{remark}

We consider unbounded $\{v_n\}$ next. 
We first consider the subcritical regime $p \in [0, p_c)$. In this regime, the natural spatial normalization is diffusive. The diffusive limit for the process has already been established in~\cite{bertenghi2024universal} for the parameter values $p\in\left(\widehat{p}, p_c\right)$ and $\gamma\ge 0$. We extend this result to the entire subcritical regime $p\in[0,p_c)$ and $\gamma>-1/2$. 
\begin{theorem}(Gaussain functional limit in the subcritical regime) \label{thm: subcrit}
         Let $\{S_n\}_{n\geq 0}$ be an SRRW-RVM with the zero mean innovation sequence $\{\xi_n\}_{n\ge 1}$ satisfying $\mathbb{E}\xi_1^2=1$. Then, for all $p\in \left[0,p_c\right)$, we have 
         \[
            \left( n^{-1/2} S_{\lfloor{nt}\rfloor}: t\geq 0\right)\to_{\rm{w}} \left(\mathcal{G}(t): t\geq 0\right),
        \]
        in $(D([0, \infty)), \mathcal D)$, where $\left(\mathcal{G}(t): t\geq 0\right)$ is a continuous and centered Gaussian process, starting from the origin, with the covariance kernel given by, for $0\leq s\leq t$,
        \begin{equation}\label{eq: cov kernel}   \mathbb{E}\left(\mathcal{G}(s)\mathcal{G}(t)\right)=
            \begin{cases}
            \frac{s}{(1-p)(\gamma-p(\gamma+1))}\left(\gamma - \frac{p((2-p)(\gamma+1)-1)}{(2(1-p)(\gamma+1)-1)}\left(\frac{s}{t}\right)^{\gamma-p(\gamma+1)}\right), &\text{for $p\ne\widehat{p}$,}\\    s\left(\gamma^2 +(\gamma+1)^2 - \gamma(\gamma+1) \log \frac{s}{t}\right), &\text{for $p=\widehat{p}$.}           
            \end{cases}
        \end{equation}
\end{theorem}

\begin{remark}{\label{CLT in diffusive regime}}
    In the setup of Theorem~\ref{thm: subcrit}, $S_n/\sqrt{n}$ converges weakly to a centered normal distribution with the limiting variance $\varsigma^2$ given by:
    \begin{equation*} 
        \varsigma^2=
        \begin{cases}
            \frac{2\gamma+1-p}{(1-p)(2(1-p)(\gamma+1)-1)}, & \text{for } p\neq \widehat{p},\\
            2\gamma^2+2\gamma+1, & \text{for } p=\widehat{p}.
        \end{cases}
    \end{equation*}
\end{remark}
\begin{remark}
    In the subcritical regime, the limit process depends on the memory parameter $p$, as well as, the regularly varying memory sequence $\{\mu_n\}$, but only through its index of regular variation $\gamma$.
\end{remark}

Finally, we consider the critical regime $p=p_c$ with unbounded $\{v_n\}$. In this case, we use the spatial normalization 
\begin{equation}
\sigma_n := \frac{v_n}{a_n\mu_n}. \label{eq: ell}
\end{equation}
The critical regime differs qualitatively from both the subcritical and the supercritical regimes. Here the sequence $\{v_n\}$ is slowly varying, and its boundedness is no longer determined by the index of regular variation of the sequence $\{\mu_n\}$. In fact, the critical regime is the only such regime. The spatial normalization is superdiffusive for $p=p_c$ and unbounded $\{v_n\}$. 

\begin{theorem}[Gaussian critical limit with unbounded \texorpdfstring{$\{v_n\}$}{{vn}}] \label{thm: crit unbdd}
     Let $\{S_n\}_{n\geq 0}$ be an SRRW-RVM with the zero mean innovation sequence $\{\xi_n\}_{n\ge 1}$ satisfying $\mathbb{E}\xi_1^2=1$. If $\{\mu_n\}$ is such that $\{v_n\}$ is unbounded for $p=p_c$, then, with $\sigma_n$ given by~\eqref{eq: ell} and $Z$ a standard Gaussian random variable, we have
     \[         
     \left(\frac{ {S_{\lfloor{nt}\rfloor}}}{\sigma_n}: t\geq 0\right)\to_{\rm{w}} \left(\sqrt{t}(2\gamma+1)Z: t\geq 0\right) \quad
     \text{in $(D([0,\infty)),\mathcal{D})$.}
     \]
\end{theorem}
\begin{remark}
The above theorem in the critical regime with unbounded \(\{v_n\}\) shows that the linear timescale and a time-independent spatial normalization lead to a nondegenerate Gaussian process limit and exponential acceleration of time is not required. However, as we demonstrate in Section~\ref{sec: examples}, the spatial normalization $\sigma_n$ does not admit a universal explicit form, but depends strongly on the slowly varying component of the memory sequence $\{\mu_n\}$. 
\end{remark}
Furthermore, it is easy to check using Karamata's theorem that
\begin{equation} \label{eq: sigma Karamata}
\sigma_n = \frac{v_n}{a_n \mu_n} \sim
\begin{cases}
    \frac{1}{2(1-p)(\gamma+1)-1} \sqrt{n}, &\text{if $p<p_c$},\\
    \frac{1}{a_n \mu_n}v_\infty, &\text{if $\{v_n\}$ is bounded with $v_n\uparrow v_\infty$}.
\end{cases}
\end{equation} 
This suggests that the spatial normalization $\sigma_n$ works for all $p\in [0,1]$ and $\gamma>-1$.  
Combining Theorems~\ref{thm: bdd quad var},~\ref{thm: subcrit} and~\ref{thm: crit unbdd} gives the following unified phase diagram.
\begin{corollary}[Phase diagram with common spatial normaization] \label{cor: phase}
     Let $\{S_n\}_{n\geq 0}$ be an SRRW-RVM with the zero mean innovation sequence $\{\xi_n\}_{n\ge 1}$ satisfying $\mathbb{E}\xi_1^2=1$. Then, with $\mathcal G$ being the weak process limit in Theorem~\ref{thm: subcrit}, $Z$ a standard normal random variable and $M_\infty$ the martingale limit in Theorem~\ref{thm: bdd quad var}, we have, in $(D([0,\infty)),\mathcal{D})$,
     \begin{align*}
     &\left(\frac{ {S_{\lfloor{nt}\rfloor}}}{\sigma_n}: t\geq 0\right) \to \l(\mathcal{L}(t): t\ge0\r)\\
     := &\begin{cases}
         \l(\sqrt{(2(1-p)(\gamma+1)-1)} \mathcal{G}(t): t\ge0\r) \text{weakly}, &\!\!\text{if $0\le p<p_c$ {\rm(}with unbounded $\{v_n\}${\rm)}},\\
         \l((2\gamma+1) Z\sqrt{t}: t\ge0\r) \text{weakly}, &\!\!\text{if $p=p_c$ and $\{v_n\}$ is unbounded},\\
         \l((2\gamma+1) \frac{M_\infty}{v_\infty} \sqrt{t}: t\ge0\r) \text{almost surely}, &\!\!\text{if $p=p_c$ and $\{v_n\}$ is bounded},\\
         \left(\frac{p(\gamma+1)}{p(\gamma+1)-\gamma} \frac{M_\infty}{v_\infty} t^{p(\gamma+1)-\gamma}: t\geq 0\right) \text{almost surely}, &\!\!\text{if $p_c<p\le1$ {\rm(}with bounded $\{v_n\}${\rm)}}.
     \end{cases}
     \end{align*}
      For bounded $\{v_n\}$, the above convergence results also hold in $L^2$ in $(D([0,\infty)), \|\cdot\|_\infty)$ and hence in $(D([0,\infty)), d_J)$.
\end{corollary}
    
For a fixed memory sequence $\{\mu_n\}$, the limiting processes $\mathcal G$ and $\mathcal L$ show continuity in the memory parameter $p\in[0,p_c)$. (Note that the limiting processes depend both on the memory parameter $p$ and the index of regular variation $\gamma$ of the memory sequence $\{\mu_n\}$, but not on its slowly varying part.) For that purpose, we indicate the dependence on the memory parameter $p$ in the limiting processes explicitly as, $\mathcal G^{(p\bullet)} \equiv \mathcal G$ and $\mathcal L^{(p\bullet)} \equiv \mathcal L$. We also study the behavior of the limiting Gaussian process as a function of $\gamma$, for a fixed memory parameter $p$. We indicate the dependence as $\mathcal G^{(\bullet\gamma)} \equiv \mathcal G$.
\begin{theorem}[Continuity of Gaussian limit process] \label{thm: tightness for G at crit}
The limiting processes satisfy the following properties:
\begin{enumerate}[label=(\alph*)]
\item For fixed $\gamma$, both the processes $\mathcal G^{(p\bullet)}$ and $\mathcal L^{(p\bullet)}$ are continuous in $p \in [0, p_c)$ under the topology of weak convergence, as processes in $C([0,\infty))$. 

Furthermore, the process $\mathcal L^{(p\bullet)}$ is also continuous from below at $p = p_c$, while the collection of processes $\{\mathcal{G}^{(p\bullet)}: p\in (p_c-\epsilon, p_c) \}$, is not tight, for any $\epsilon>0$. \label{thm: cont from below}

\item For any fixed $p<1$, the SRRW-RVM is in subcritical regime for all large enough $\gamma$. As $\gamma\to\infty$, the limiting Gaussian process $\mathcal G^{(\bullet\gamma)}$ converges weakly to $\mathcal B/\sqrt{1-p}$ in $C([0,\infty))$, where $\mathcal B$ is a standard Brownian motion.
\end{enumerate}
\end{theorem}
\begin{remark}
The Gaussian limit process in the subcritical regime under the $\sqrt{n}$-normalization is not tight as $p\uparrow p_c$. However, once we use the common normalization $\sigma_n$ for all the regimes, the Gaussian limit process in the subcritical regime converges to the Gaussian limit process in the critical regime weakly as a process, when $p\uparrow p_c$, provided $\{v_n\}$ is unbounded at $p=p_c$. 

The limit process under the linear timescale, when $\{v_n\}$ is bounded, is a nondegenerate random multiple of a deterministic power law function. The index of the power law function in the critical regime is $1/2$, irrespective of $\{v_n\}$ being bounded or unbounded. The index $1/2$ of the power law function in the critical regime is the limit of the index of the power law function in the supercritical case, but the random variable is distributed as Gaussian, whenever $\{v_n\}$ is unbounded, unlike the almost sure limit. In case of almost sure limit and symmetric Rademacher innovations, the limit is non-Gaussian and platykurtic. 
The above continuity property of the limiting process $\mathcal L$ and similarity of the structure of the limiting process in the critical regime between the bounded and unbounded $\{v_n\}$ show that the critical regime provides a bridge between the subcritical weak Gaussian process limit and the supercritical almost sure and $L^2$ limit.

We also record interesting behavior as the index of regular variation $\gamma$ of the memory sequence increases to $\infty$. In such a limiting scenario, the random variable $\beta_n$ puts nearly entire weight on $n-1$. Thus, the SRRW-RVM selects the same last step a geometric number of times before getting replaced by a fresh innovation to be repeated again an independent geometric number of times. Thus, we get fewer number (based on the memory parameter $p$) of distinct i.i.d.\ steps, which is an independent geometric multiple of an innovation. This suggests the $\sqrt{n}$-diffusive normalization and a multiple (dependent on $p$) of the standard Brownian motion as the limiting process.
\end{remark}

\begin{remark}     
In~\cite{Laulin2022}, a generalization of the ERW, which they refer to as the Amnesic Elephant Random Walk (AERW) was considered. The SRRW-RVM with symmetric Rademacher innovations, memory parameter $p$ and the memory sequence $\mu_n={\Gamma(n+\gamma)} / (\Gamma(\gamma+1) \Gamma(n))$ is equivalent to the AERW with the corresponding memory parameter $p'=(p+1)/{2}$. Then, with $p'\in [1/2, 1]$, Theorems~2.1,~2.3 and~2.7 of \cite{Laulin2022} follow directly from Theorems~\ref{thm: SLLN},~\ref{thm: subcrit} and~\ref{thm: bdd quad var} above respectively. Our proofs allow arbitrary regularly varying memory sequence $\{\mu_n\}$ of index $\gamma>-1$ and also fill the gap in the proof corresponding to $p=\gamma/(\gamma+1)$ in the setup of the present article, which corresponds to $p'=(\gamma+1/2)/(\gamma+1)$ in the setup of~\cite{Laulin2022}. The case $p'< 1/2$ corresponds to a negative reinforcement in our model (where the incremental steps are given by $X_{n+1}=-\alpha_{n+1}X_{\beta_{n+1}}+(1-\alpha_{n+1})\xi_{n+1}$ in~\eqref{RW definition}). Although our results are presented in the context of a positively reinforced step-reinforced random walk, the same techniques can also be used to obtain analogous limit theorems (for example, Theorems~2.1,~2.3 and~2.7 of~\cite{Laulin2022}) for the negatively reinforced walk. In fact, such a walk is diffusive for all $p\in [0,1]$ and $\gamma>-1$.
\end{remark}

 We end this section with results on the barycenter process. We define the barycenter sequence $\{\overline S_n: n\ge0\}$ of SRRW-RVM $\{S_n: n\ge0\}$ as 
\begin{equation*}
    \overline S_0=0 \qquad \text{and} \qquad \overline S_n = \frac1n \sum_{i=1}^n S_i, \quad \text{for $n\ge 1$}.
\end{equation*}
We then have the following functional law of large numbers and phase diagram for the barycenter process.
\begin{theorem}[Barycenter process]
    \label{thm: barycenter}
    Let $\{\overline S_n\}_{n\ge 0}$ be the barycenter process associated with SRRW-RVM $\{S_n\}_{n\ge 0}$ with a zero-mean innovation sequence $\{\xi_n\}_{n\ge 1}$.

    If the memory parameter $p<1$, then we have
    \[
    \l( \frac1n \overline S_\ntfloor: t\ge 0\r) \to \l(0: t\ge0\r)
    \]
    almost surely in $(D([0, \infty)), \mathcal D)$ and in $L^1$ in $(D([0, \infty)), \|\cdot\|_\infty)$ and hence in $(D([0, \infty)), d_J)$. For $p=1$, the limit process is $(t \xi_1/2: t\ge0)$. If the innovation sequence has finite variance, the above convergence results are also in $L^2$ in $(D([0, \infty)), \|\cdot\|_\infty)$ and hence in $(D([0, \infty)), d_J)$.

    If the innovations are zero-mean and finite-variance, we also have
    \[
     \left(\frac1{\sigma_n}{\overline{S}_{\floor{nt}}}: t\geq 0\right) 
     \begin{cases}
         \to_{{\rm w}} \l(\mathcal{K}(t): t\ge0\r), &\text{if $0\le p<p_c$ {\rm(}unbounded $\{v_n\}${\rm)}},\\
         \to_{{\rm w}} \l(\frac23(2\gamma+1) Z\sqrt{t}: t\ge0\r), &\text{if $p=p_c$ and unbounded $\{v_n\}$},\\
         \to_{{\rm a.s.}} \l(\frac23(2\gamma+1) \frac{M_\infty}{v_\infty} \sqrt{t}: t\ge0\r), &\text{if $p=p_c$ and bounded $\{v_n\}$},\\
         \to_{{\rm a.s.}} \left(\frac{\gamma+b}{b(1+b)} \frac{M_\infty}{v_\infty} t^{b}: t\geq 0\right), &\text{if $p_c<p\le1$ {\rm(}bounded $\{v_n\}${\rm)}},
     \end{cases}
     \]
    in $(D([0, \infty)), \mathcal D)$, where $b=p(\gamma+1)-\gamma$. When $\{v_n\}$ is bounded, the convergence results also hold in $L^2$ in $(D([0, \infty)), \|\cdot\|_\infty)$ and hence also in $(D([0, \infty)), d_J)$. The subcritical limit process $\mathcal K$ is a continuous, real valued, centered Gaussian process starting at the origin. The covariance kernel of the process $\mathcal K$, for $0\le s\le t$, is given by
\begin{multline}
    \bb E (\mathcal K(s) \mathcal K(t) )\\ = \begin{cases}
s (1-2b) \l[ \frac{(\gamma+2)(\gamma-1)}{6(1+b)(2-b)} \cdot \frac{s}{t} + \frac{1}{b(1-b)} \l\{ \frac{(\gamma+b)(\gamma+1-b)}{(1+b)(2-b)(1-2b)} \l( \frac{s}{t} \r)^{-b} - \frac{\gamma( \gamma +1 )}{2}\r\} \r], & \text{for } p\neq \hp \\
s \l[ \frac{(\gamma+2)(\gamma-1)}{12} \cdot \frac{s}{t} + \frac{3\gamma(\gamma+1)+2}{4} -\frac{\gamma(\gamma + 1)}{2} \log\l(\frac{s}{t}\r)\r], & \text{for } p=\hp.
    \end{cases}\label{eq: lim cov} 
\end{multline}
The limit random variable $M_\infty$ in case of bounded $\{v_n\}$ is as given in Theorem~\ref{thm: bdd quad var}.
\end{theorem}
\begin{remark}
    As noted earlier, the walk considered in~\cite{Laulin2022} is a special case of the SRRW-RVM considered in this article. The barycenter process corresponding to the walk in~\cite{Laulin2022} was studied in~\cite{Chen2023}, where a continuous mapping was used to transfer the process weak convergence of the underlying walk to the barycenter process. However, the analysis in~\cite{Chen2023} had several limitations. The analysis did not consider the error term due to the transform. The covariance kernel of the Gaussian process limit in the subcritical regime was not correct. In~\cite{Chen2023}, exponential timescale was used in the critical regime, but it was not explained how the continuous mapping theorem was applied. In the supercritical regime, only process weak convergence was obtained. Even then, the application of the continuous mapping theorem was flawed, as the convergence of the underlying process was only in finite dimensional distributions. We settle these issues about the barycenter process and provide a complete picture. The phase diagram for the barycenter process is the same as that of SRRW-RVM and the spatial normalizations are same as well. The limit processes in the critical and the supercritical regimes are constant multiples of the corresponding limit processes for SRRW-RVM.
\end{remark}

\section{Fluctuations when Asymptotic Variance Scale is Bounded}
\label{as fluc}
In this section, we consider the fluctuations of the SRRW-RVM $\{S_n\}$ with bounded asymptotic variance scale $\{v_n\}$, as given in Theorem~\ref{thm: bdd quad var}. We first introduce the required notation and establish several preliminary results.
The first result shows that the steps of SRRW-RVM $\{S_n\}$ are dependent due to recollection, but are identically distributed.
\begin{lemma}{\label{Variance of the increments}}
For all $n\ge 1$, the marginal distribution of the $n^{th}$ increment $X_n$ of an SRRW-RVM $\{S_n\}_{n\geq 0}$, is same as the common distribution of the innovation sequence.
\end{lemma}
\begin{proof}
The independence of $\alpha_n$, $\beta_n$ and $\xi_n$ in~\eqref{RW definition} gives, for $n\ge 1$,
$\mathbb{P} (X_{n+1}\le x | \mathcal F_{n}) = (1-p) \mathbb{P} (\xi_1\le x) + {p}{\nu_n^{-1}} \sum_{k=1}^n \mu_k \mathbb 1_{[X_k\le x]}$.
Unconditioning and noting that $X_1=\xi_1$, we have $X_n {=}_{\rm d} \xi_1$, using induction. 
\end{proof}

In order to study the martingales central to the discussion in this article, we introduce the filtration $\{\mathcal F_n\}_{n\ge0}$ associated to the process $\{S_n\}_{n\ge0}$, where $\mathcal F_0$ is the trivial $\sigma$-field, $\mathcal{F}_1=\sigma(\xi_1)$ and, for $n\geq 2$, 
$\mathcal{F}_n =\sigma\l(\{\xi_k\}_{k=1}^n,\{\beta_k\}_{k=2}^ n, \{\alpha_k\}_{k=2}^ n\r)$.

From~\eqref{RW definition}, we immediately get, for $n\geq 1$,
\begin{equation}{\label{Conditional Expectation}}
    \mathbb{E}(X_{n+1}|\mathcal{F}_{n})=\frac{p}{\nu_n}\sum_{k=1}^nX_k\mu_k = \frac{p}{\nu_n} Y_n,
\end{equation}
where $Y_n = \sum_{k=1}^n \mu_k X_k$. This gives $\mathbb{E}(Y_{n+1}|\mathcal{F}_{n})= \l(1+{p\mu_{n+1}}/{\nu_n}\r) Y_n$. Then, it follows immediately from~\eqref{eq: def M} and~\eqref{contd prod} that
\begin{equation}
    \label{eq: M mart}
    \mathbb{E}(M_{n+1}|\mathcal{F}_{n})= M_n,
\end{equation}
making $\{M_n\}$ a martingale. The corresponding martingale differences are given by 
\begin{equation}
       \label{martingale diff M_n} \Delta M_n=a_n\mu_n\left(X_n-\mathbb{E}(X_n|\mathcal{F}_{n-1})\right), \quad \text{for $n\ge 2$} \qquad \text{and} \quad \Delta M_1 = \mu_1 X_1.
\end{equation}

Further, considering the conditional expectation of $S_n$, given by
$\mathbb E( S_{n+1}\mid \mathcal F_n) = S_n + p M_n/ (a_n\nu_n)$,
we have the next martingale of interest, which is given by, 
\begin{equation}
    L_1 = X_1, \quad \text{and for $n\ge 2$,} \quad L_n=  S_n - p \sum_{k=1}^{n-1} \frac1{a_k\nu_k} M_k =\sum_{k=1}^n(X_k-\mathbb{E}(X_k|\mathcal{F}_{k-1})). \label{conn L S}
\end{equation}
Note that $\{M_n\}$ is a martingale transform of $\{L_n\}$ with
\begin{equation} \label{eq: M mg transf}
    \Delta M_n = a_n \mu_n \Delta L_n, \quad \text{for $n\ge 1$.}
\end{equation}

The following proposition summarizes the growth of the variance of $M_n$, which in turn determines the phase transition based on the boundedness of the sequence $\{v_n\}$.
\begin{proposition}{\label{Asymptotics of martingale M_n}}
    Let $\{S_n\}_{n\geq 0}$ be an SRRW-RVM with the zero mean innovation sequence $\{\xi_n\}_{n\ge 1}$ satisfying $\mathbb{E}\xi_1^2=1$. When $\{v_n\}$ is bounded, $\{M_n\}_{n\ge 1}$ is an $L^2$-bounded martingale, while for unbounded $\{v_n\}$, we have 
    \begin{align} \label{eq: Mn asymp}
        \mathbb{E}(M_n^2) \sim \begin{cases}
             \frac{a_n^2\mu_n^2}{2(1-p)(\gamma+1)-1}n, & \text{if } 0\leq p<p_c,\\
            v_n^2, & \text{if } p=p_c,
        \end{cases}.
    \end{align}
\end{proposition}
\begin{proof}
    Using the expression for the conditional expectation and the martingale difference $\Delta M_n$ from~\eqref{Conditional Expectation} and~\eqref{martingale diff M_n}, along with Lemma~\ref{Variance of the increments}, we obtain
    \begin{equation}{\label{EM_n^2}}
        \mathbb{E}(M_n^2)=\sum_{k=1}^{n}a_k^2\mu_k^2-p^2\sum_{k=1}^{n-1}\left(\frac{a_{k+1}\mu_{k+1}}{a_k\nu_{k}}\right)^2\mathbb{E}(M_{k}^2) \le v_n^2.
    \end{equation}
    
The result follows immediately when $\{v_n\}$ is bounded.

Next, we consider unbounded $\{v_n\}$. For the second term of~\eqref{EM_n^2}, observe that,
    \[
        p^2 \sum_{k=1}^{n-1}\left(\frac{a_{k+1}\mu_{k+1}}{a_k\nu_{k}}\right)^2\mathbb{E}M_{k}^2 \le \sum _{k=1}^{n-1}\left(\frac{\mu_{k+1}}{\nu_k}\right)^2 v_k^2 = \oh \l(v_n^2\r)
    \]
    using Kronecker's lemma, since the sequence $\{{\mu_{n+1}^2}/{\nu_n^2}\} \sim \{(\gamma+1)^2/n^2\}$ is summable and $v_n^2\uparrow\infty$. This gives 
  $\mathbb{E}\left(M_n^2\right) \sim v_n^2$ and~\eqref{eq: Mn asymp} for $0\le p<p_c$ follows from Karamata's theorem.
\end{proof}

We are now ready to proceed with the proof of Theorem~\ref{thm: bdd quad var}.
\begin{proof}[Proof of Theorem~\ref{thm: bdd quad var}]
    \begin{enumerate}[label=(\alph*)]
        \item The result follows immediately from Proposition~\ref{Asymptotics of martingale M_n}.        
        \item We now show that $M_\infty$ is platykurtic, when $\xi_1$ is symmetric Rademacher random variable. For Rademacher innovations, we have $\l|X_n-\bb{E}(X_n|\mathcal{F}_{n-1})\r|\leq 2$ almost surely. Writing $M_n=M_{n-1} + \Delta M_n$ and expanding $(M_{n-1} + \Delta M_{n})^4$, we have using~\eqref{eq: M mg transf}  
        \begin{align*}
            \bb EM_n^4 &= \bb EM_{n-1}^4 + a_n^4\mu_n^4\bb E(\Delta L_n)^4 + 6a_n^2\mu_n^2 \bb E (M_{n-1}^2 (\Delta L_n)^2) + 4a_n^3\mu_n^3 \bb E (M_{n-1} (\Delta L_n)^3)\\
            &\le \bb E M_{n-1}^4 + 16 a_n^4\mu_n^4 + 24 a_n^2\mu_n^2\sup_{m\ge 1}\bb E M_m^2 + 32 a_n^3\mu_n^3 \sup_{m\ge 1}\bb E|M_m|
            \intertext{which then gives}
            \bb EM_n^4 &\le 16\sum_{k=1}^n a_k^4 \mu_k^4 +24\sup_{m\ge 1}\bb EM_m^2 \sum_{k=2}^na_k^2\mu_k^2 + 32\sup_{m\ge 1}\bb E|M_m| \sum_{k=2}^n a_k^3\mu_k^3
        \end{align*}
For bounded $\{v_n\}$, all the sums on the right side converge. As $M_n$ is $L^2$ bounded, the martingale $\{M_n\}$ then becomes $L^4$-bounded and converges to $M_\infty$ in $L^4$.

Note that $\{M_n\}_{n\geq 1}$ being a mean zero martingale sequence, $\bb{E}M_\infty=0$. We study the kurtosis of $M_\infty$, given by $\kappa:={\bb{E}M_\infty^4}/{(\bb{E}M_\infty^2)^2}$. For $M_\infty$ to be platykurtic, we show 
\begin{equation} \label{eq: lepto}
3 (\bb{E}M_\infty^2)^2 - \bb{E}M_\infty^4 >0.
\end{equation} 

Towards proving~\eqref{eq: lepto}, for $x>0$ and $n\ge1$, define 
    $a_n(x)=\prod_{i=1}^{n-1}\l(1+{x\mu_{i+1}}/{\nu_i}\r)^{-1}$.
Note that $a_n(p)\equiv a_n$ and, for all $k\in\mathbb N$, $a_n(kp) \sim a_n^k$ as $n\to\infty$.

Due to Lemma~\ref{Variance of the increments}, $X_n$ also has symmetric Rademacher distribution, giving $X_n^2=1$ and $\mathbb{E} X_n =0$. Then, the sequence $\{Y_n\}_{n\geq 1}$ defined through~\eqref{Conditional Expectation} has zero mean. Hence, for $n\ge 1$, we have
$\bb{E}Y_n^3 =\l(1+{3p\mu_n}/{\nu_{n-1}}\r)\bb{E}Y_{n-1}^3$,
giving $\bb{E}Y_n^3=0$. Similarly, the second and fourth moments of $Y_n$ satisfy:
    \begin{align*}    
        \bb{E}Y_n^2&=\l(1+\frac{2p\mu_n}{\nu_{n-1}}\r)\bb{E}Y_{n-1}^2+\mu_n^2,
        \intertext{and}
        \bb{E}Y_n^4 &=\l(1+\frac{4p\mu_n}{\nu_{n-1}}\r)\bb{E}Y_{n-1}^4+6\mu_n^2\l(1+\frac{2p\mu_n}{3\nu_{n-1}}\r)\bb{E}Y_{n-1}^2+\mu_n^4.
    \end{align*}
    Therefore, from the above quantities, we have
    \[
    b_n=\l(1+\frac{4p\mu_n}{\nu_{n-1}}\r)b_{n-1}+\frac{12p^2\mu_n^2}{\nu_{n-1}^2}\l(\bb{E}Y_{n-1}^2\r)^2+\frac{8p\mu_n^3}{\nu_{n-1}}\bb{E}Y_{n-1}^2+2\mu_n^4,
    \]
    where $b_n=3(\bb{E}Y_n^2)^2-\bb{E}Y_n^4$. Solving the recursion, we obtain
    \[   
    a_n(4p)b_n = 2\sum_{k=1}^na_k(4p)\mu_k^4 + 8p\sum_{k=2}^n\frac{a_k(4p)\mu_k^3}{\nu_{k-1}}\bb{E}Y_{k-1}^2+ 12p^2\sum_{k=2}^n\frac{a_k(4p)\mu_k^2}{\nu_{k-1}^2}\l(\bb{E}Y_{k-1}^2\r)^2.
    \]
As $a_n(4p)\sim a_n^4$, taking limits, we get
    \begin{multline*}
    3(\bb{E}M_\infty^2)^2-\bb{E}M_\infty^4 = 2\sum_{n=1}^\infty a_n(4p)\mu_n^4\\ +8p\sum_{n=2}^\infty\frac{a_n(4p)\mu_n^3}{\nu_{n-1}}\bb{E}Y_{n-1}^2+ 12p^2\sum_{n=2}^\infty\frac{a_n(4p)\mu_n^2}{\nu_{n-1}^2}  \l(\bb{E}Y_{n-1}^2\r)^2,
    \end{multline*}
    which gives~\eqref{eq: lepto}.
    \item We first prove that the result holds marginally, using~\eqref{conn L S}. By Karamata's theorem, $\sum_{n=1}^\infty 1/(a_n \nu_n) = \infty$ holds for $p>\hp$, in particular for bounded $\{v_n\}$. 
    Define 
    \begin{equation}
        \label{eq: def eta}
        \eta_n = \sum_{k=1}^{n-1} \frac1{a_n \nu_n} \sim \frac{\gamma+1}{p(\gamma+1)-\gamma} \cdot \frac{1}{a_n \mu_n},
    \end{equation}
    a regularly varying sequence of index $p(\gamma+1) - \gamma$, for $p>\widehat{p}$. Next consider the martingale transform $\{\sum_{k=1}^n \eta_k^{-1} \Delta L_k\}$ of $\{L_n\}$. Then,
    $\bb E \l( \sum_{k=1}^n \eta_k^{-1} \Delta L_k \r)^2 = \sum_{k=1}^n \eta_k^{-2} \bb E (\Delta L_k)^2 \le \sum_{k=1}^n \eta_k^{-2} \bb E (X_k)^2 \le \sum_{k=1}^n \eta_k^{-2} < \infty$,
    using~\eqref{eq: def eta} and boundedness of $\{v_n\}$. Thus, the martingale $\{\sum_{k=1}^n \eta_k^{-1} \Delta L_k\}$ is $L^2$-bounded and converges almost surely and in $L^2$. Since $\eta_n\uparrow\infty$, by Kronecker's lemma and~\eqref{eq: def eta}, we have 
    \begin{equation}
        \label{eq: scaled L}
        a_n \mu_n L_n \to 0 \quad \text{almost surely and in $L^2$.}
    \end{equation}
    Furthermore, an application of Toeplitz lemma to the martingale $\{M_n\}$ yields 
\begin{equation} \label{eq: scaled M}
p a_n \mu_n \sum_{k=1}^{n-1} \frac1{a_k \nu_k} M_k \to \frac{p(\gamma+1)}{p(\gamma+1) - \gamma} M_\infty \quad \text{almost surely and in $L^2$.}
\end{equation}
Combining~\eqref{eq: scaled L} and~\eqref{eq: scaled M}, we get the required marginal convergence for $t=1$ from~\eqref{conn L S}. A normalization argument gives the marginal convergence for $t>0$.

Next we establish the almost sure and in $L^2$ process convergence of SRRW-RVM. By Proposition~2.3 of~\cite{Janson1994}, it is enough to show for any $T>0$,
\begin{equation} \label{eq: target as}
\sup_{0\le t\le T} \left| a_n \mu_n S_{\lfloor nt \rfloor} - \frac{p(\gamma+1)}{p(\gamma+1)-\gamma} t^{p(\gamma+1)-\gamma} M_\infty \right| \to 0 \quad \text{almost surely and in $L^2$.}
\end{equation}

For almost sure convergence of~\eqref{eq: target as}, fix $T>0$. For any $0<t_0<T$, we have
\begin{align*}
    &\sup_{0\le t\le T} \left| a_n \mu_n S_{\lfloor nt \rfloor} - \frac{p(\gamma+1)}{p(\gamma+1)-\gamma} t^{p(\gamma+1)-\gamma} M_\infty \right| \\ 
    &\qquad\le
    \sup_{0\le t\le T} \left| \frac{a_n \mu_n}{a_{\lfloor nt \rfloor} \mu_{\lfloor nt \rfloor}} - t^{p(\gamma+1)-\gamma} \right| \sup_{l\ge 1} \left| a_l \mu_l S_l \right|\\
    &\qquad \qquad + t_0^{p(\gamma+1)-\gamma} \left( \sup_{l\ge 1} \left| a_l \mu_l S_l \right| + \frac{p(\gamma+1)}{p(\gamma+1)-\gamma} \left|M_\infty\right| \right) \\
    &\qquad \qquad + T^{p(\gamma+1)-\gamma} \sup_{l\ge \lfloor nt_0 \rfloor} \left| a_l \mu_l S_l - \frac{p(\gamma+1)}{p(\gamma+1)-\gamma} M_\infty \right|.
\end{align*}
Since $\{{a_n\mu_n}\}$ is regularly varying of index ${\gamma-p(\gamma+1)}$, which is negative for $p\ge p_c$, using Proposition~0.5 of~\cite{Resnick1987} and the almost sure marginal convergence at $t=1$, we obtain the negligibility of the first and third terms on the RHS of the above inequality, as $n\rightarrow\infty$. Then, letting $t_0\downarrow0$ gives the negligibility of the second term, proving~\eqref{eq: target as} almost surely.

For $L^2$ process convergence, we use the decomposition~\eqref{conn L S} to obtain
\begin{multline*}
    \sup_{0\le t\le T} \left| a_n \mu_n S_{\lfloor nt \rfloor} - \frac{p(\gamma+1) t^{p(\gamma+1)-\gamma}}{p(\gamma+1)-\gamma}  M_\infty \right| \\
    \le a_n \mu_n \sup_{0\le t\le T} \left| L_{\lfloor nt \rfloor} \right| + a_n \mu_n \sum_{k=1}^{\lfloor nT \rfloor -1} \frac1{a_k \nu_k} \left| M_k - M_\infty \right|\\
    + |M_\infty| \sup_{0\le t\le T} \left| a_n \mu_n \sum_{k=1}^{\lfloor nt \rfloor -1} \frac{1}{a_k \nu_k} - \frac{\gamma+1}{p(\gamma+1)-\gamma} t^{p(\gamma+1) - \gamma} \right|. 
\end{multline*}

Observe $\bb E L_n^2 = \sum_{k=1}^n \bb E (\Delta L_k)^2 \le \sum_{k=1}^n \bb E X_k^2 = n$. Then, using Doob's $L^2$-inequality and Karamata's theorem, the first term on the right side above is $L^2$-negligible:
\begin{equation*}
a_n^2 \mu_n^2 \mathbb{E} \left( \sup_{0\le t\le T} L_{\lfloor nt \rfloor}^2 \right) \le 4 a_n^2 \mu_n^2 \mathbb{E} L_{\lfloor nT \rfloor}^2 
\le 4 {nT} a_n^2 \mu_n^2= \Oh \l(\sum_{k=n}^\infty a_k^2 \mu_k^2 \r) \to 0.
\end{equation*}

The second term on the right side above is $L^2$-negligible by Toeplitz lemma and part~\eqref{Mn limit} of this theorem, while the third term is $L^2$-negligible using Karamata's theorem and {Proposition 0.5 of~\cite{Resnick1987}}.
    \end{enumerate}
\end{proof}

\section{Laws of Large Numbers when Innovations have Finite Variance}{\label{sec: SLLN}}
In this section we prove the laws of large numbers. However, for simplicity and ease of understanding, we provide the proofs only for innovation sequences with finite variance, which we assume to be $1$ without loss of generality. The complete proofs, using truncation arguments, are deferred till Section~\ref{sec: SRRW}.

We first provide the marginal laws of large numbers for the martingale sequences $\{M_n\}$ and $\{L_n\}$, in the following lemma. 
\begin{lemma}{\label{lemma: trunc mart conv ERW}}
Let $\{S_n\}_{n\ge 0}$ be an SRRW-RVM with zero mean innovation sequence $\{\xi_n\}_{n\ge 1}$. Then we have, for all $p\in [0,1)$,
${M_n}/({a_n \nu_n}) \to 0$ and ${L_n}/{n} \to 0$ almost surely and in $L^1$. Further, for finite variance innovations, the convergence is in $L^2$.
\end{lemma}
We first prove the result under the assumption that the innovations have finite variance. 
\begin{proof}[Proof of Lemma~\ref{lemma: trunc mart conv ERW} for unit variance innovations]
    Assume, without loss of generality, the variance of $\xi_n$ to be $1$.
    Lemma~\ref{Variance of the increments} gives $\mathbb E ( \Delta L_n )^2 \le \mathbb E X_n^2 = \mathbb E \xi_1^2 = 1$ for Rademacher innovations. Then, using~\eqref{martingale diff M_n} and~\eqref{eq: M mg transf}, we obtain
    \[\sum_{n=1}^\infty \mathbb E ( \Delta L_n )^2 / {n^2} \le \sum_{n=1}^\infty \mathbb E X_n^2 / {n^2} < \infty \quad \text{and} \quad \sum_{n=1}^\infty \mathbb E ( \Delta M_n )^2 / (a_n^2 \nu_n^2) \le \sum_{n=1}^\infty \mu_n^2/\nu_n^2 < \infty.\]
    So, 
    $\left\{\sum_{k=1}^n \Delta M_k/(a_k \nu_k)\right\}$ and $\left\{\sum_{k=1}^n \Delta L_k/k\right\}$ 
    are $L^2$-bounded martingales, and converge almost surely and in $L^2$. 
    Further note that $(a_{n+1} \nu_{n+1}) / (a_n \nu_n) = \nu_{n+1} / (\nu_n + p \mu_{n+1}) \ge 1$ making $\{a_n \nu_n\}$ an increasing sequence, which, being regularly varying of index $(1-p)(\gamma+1)$, increases to $\infty$. Thus, Kronecker's lemma gives the results. 
\end{proof}

Next, we obtain the functional law of large numbers for unit variance innovation.
\begin{proof}[Proof of Theorem~\ref{thm: SLLN} for unit variance innovations]
 Since we consider innovations with unit variance, we prove almost sure and in $L^2$ convergence. For $p\in[0,1)$, using~\eqref{conn L S}, the result follows immediately for $t=1$ from Lemma~\ref{lemma: trunc mart conv ERW}, and extension to any $t>0$ follows by spatial normalization. For process convergence, again, due to Proposition~2.3 of~\cite{Janson1994}, it is enough to show for any $T>0$, 
$S^*_{\lfloor nT \rfloor} /n \to 0$ 
in appropriate mode of convergence, where 
$S^*_K := \sup_{0\le k\le K} |S_k|$.

We first consider almost sure convergence. As $S_{\floor{nt}}=0$ for $0\le t<1/n$, we have for any $T>0$ and $t_0>0$,
    \[\frac{1}{n}\sup_{0\leq t\leq T}|S_{\lfloor nt\rfloor}|
        \leq \sup_{\frac1n\leq t\leq t_0} \left(t \frac{\l|S_{\lfloor nt\rfloor}\r|}{\lfloor nt\rfloor} \right)+\sup_{t_0\leq t\leq T}\left(t\frac{\l|S_{\lfloor nt\rfloor}\r|}{\lfloor nt\rfloor}\right) 
        \leq t_0 \sup_{l\ge1} \frac{|S_l|}{l} +T\sup_{l\geq \lfloor nt_0\rfloor}\frac{|S_l|}{l}.\] 
    Using the marginal almost sure convergence for $t=1$ and $t_0>0$ being arbitrarily small, the right side above is almost surely negligible.
    
Next, we prove $L^2$ convergence for symmetric Rademacher innovations. Using~\eqref{conn L S} and careful applications of the Cauchy-Schwarz inequality in the first step and Doob's $L^2$ inequality in the next, we get
\begin{align*}
    \frac1n \sqrt{ \mathbb{E} \left( S^*_{\lfloor nT \rfloor} \right)^2 }
    &\le \frac1n \sqrt{ \mathbb{E} \left( \sup_{1\le l\le \lfloor nT \rfloor}  L_l^2 \right) }
    + \frac1n \sum_{k=1}^{\lfloor nT \rfloor-1} \sqrt{ \mathbb{E} \left( \frac{M_k}{a_k \nu_k} \right)^2 }\\
    &\le 2 \sqrt{\frac1{n^2} \mathbb{E} L_{\lfloor nT \rfloor}^2}
    + \frac1n \sum_{k=1}^{\lfloor nT \rfloor-1} \sqrt{ \mathbb{E} \left( \frac{M_k}{a_k \nu_k} \right)^2 }.
\end{align*}
Then $S^*_{\lfloor nT \rfloor} /n \to 0$ in $L^2$ follows using Lemma~\ref{lemma: trunc mart conv ERW}.

The case $p=1$ follows easily as $S_{\floor{nt}} = \floor{nt} \xi_1$.
\end{proof}
\begin{remark} \label{rem: as ERW}
    The proof of Theorem~\ref{thm: SLLN} uses the unit variance assumption on the innovation sequence only through Lemma~\ref{lemma: trunc mart conv ERW}.
\end{remark}
\section{Process Weak Convergence for SRRW-R-RVM} \label{ERW}
In this section, we study the fluctuations of the SRRW-RVM process in the subcritical regime $p\in [0,p_c)$ as well as in the critical regime $p=p_c$ with unbounded $\{v_n\}$. We first introduce two additional martingales that will be useful in the analysis. For simplicity, we shall prove some of the results in this section only for SRRW-R-RVM, with the corresponding proofs for general innovations being deferred to Section~\ref{sec: SRRW}.
\subsection{Two martingales for the normalized SRRW-RVM process}{\label{Functional CLT regime}}
 In this subsection, we introduce two martingales that will be important in studying the SRRW-RVM process for unbounded  $\{v_n\}$. We distinguish four subcases: namely $p\in[0, \widehat p)$, $p = \widehat p$, $p\in (\widehat p, p_c)$ and finally $p=p_c$. In the first three subcases, $\{v_n\}$ is unbounded, while, for $p=p_c$, we consider only those $\{\mu_n\}$ such that $\{v_n\}$ is unbounded. As noted in Corollary~\ref{cor: phase}, we can use the normalization $\sigma_n$ for all four subcases. However, we use the equivalent diffusive $\sqrt{n}$-normalization in the first three subcases as it simplifies the calculations. In the fourth subcase, the normalization is not explicit, but for interesting illustrative examples in Section~\ref{sec: examples}, we simplify $\sigma_n$ and provide more explicit rates.

As noted earlier, the decomposition~\eqref{conn L S} plays an important role in the analysis, but it needs to be further simplified before it is applied. Using~\eqref{martingale diff M_n}~--~\eqref{eq: M mg transf}, we have
\begin{align}
    S_n &= L_n + p \sum_{k=1}^{n-1} \frac1{a_k \nu_k} M_k = L_n + p \sum_{k=1}^{n-1} \frac1{a_k \nu_k} \sum_{l=1}^k \Delta M_l = L_n + p \sum_{l=1}^{n-1} \sum_{k=l}^{n-1} \frac1{a_k \nu_k} \Delta M_l\nonumber\\
    &= \begin{cases}
        L_n + p \sum_{l=1}^{n} (\eta_n - \eta_l) \Delta M_l, &\text{when $\sum_n \frac1{a_n \nu_n} = \infty$,}\\
        L_n + p \sum_{l=1}^{n} (\overline\eta_l - \overline\eta_n) \Delta M_l, &\text{when $\sum_n \frac1{a_n \nu_n} < \infty$,} 
        \end{cases} \nonumber\\
    &= \begin{cases}
        \sum_{l=1}^{n} (1 - p a_l \mu_l \eta_l) \Delta L_l + p \eta_n M_n, &\text{when $\sum_n \frac1{a_n \nu_n} = \infty$,}\\
        \sum_{l=1}^{n} (1 + p a_l \mu_l \overline \eta_l) \Delta L_l - p \overline \eta_n M_n, &\text{when $\sum_n \frac1{a_n \nu_n} < \infty$,}
       \end{cases} \label{eq: L eta}\\
    &= \begin{cases}
        N_n + p \eta_n M_n, &\text{when $\sum_n \frac1{a_n \nu_n} = \infty$,}\\
        \overline N_n - p \overline \eta_n M_n, &\text{when $\sum_n \frac1{a_n \nu_n} < \infty$,}
    \end{cases} \label{eq: S N split}
\end{align}
where, we define,
\begin{align}
    \eta_n &= \sum_{l=1}^{n-1} \frac1{a_l \nu_l}, 
    &N_n &= \sum_{l=1}^{n} (1 - p a_l \mu_l \eta_l) \Delta L_l &\text{when } \sum_n \frac1{a_n \nu_n} &= \infty, \label{eq: N}\\
    \text{and} \quad \overline \eta_n &= \sum_{l=n}^{\infty} \frac1{a_l \nu_l}, 
    &\overline N_n &= \sum_{l=1}^{n} (1 + p a_l \mu_l \overline \eta_l) \Delta L_l &\text{when } \sum_n \frac1{a_n \nu_n} &< \infty, \nonumber
\end{align}
where the sequences $\{\eta_n\}$ and $\{\overline\eta_n\}$ are regularly varying of index $p(\gamma+1)-\gamma$. (Note that $\eta_n$ has already been defined in~\eqref{eq: def eta}.) Thus, for the case $\sum_n 1/({a_n \nu_n}) = \infty$, we consider the joint weak convergence of $(N_n, \eta_n M_n) / {\sqrt{n}}$, while for the case $\sum_n 1 / ({a_n \nu_n}) < \infty$, we consider the joint weak convergence of $(\overline N_n, \overline \eta_n M_n) / {\sqrt{n}}$. Then an appropriate linear transformation would give the weak convergence of $S_n/\sqrt{n}$. For the case $p=\widehat p$, the proof is similar, but more subtle -- see Remark~\ref{rem: triang}. For $p\in(\widehat p, p_c)$, we have $\sum_n 1/({a_n \nu_n}) = \infty$ and was considered by~\cite{bertenghi2024universal}. For $p\in[0, \widehat p)$, we have $\sum_n 1/({a_n \nu_n}) < \infty$, but the process convergence requires more careful analysis of tightness near $0$.  

\begin{remark} \label{rem: triang}
For the case $p=\widehat p$, in~\eqref{eq: L eta}, the term $-p \sum_{l=1}^n a_l \mu_l \eta_l \Delta L_l$ (respectively, $p \sum_{l=1}^n a_l \mu_l \overline \eta_l \Delta L_l$) dominates due to Karamata's theorem, but cancels with the term $p \eta_n M_n = p \eta_n \sum_{l=1}^n a_l \mu_l \Delta L_l$ (respectively, $- p \overline{\eta}_n M_n = - p \overline{\eta}_n \sum_{l=1}^n a_l \mu_l \Delta L_l$), leaving out terms which are diffusive in growth. This motivates a more careful analysis using triangular normalizing matrices. See Remark~\ref{rem: p hat}.
\end{remark}
\subsection{Subcritical regime for SRRW-R-RVM}{\label{subcrit reg}}  In the subcritical regime, both terms in~\eqref{eq: S N split} contribute. We obtain a functional martingale central limit theorem for an appropriate triangular array of bivariate martingale difference sequence, in each of the cases $p\in[0, \widehat{p})$, $p=\widehat{p}$ and $p\in(\widehat{p}, p_c)$, using the Corollary to Theorem~2 of~\cite{Touati1992}. While the martingale difference arrays are defined differently across the cases, their quadratic variation processes are computed and the Lindeberg conditions checked, in a unified way.

For $p\in[0,\widehat{p})$, we use the diagonal normalizing matrix
$\overline{\boldsymbol{V}}_n
        = n^{-1/2} \diag(1, {\overline{\eta}_n})$ for $n\ge 1$,
and the corresponding triangular array of bivariate martingale difference sequence is, for $n\ge 1$ and $1\le k\le n$,
\begin{equation*} 
\Delta\overline{\boldsymbol{T}}'_{n,k}:=\overline{\boldsymbol{V}}_n(\Delta \overline{N}_k,\Delta M_k)' = \frac1{\sqrt{n}} \left( (1+p a_k \mu_k \oeta_k), a_k \mu_k \oeta_n \right)' \Delta L_k.
\end{equation*}
Correspondingly, for $p\in(\widehat p, p_c)$, we use the diagonal normalizing matrix
$\boldsymbol{V}_n
        = n^{-1/2} \diag(1, {\eta_n})$ for $n\ge 1$, as in~\cite{bertenghi2024universal},
and the corresponding triangular array of bivariate martingale difference sequence is, for $n\ge 1$ and $1\le k\le n$,
\begin{equation*} 
\Delta\boldsymbol{T}'_{n,k}:=\boldsymbol{V}_n(\Delta N_k, \Delta M_k)' = \frac1{\sqrt{n}} \left( (1-p a_k \mu_k \eta_k), a_k \mu_k \eta_n \right)' \Delta L_k.
\end{equation*}
 
For $p=\widehat{p}$, we need an upper triangular normalizing matrix, as suggested in Remark~\ref{rem: triang}, given by,
\begin{align*}
&&\widetilde{\boldsymbol{V}}_n &= \frac{1}{\sqrt{n}} \begin{pmatrix}
    1 & {p\eta_n}\\
    0 & \frac{1}{a_n\mu_n}
\end{pmatrix}, & \text{when $\sum_n \frac1{a_n \nu_n} = \infty$,} \\ \text{and} 
&&\widetilde{\boldsymbol{V}}_n &= \frac{1}{\sqrt{n}} \begin{pmatrix}
    1 & {-p\overline\eta_n}\\
    0 & \frac{1}{a_n\mu_n}
\end{pmatrix}, & \text{when $\sum_n \frac1{a_n \nu_n} < \infty$,}
\end{align*}
for $n\ge 1$, and the corresponding triangular array of bivariate martingale difference sequence is, for $n\ge 1$ and $1\le k\le n$,
\begin{equation*} 
\Delta\widetilde{\boldsymbol{T}}'_{n,k} :=
\begin{cases}
\widetilde{\boldsymbol{V}}_n(\Delta N_k, \Delta M_k)' = \frac1{\sqrt{n}} \left( \l(1+p a_k \mu_k(\eta_n-\eta_k)\r), \frac{a_k \mu_k}{a_n\mu_n}  \right)' \Delta L_k, &\!\!\!\text{if } \sum_n \frac1{a_n \nu_n} = \infty,\\
\widetilde{\boldsymbol{V}}_n(\Delta \overline{N}_k, \Delta M_k)' = \frac1{\sqrt{n}} \left( \l(1+p a_k \mu_k(\overline\eta_k-\overline\eta_n)\r), \frac{a_k \mu_k}{a_n\mu_n}  \right)' \Delta L_k, &\!\!\!\text{if } \sum_n \frac1{a_n \nu_n} < \infty.
\end{cases}
\end{equation*}

\begin{remark} \label{rem: p hat}
    While it seems that $\Delta\widetilde{\boldsymbol{T}}'_{n,k}$ have two distinct forms depending on the summability of $(a_n \nu_n)^{-1}$, actually they lead to one single expression, as both $(\eta_n - \eta_k)$ when $\sum_n 1/(a_n \nu_n) = \infty$ and $(\overline\eta_n - \overline\eta_k)$ when $\sum_n 1/(a_n \nu_n) < \infty$, lead to same quantity $\sum_{l=k}^{n-1} 1/(a_l \nu_l)$. We use this common quantity.
\end{remark}

\begin{remark}
While the second diagonal entries of $\boldsymbol{V}_n$ and $\overline{\boldsymbol{V}}_n$ seem to differ from that of $\widetilde{\boldsymbol{V}}_n$, using Karamata's theorem, the second diagonal entries of $\boldsymbol{V}_n$ and $\overline{\boldsymbol{V}}_n$ have the same order as that of $\widetilde{\boldsymbol{V}}_n$. The alternate form helps in calculation.
\end{remark} 

Next, we collect results on the limiting quadratic variation, a Lindeberg condition and the subsequent martingale central limit theorem, corresponding to each of the bivariate martingale processes defined above. We state the results for SRRW-RVM, but prove them only for SRRW-R-RVM in this section, leaving the general proof for Section~\ref{sec: SRRW}. We indicate explicitly the proofs that do not use this additional assumption. We start with the limiting quadratic variation.
\begin{lemma}{\label{lem: QV}}
    Let $\{S_n\}_{n\geq 0}$ be an SRRW-RVM, with the zero mean innovation sequence satisfying $\bb E\xi_1^2 = 1$. Then we have:
    
    For $p\in [0, \widehat p)$, the quadratic variation process of $ \{\omT_{n,k}\}_{k\ge 1}$, for each $t\ge0$, satisfies
        \[
        \langle \omT \rangle_{n, \ntfloor}
        \to_{L^1} \boldsymbol{\overline{W}}(t) := \frac{1}{( p(\gamma+1)-\gamma )^2}\begin{pmatrix}
        \gamma^2t & \frac{\gamma t^{(1-p)(\gamma+1)}}{1-p}\\
        \frac{\gamma t^{(1-p)(\gamma+1)}}{1-p} & \frac{(\gamma+1)^2t^{2(1-p)(\gamma+1)-1}}{2(1-p)(\gamma+1)-1}
    \end{pmatrix},
        \]
        with $\overline{\boldsymbol{W}}(t)$ positive definite for $t>0$, $\overline{\boldsymbol{W}}(0)=\boldsymbol{0}$. 
        
        For $p\in(\widehat p, p_c)$, the quadratic variation process of $ \{\mT_{n,k}\}_{k\ge 1}$, for each $t\ge 0$, satisfies
        \[
        \langle \mT \rangle_{n, \ntfloor}
        \to_{L^1}\boldsymbol{W}(t):=\frac{1}{(p(\gamma+1)-\gamma)^2}
        \begin{pmatrix}
        \gamma^2t & -\frac{\gamma t^{(1-p)(\gamma+1)}}{1-p}\\
        -\frac{\gamma t^{(1-p)(\gamma+1)}}{1-p} & \frac{(\gamma+1)^2t^{2(1-p)(\gamma+1)-1}}{2(1-p)(\gamma+1)-1}
    \end{pmatrix},
        \]
        with $\boldsymbol{W}(t)$ positive definite for $t>0$, $\boldsymbol{W}(0)=\boldsymbol{0}$.
        
        For $p=\widehat p$, the quadratic variation process of $ \{\tmT_{n,k}\}_{k\ge 1}$, for each $t\ge 0$, satisfies
        \[
        \langle \tmT \rangle_{n, \ntfloor}
        \to_{L^1} \widetilde{\boldsymbol{W}}(t) := t\begin{pmatrix}
        \gamma^2+(\gamma+1-\gamma\log t)^2 & \gamma+1-\gamma\log t\\
        \gamma+1-\gamma\log t & 1
    \end{pmatrix},\] 
    with $\boldsymbol{\widetilde{W}}(t)$ is positive definite matrix for $t>0$, $\boldsymbol{\widetilde{W}}(0)=\boldsymbol{0}$.
\end{lemma}

For the Lindeberg conditions, the following bounds on the norms of each of the martingale difference arrays are useful.
\begin{lemma} \label{lem: Linden bd ERW}
Let $\{S_n\}_{n\geq 0}$ be an SRRW-RVM with the zero mean innovation sequence $\{\xi_n\}$ satisfying $\bb E\xi_1^2=1$. Then, for every $t>0$, there is a positive constant $C_t$ such that for all $n\ge 1$ and $1\le k\le \floor{nt}$, the following hold:

When $p\in[0,\hp)$, we have $\| \Delta \omT_{n,k} \| \le C_t |\Delta L_k| / \sqrt{n}$.

When $p\in(\hp, p_c)$, we have $\| \Delta \mT_{n,k} \| \le C_t |\Delta L_k| / n^\rho$, for some $\rho>0$.

When $p=\hp$, we have $\| \Delta \tmT_{n,k} \| \le C_t |\Delta L_k| / n^{1/4}$.
\end{lemma}

The following lemma gives the Lindeberg conditions for $p\in [0,p_c)$.
\begin{lemma} \label{lem: Lindeberg ERW}
Let $\{S_n\}_{n\geq 0}$ be an SRRW-RVM with the zero mean innovation sequence $\{\xi_n\}$ satisfying $\bb E\xi_1^2=1$. Then, for any $\epsilon>0$ and for every $t>0$, we have:

For $p\in[0,\hp)$, we have
$\sum_{k=1}^{\lfloor nt\rfloor}\mathbb{E}\left(\|\Delta\overline{\boldsymbol{T}}_{n,k}\|^2\mathbbm{1}_{\{\|\Delta\overline{\boldsymbol{T}}_{n,k}\|>\epsilon\}}\middle|\mathcal{F}_{k-1}\right) \to_{L^1} 0$.

For $p\in(\hp, p_c)$, we have
$\sum_{k=1}^{\lfloor nt\rfloor}\mathbb{E}\left(\|\Delta\boldsymbol{T}_{n,k}\|^2\mathbbm{1}_{\{\|\Delta\boldsymbol{T}_{n,k}\|>\epsilon\}}\middle|\mathcal{F}_{k-1}\right) \to_{L^1} 0$.

For $p=\hp$, we have
$\sum_{k=1}^{\lfloor nt\rfloor}\mathbb{E}\left(\| \Delta \widetilde{\boldsymbol{T}}_{n,k}\|^2\mathbbm{1}_{\{\|\Delta\widetilde{\boldsymbol{T}}_{n,k}\|>\epsilon\}}\middle|\mathcal{F}_{k-1}\right) \to_{L^1} 0$.
\end{lemma}

Using Lemmas~\ref{lem: QV} and~\ref{lem: Lindeberg ERW}, the following functional central limit theorems for each of the bivariate martingale processes are then immediate from the Corollary to Theorem~2 of~\cite{Touati1992}.
\begin{proposition} {\label{neq gamma/gamma+1 joint invariance}}
Let $\{S_n\}_{n\geq 0}$ be an SRRW-RVM with zero mean innovation sequence $\{\xi_n\}$ satisfying $\bb E\xi_1^2=1$. Then:
    \begin{align*}
        \l(\omT_{n,\lfloor nt\rfloor}:t\geq 0\r)&\to_{\rm{w}} \l(\mathcal{\overline{W}}(t): t\geq 0\r) \quad \text{in $(D([0,\infty)),\mathcal{D})$,} \quad \text{for $p\in[0,\hp)$},\\
        \l(\tmT_{n,\lfloor nt\rfloor}: t\geq 0\r)&\to_{\rm{w}} \l(\widetilde{\mathcal{W}}(t): t\geq 0\r) \quad \text{in $(D([0,\infty)),\mathcal{D})$,} \quad \text{for $p=\hp$},\\
     \l(\mT_{n,\lfloor nt\rfloor}: t\geq 0\r) &\to_{\rm{w}} \l(\mathcal{W}(t): t\geq 0\r) \quad \text{in $(D([0,\infty)),\mathcal{D})$,} \quad \text{for $p\in(\hp,p_c)$},
    \end{align*}
    where $(\mathcal{\overline{W}}(t) :t\geq 0)$, $(\widetilde{\mathcal{W}}(t) :t\geq 0)$ and $(\mathcal{W}(t) :t\geq 0)$ are continuous $\mathbb{R}^2$ valued centered Gaussian processes, with the covariance kernels given by
    \begin{align*}
        \mathbb{E}(\mathcal{\overline{W}}(s)\mathcal{\overline{W}}(t))&=\boldsymbol{\overline{W}}(s), \quad \text{for } 0\leq s\leq t, \quad \text{when $p\in[0,\hp)$},\\
    \mathbb{E}(\widetilde{\mathcal{W}}(s) \widetilde{\mathcal{W}}(t))&=\widetilde{\boldsymbol{W}}(s), \quad \text{for } 0\leq s\leq t, \quad \text{when $p=\hp$},\\
         \mathbb{E}(\mathcal{W}(s)\mathcal{W}(t))&=\boldsymbol{W}(s), \quad \text{for } 0\leq s\leq t, \quad \text{when $p\in(\hp, p_c)$}.
    \end{align*}
\end{proposition}

While the martingale central limit theorem above holds in $D([0, \infty))$ above, the continuous mapping theorem can only be applied on $D((0, \infty))$, to obtain the limit of the normalized walk as the multiplier is undefined at $t=0$ for $p\in[0, \hp]$ - see~\eqref{eq: weak super hat} and~\eqref{eq: weak hat}. The behavior around $0$ is taken care of by the following result on uniform equicontinuity at $0$.
\begin{lemma}
    \label{lem: equi unif cont}
    Let $\{S_n\}_{n\geq 0}$ be an SRRW-RVM with the zero mean innovation sequence $\{\xi_n\}_{n\ge 1}$ satisfying $\bb E\xi_1^2=1$. Then, for all $p\in[0,p_c)$, we have
\[\adjustlimits\lim_{t\rightarrow 0} \limsup_{n\rightarrow\infty} \frac1n \mathbb{E} \left( S^*_{\lfloor nt\rfloor}\right)^2 =0 \quad \text{and, hence for all  $\epsilon>0$,} \quad \adjustlimits\lim_{t\rightarrow 0} \limsup_{n\rightarrow\infty} \mathbb{P} \left( S^*_{\lfloor nt\rfloor} > \epsilon \sqrt{n} \right) = 0.\]
\end{lemma}

We now provide the proofs of the above results. Some of the results have been proved here only for SRRW-R-RVM, for the ease of understanding the main ideas. The proofs for general innovations have been deferred to Section~\ref{sec: SRRW}. All the proofs that apply only to SRRW-R-RVM have been clearly marked. We begin with the proof of Lemma~\ref{lem: QV}, done here only for SRRW-R-RVM.
\begin{proof}[Proof of Lemma~\ref{lem: QV} for SRRW-R-RVM]
For the case $p\in[0,\hp)$, as $X_k^2=1$ for SRRW-R-RVM, the quadratic variation process is
\begin{equation} \label{eq: qv omT}
\langle \omT \rangle_{n, \ntfloor} = \sum_{k=1}^{\ntfloor} \omQ_{n,k} \mathbb{E}(X_k^2|\mathcal{F}_{k-1}) - p^2 \sum_{k=2}^{\ntfloor} \omQ_{n,k} \left( \frac{Y_{k-1}}{\nu_{k-1}} \right)^2 = \sum_{k=1}^{\ntfloor} \omQ_{n,k} - p^2 \sum_{k=2}^{\ntfloor} \omQ_{n,k} \left( \frac{Y_{k-1}}{\nu_{k-1}} \right)^2
\end{equation}
where
\[
\omQ_{n,k} = \frac1n 
\begin{pmatrix}
    (1+p a_k \mu_k \oeta_k)^2 & a_k \mu_k (1+p a_k \mu_k \oeta_k) \oeta_n\\
    a_k \mu_k (1+p a_k \mu_k \oeta_k) \oeta_n & a_k^2 \mu_k^2 \oeta_n^2
\end{pmatrix}.
\]
Using Karamata's theorem, for each $t\ge 0$,
$\sum_{k=1}^\ntfloor \omQ_{n,k} \to \omW(t)$ and, for each fixed $k$, $\omQ_{n,k}$ converges to the zero matrix.
As $Y_n^2/\nu_n^2 \to 0$ in $L^1$ from Lemma~\ref{lemma: trunc mart conv ERW}, the last term of~\eqref{eq: qv omT} converges to the zero matrix by Toeplitz lemma.

Similarly, for the case $p\in (\hp, p_c)$, the quadratic variation process is given by
\begin{equation*} 
\langle \mT \rangle_{n, \ntfloor} = \sum_{k=1}^{\ntfloor} \mQ_{n,k} \mathbb{E}(X_k^2|\mathcal{F}_{k-1}) - p^2 \sum_{k=2}^{\ntfloor} \mQ_{n,k} \left( \frac{Y_{k-1}}{\nu_{k-1}} \right)^2 = \sum_{k=1}^{\ntfloor} \mQ_{n,k} - p^2 \sum_{k=2}^{\ntfloor} \mQ_{n,k} \left( \frac{Y_{k-1}}{\nu_{k-1}} \right)^2
\end{equation*}
where
\[
\mQ_{n,k} = \frac1n 
\begin{pmatrix}
    (1-p a_k \mu_k \eta_k)^2 & a_k \mu_k (1-p a_k \mu_k \eta_k) \eta_n\\
    a_k \mu_k (1-p a_k \mu_k \eta_k) \eta_n & a_k^2 \mu_k^2 \eta_n^2
\end{pmatrix}.
\]
The rest of the proof for the case $p\in (\hp, p_c)$, is analogous to that of the case $p\in [0,\hp)$ and is omitted.

Again, for the case $p=\hp$, the quadratic variation process is given by
\begin{equation} \label{eq: qv tmT}
\langle \tmT \rangle_{n, \ntfloor} = \sum_{k=1}^{\ntfloor} \tmQ_{n,k} \mathbb{E}(X_k^2|\mathcal{F}_{k-1}) - \hp^2 \sum_{k=2}^{\ntfloor} \tmQ_{n,k} \left( \frac{Y_{k-1}}{\nu_{k-1}} \right)^2 = \sum_{k=1}^{\ntfloor} \tmQ_{n,k} - \hp^2 \sum_{k=2}^{\ntfloor} \tmQ_{n,k} \left( \frac{Y_{k-1}}{\nu_{k-1}} \right)^2
\end{equation}
where
\[
\tmQ_{n,k} = \frac1n 
\begin{pmatrix}
        (1+\hp a_k\mu_k(\eta_n-\eta_k))^2 & \l(a_k\mu_k(1 + \hp a_k\mu_k(\eta_n-\eta_k))\r)/(a_n\mu_n)\\
        \l(a_k\mu_k(1 + \hp a_k\mu_k(\eta_n-\eta_k))\r)/(a_n\mu_n) & \l(a_k^2\mu_k^2\r)/\l(a_n^2\mu_n^2\r)
\end{pmatrix}.
\]
We consider the limits of each element of $\sum_{k=1}^\ntfloor \tmQ_{n,k}$ separately. We note that
\begin{align*}
    \frac1n\sum_{k=1}^\ntfloor a_k^2 \mu_k^2 (\eta_n - \eta_k) &= \frac1n \sum_{k=1}^\ntfloor a_k^2 \mu_k^2 (\eta_\ntfloor - \eta_k) + \frac{\eta_n - \eta_\ntfloor}{n} v_\ntfloor^2,\\
    \frac1n\sum_{k=1}^\ntfloor a_k \mu_k (\eta_n - \eta_k) &= \frac1n \sum_{k=1}^\ntfloor a_k \mu_k (\eta_\ntfloor - \eta_k) + \frac{\eta_n - \eta_\ntfloor}{n} \sum_{k=1}^\ntfloor a_k \mu_k,
\end{align*}
and 
\begin{multline*}
    \frac1n\sum_{k=1}^\ntfloor a_k^2 \mu_k^2 (\eta_n - \eta_k)^2 = \frac1n \sum_{k=1}^\ntfloor a_k^2 \mu_k^2 (\eta_\ntfloor - \eta_k)^2 + 2 \frac{\eta_n - \eta_\ntfloor}{n} \sum_{k=1}^\ntfloor a_k \mu_k (\eta_\ntfloor - \eta_k) \\
    + \frac{(\eta_n - \eta_\ntfloor)^2}{n} v_\ntfloor^2.
    \end{multline*}
We then use the limits from Lemmas~\ref{lem: rate sum} and~\ref{lem: tricky conv} to conclude $\sum_{k=1}^\ntfloor \tmQ_{n,k} \to \tmW(t)$.  Further, for each fixed $k$, $\tmQ_{n,k}$ converges to the zero matrix. Finally, the limiting quadratic variation process for the case $p=\hp$, under Rademacher innovations, is obtained by arguing as in the case $p\in[0, \hp)$.
\end{proof}

 \begin{remark}
    \label{rem: QV sub ERW}
    In the above proof of Lemma~\ref{lem: QV}, Rademacher distribution of the innovations are used only to apply $\mathbb{E}(X_k^2|\mathcal{F}_{k-1})=1$ in~\eqref{eq: qv omT}~--~\eqref{eq: qv tmT}.
\end{remark}

Next, we consider the proof of Lemma~\ref{lem: Linden bd ERW}, which again does not require the assumption of Rademacher innovations.
\begin{proof}[Proof of Lemma~\ref{lem: Linden bd ERW}]
For the case $p\in[0, \hp)$, since $\{\oeta_n\}$ is a nonincreasing sequence, we have, for $1\le k\le \ntfloor$,
\[
\| \Delta \omT_{n,k} \|^2 \le \frac{1}{n} \left[ (1+pa_k \mu_k \oeta_k)^2 + a_k^2 \mu_k^2 \oeta_k^2 \frac{\oeta_n^2}{\oeta_\ntfloor^2} \right] (\Delta L_k)^2,
\]
and the expression in the square bracket is bounded by a term, possibly dependent on $t$.

For the case $p\in(\hp, p_c)$, choose $\rho>0$ such that $0<\rho < \gamma + 1/2 -p(\gamma +1)$. Then, we have, for $1\le k\le \ntfloor$,
\begin{align*}
\| \Delta \mT_{n,k} \|^2 &= \frac1n \left[ (1 - pa_k\mu_k\eta_k)^2 +  a_k^2 \mu_k^2 \eta_n^2 \right] (\Delta L_k)^2 \\
&= \frac1{n^{2\rho}} \left[ n^{2\rho-1} (1 - pa_k\mu_k\eta_k)^2 +  a_k^2 \mu_k^2 n^{2\rho-1} \eta_n^2 \right] (\Delta L_k)^2.
\end{align*}
The expression in the square bracket is bounded, as $\{a_k \mu_k \eta_k\}$ and $\{a_k \mu_k\}$ are convergent and $n^{2\rho-1} \eta_n^2$ converges to $0$.

Finally, for the case $p=\hp$, we have, for $1\le k\le \ntfloor$,
\begin{align*}
\|\Delta \boldsymbol{\widetilde{T}}_{n,k}\|^2 &\le \frac1n \left[ 2+ 2 a_k^2 \mu_k^2 \eta_n^2 + \frac{a_k^2\mu_k^2}{a_n^2\mu_n^2} \right] (\Delta L_k)^2 \\
&\le \frac{1}{\sqrt{n}} \left[ \frac2{\sqrt{n}} +  t^{1/4} \frac{a_k^2 \mu_k^2}{k^{1/4}} \left( 2 \frac{\eta_n^2}{n^{1/4}} + \frac1{n^{1/4} a_n^2\mu_n^2} \right) \right] (\Delta L_k)^2. 
\end{align*}
Again, the expression in the square bracket is bounded, as $\{a_n\mu_n\}$ and $\{\eta_n\}$ are slowly varying for $p=\hp$.
\end{proof}

The Lindeberg conditions have been derived only for SRRW-R-RVM.
\begin{proof}[Proof of Lemma~\ref{lem: Lindeberg ERW} for SRRW-R-RVM]
The proof of the Lindeberg conditions is immediate using $|\Delta L_k| \le 2$ for Rademacher innovations and Lemma~\ref{lem: Linden bd ERW}. 
\end{proof}

Finally, we establish uniform equicontinuity of the SRRW-RVM process at $0$. Here, we do not require the innovations to be Rademacher.
\begin{proof}[Proof of Lemma~\ref{lem: equi unif cont}]
Using~\eqref{conn L S}, we have
\[
\left(S_\ntfloor^*\right)^2 \le 2 \max_{1\le k\le\ntfloor} L_k^2 + 2 \left( \sum_{k=1}^\ntfloor \frac{|M_k|}{a_k \nu_k} \right)^2 = 2 \max_{1\le k\le\ntfloor} L_k^2 + 2 \sum_{k=1}^\ntfloor \sum_{l=1}^\ntfloor \frac{|M_k| |M_l|}{a_k a_l \nu_k \nu_l}.
\]
Taking expectation, dividing by $n$ and using Doob's $L^2$ inequality on the first term and Cauchy-Schwarz inequality on the second term, we get
$$\frac1n \mathbb{E} \Big( S_\ntfloor^* \Big)^2 \le \frac8n \mathbb{E} L_\ntfloor^2 + \frac2n \Big( \sum_{k=1}^\ntfloor {\sqrt{\mathbb{E} M_k^2}}/ (a_k \nu_k) \Big)^2.$$
By Lemma~\ref{Variance of the increments}, the first term is bounded by $8t$,
while, plugging in the asymptotics of $\mathbb{E} M_k^2$ for $p\in[0, p_c)$ from Proposition~\ref{Asymptotics of martingale M_n}, the second term is bounded by a constant multiple of $t$.
Combining and letting $t\to0$, the result follows. Finally, we use Chebyshev's inequality to conclude the proof.
\end{proof}
\begin{remark}
The above lemma shows that $S^*_n/\sqrt{n}$ is $L^2$-bounded when $p\in[0,p_c)$. Compare with Theorem~\ref{thm: SLLN}, where $S^*_n/n$ is $L^2$-negligible for all $p\in[0,1]$.
\end{remark}

We now prove the weak convergence of the normalized SRRW-R-RVM process in the subcritical regime. The proof of the theorem does not require the additional assumption, but is used only in the proofs of some of the lemmas, stated above, leading to the proof.
\begin{proof}[Proof of Theorem \ref{thm: subcrit} for SRRW-R-RVM]
Consider a decomposition of $ S_\ntfloor/{\sqrt{n}}$ in each of the following cases with $t>0$:
\begin{align}
    \text{for $p\in[0,\widehat{p})$,}
    &\qquad\frac1{\sqrt{n}}{S_{\lfloor nt\rfloor}} =\left(1,-pt^{p(\gamma+1)-\gamma}\right)\overline{\boldsymbol{T}}'_{n, \lfloor nt\rfloor}-p\left(0,\frac{\overline{\eta}_{\lfloor nt\rfloor}}{\overline{\eta}_n}-t^{p(\gamma+1)-\gamma} \right) \overline{\boldsymbol{T}}'_{n, \lfloor nt\rfloor}, \label{eq: weak sub hat}\\
    \text{for $p\in(\widehat{p},p_c)$,}
    &\qquad \frac1{\sqrt{n}}{S_{\lfloor nt\rfloor}} = \left(1,pt^{p(\gamma+1)-\gamma}\right) \boldsymbol{T}'_{n, \lfloor nt\rfloor} + p \left(0, \frac{\eta_{\lfloor nt\rfloor}}{\eta_n} -t^{p(\gamma+1)-\gamma}\right) \boldsymbol{T}'_{n, \lfloor nt\rfloor}, \label{eq: weak super hat}\\
    \text{and for $p=\widehat{p}$,}
    &\qquad \frac1{\sqrt{n}}{S_{\lfloor nt\rfloor}} =(1,\gamma\log t) \widetilde{\boldsymbol{T}}'_{n, \floor{nt}}+(0,\hp(\eta_{\floor{nt}}-\eta_n)a_n\mu_n-\gamma\log t) \widetilde{\boldsymbol{T}}'_{n, \floor{nt}}.\label{eq: weak hat}
    \end{align}
    In each case, we show the process weak convergence of the first term and the negligibility of the second in probability. Note that the multiplier functions of the first terms of the decompositions~\eqref{eq: weak super hat} and~\eqref{eq: weak hat} are undefined at $0$. 

    Using Proposition~\ref{neq gamma/gamma+1 joint invariance}, the bivariate martingales in the decompositions~\eqref{eq: weak sub hat}~-~\eqref{eq: weak hat} converge to processes with continuous paths on $[0, \infty)$. In the first terms of the decompositions~\eqref{eq: weak sub hat}~-~\eqref{eq: weak hat}, the multiplier functions are continuous on $(0,\infty)$ and hence bounded on $[1/T,T]$ for any $T>0$. For a fixed continuous function on $[1/T, T]$, if we consider the map which takes every function on $D([1/T,T])$ to its pointwise product with the fixed continuous function, then the map takes $D([1/T,T])$ to $D([1/T,T])$ and is continuous at every function in $D([1/T,T])$. Since, the Gaussian process limit of the bivariate martingale has continuous paths, by continuous mapping theorem, the first terms of the decompositions~\eqref{eq: weak sub hat}~-~\eqref{eq: weak hat} converge weakly to a centered Gaussian process with continuous paths in $D([1/T,T])$ for any $T>0$. The limiting process has the covariance kernel given by~\eqref{eq: cov kernel}. Thus, the limiting processes of the first terms of the decompositions~\eqref{eq: weak sub hat}~-~\eqref{eq: weak hat} have the same distribution as the process $(\mathcal{G}(t): t\in [1/T,T])$, given in the statement of Theorem~\ref{thm: subcrit}.
    
Next, by Proposition~0.5 of~\cite{Resnick1987}, the multiplier functions of the second terms in the decompositions~\eqref{eq: weak sub hat} and~\eqref{eq: weak super hat} converge to $0$ uniformly on $[1/T,T]$, for every $T>0$. The same result holds for the multiplier function of the second term of the decomposition~\eqref{eq: weak hat} due to Lemma~\ref{lem: tricky conv}. For any $T>0$, if a sequence of functions in $D([1/T,T])$ converge to a continuous function in Skorohod metric, then it also converges in supremum metric. Thus, the map taking $D([1/T,T])$ functions to their supremum over $[1/T,T]$ is continuous at every function in $C([1/T,T])$. Since the bivariate martingale has limits with continuous paths, their suprema over $[1/T,T]$ in the second terms of~\eqref{eq: weak sub hat}~-~\eqref{eq: weak hat} converge weakly. Hence, the second terms too converge to the zero process under the uniform metric (and thus, under Skorohod metric) on $[1/T, T]$ in probability.

Slutsky's theorem then gives us the required convergence in $D([1/T,T])$ under Skorohod topology for every $T>0$. The convergence is extended to $D((0,\infty))$ under Skorohod topology using Proposition~2.3 of~\cite{Janson1994}, and finally to $(D([0,\infty)),\mathcal{D})$, using Proposition~2.4 of~\cite{Janson1994} and Lemma~\ref{lem: equi unif cont}.
\end{proof}
\begin{remark} \label{rem: subcrit ERW}
    The above proof of Theorem~\ref{thm: subcrit} use the Rademacher assumption only through Lemmas~\ref{lem: QV} and~\ref{lem: Lindeberg ERW}.
\end{remark}

\begin{remark} \label{rem: BM subcrit}
For usual SRRW (constant $\{\mu_n\}$), whenever $0\le p < p_c$, the covariance kernel~\eqref{eq: cov kernel} simplifies to
$\bb{E}\mathcal G(s) \mathcal G(t)= s^{1-p} {t}^{p}/(1-2p)$,
for $0\le s\le t$. As observed by \cite{Bertoin2021b} (see Equation (3) of the same), the limiting process $(\mathcal G(t): t\ge 0)$ in this case, has the same law as that of $({t^p}B(t^{1-2p})/{\sqrt{1-2p}} : t\ge 0)$. This time changed Brownian motion arises as the scaling limit of the suitably normalized martingale process $(M_{\floor{nt}} : t\ge 0)$ and the factor $t^p$ is the limit of $a_n/a_{\floor{nt}}$.  Similarly, the SRRW, in the subcritical regime, viewed at a suitable (nonlinear) timescale $(nt^{{1}/(1-2p)})$, with a time-dependent spatial normalization $(\sqrt{n}t^{p/(1-2p)})$, also converges weakly to a constant multiple of a standard Brownian motion, namely,
$({\mathcal B(t)}/{\sqrt{1-2p}} : t\ge 0)$.
The idea of the proof is similar as above, except for noting the change in the space-time scale. 

For general $\mu_n$, the location $S_n$ admits either of the decompositions in \eqref{eq: S N split}. The joint convergence of the process $\omT'_{n, \floor{nt}} = \l( {\overline N_{\floor {nt}} }, { \overline \eta_n M_{ \floor{nt}}} \r) /{\sqrt n}$ (for $p\in [0,\hp)$) or $\mT'_{n, \floor{nt}} = \l( {N_{\floor {nt}}}, { \eta_n M_{ \floor{nt} }} \r) /{\sqrt n}$ (for $p\in (\hp, p_c)$) is given by Proposition~\ref{neq gamma/gamma+1 joint invariance}. From Lemma~\ref{lem: QV}, we see that the first component of the limiting process is a constant multiple (which is nonzero for $\gamma\ne 0$) of the standard Brownian motion, whereas the other is a constant multiple of the standard Brownian motion seen at $t^{2(1-p)(\gamma+1)-1}$ and the components are dependent unless $\gamma=0$. Thus, for $\mu_n$ with nonzero index of regular variation, one cannot obtain a suitable space-time scale which gives a Brownian motion limit for both the components, and as a consequence, for $( {S_{\floor {nt} }}/{\sqrt n} : t\ge 0)$. However, see Section~\ref{subsec: non lin time scale} for a comparison with the corresponding space-time scales leading to Brownian limit in the critical regime (in fdd), where ${N_n}/{\sigma_n}$ is negligible in probability and the asymptotic behavior of ${S_n}/{\sigma_n}$ is driven completely by $\eta_nM_n$. This suggets one to use a nonlinear timescale and time-dependent spatial normalization to obtain a Brownian motion limit.
\end{remark}
\subsection{Weak limit in the critical regime with unbounded \texorpdfstring{$\{v_n\}$}{vn} for SRRW-R-RVM}{\label{subsec: critical weak limit}}
In the critical regime, for the sequences $\{\mu_n\}$ with unbounded $\{v_n\}$ at $p=p_c(>\hp)$, we use~\eqref{eq: S N split}. In this case, $\sum_n 1/(a_n \nu_n)$ diverges. The contribution of the martingale $\{M_n\}$ is $\eta_n M_n$, which is of the order of $M_n/({a_n \mu_n})$. From Proposition~\ref{Asymptotics of martingale M_n}, the variance of $M_n$ grows like $v_n\to\infty$. Thus, $\eta_nM_n$ should be normalized by $\sigma_n= v_n/(a_n \mu_n)$. We also show the negligibility of $N_n/\sigma_n$ in probability. As a result, the weak limit of ${S_n}/{\sigma_n}$ is driven by that of ${\eta_nM_n}/{\sigma_n} \sim M_n/v_n$. We study the convergence of the triangular array of the martingale difference sequence
\begin{equation}{\label{eq: mg array in critical}}    
\Delta \widecap{M}_{n,k} = \frac1{v_n} \Delta M_k = \frac{a_k \mu_k}{v_n} \Delta L_k.
\end{equation}

We first state the result on the negligibility of the process $\l(N_{\floor{nt}}/\sigma_n : t\ge 0\r)$. 
 \begin{lemma}{\label{prop: L^2 in prob lim of mg M}}
    Let $\{S_n\}_{n\geq 0}$ be an SRRW-RVM with zero mean innovation sequence $\{\xi_n\}$ satisfying $\bb E\xi_1^2=1$. 
    Further assume that $p=p_c$ and the memory sequence $\{\mu_n\}$ is such that $\{v_n\}$ is unbounded. Then, for all $t\ge0$, we have $\max_{1\le k\le\floor{nt}} N_k/\sigma_n \to 0$ in $L^2$. 
\end{lemma}

Next, we collect results on the limiting quadratic variation, a Lindeberg condition, and the functional central limit theorem of $(\widecap{M}_{n,\floor{nt}} : t\ge0 )$. We start with the limiting quadratic variation.
 \begin{lemma}{\label{prop: QV proc in the critical regime}}
    Let $\{S_n\}_{n\geq 0}$ be an SRRW-RVM with the zero mean innovation sequence $\{\xi_n\}$ satisfying $\bb E\xi_1^2=1$. Assume $p=p_c$ and $\{v_n\}$ is unbounded. Then the quadratic variation process of $\{\widecap{M}_{n,\floor{nt}}\}$ satisfies
    $\ang{\,\widecap{M}\,}_{n,\floor{nt}} \to 1$ in $L^1$, for all $t>0$.
\end{lemma}
\begin{remark} \label{rem: QV cont 0}
    Note that the above limiting quadratic variation cannot be continuously extended to $0$.
\end{remark}

The following result provides the Lindeberg condition.
\begin{lemma}{\label{lemma: cond lind cond in critical case}}
    Let $\{S_n\}_{n\geq 0}$ be an SRRW-RVM with the zero mean innovation sequence $\{\xi_n\}$ satisfying $\bb E\xi_1^2=1$. Assume $p=p_c$ and $\{v_n\}$ is unbounded. Then, for all $t\ge0$, we have
    $$\sum_{k=1}^{\ntfloor} \bb{E} \Big( (\Delta\widecap{M}_{n,k})^2\mathbbm{1}_{\{|\Delta \widecap{M}_{n,k}| >\epsilon \}} | \mathcal{F}_{k-1} \Big) \to_{{\rm L_1}} 0.$$
\end{lemma}

The limiting quadratic variation process from Lemma~\ref{prop: QV proc in the critical regime} and the Lindeberg condition from Lemma~\ref{lemma: cond lind cond in critical case} give the martingale central limit theorem on $D([1/T, T])$, for any $T>0$ from Theorem~2.5 of~\cite{Durrett1978}. The process weak convergence is extended to $D((0, \infty))$ by Proposition~2.3 of~\cite{Janson1994}. As noted in Remark~\ref{rem: QV cont 0}, the convergence cannot be extended to $D([0, \infty))$ directly.
 \begin{proposition}
    {\label{lemma: time deformed M_n FCLT}}
    Let $\{S_n\}_{n\geq 0}$ be an SRRW-RVM with the zero mean innovation sequence $\{\xi_n\}$, satisfying $\bb E\xi_1^2=1$. Assume $p=p_c$ and $\{v_n\}$ is unbounded. Then, $( M_{\ntfloor}/v_n: t> 0)$ converges weakly in $D((0,\infty))$ with Skorohod topology to the process which always takes the value $Z$.
\end{proposition}
The limiting process above cannot be extended continuously to $0$, while the limit of $M_{\floor{nt}}/v_n$ at $t=0$ is $0$. We use the following analog of Lemma~\ref{lem: equi unif cont} to prove equicontinuity at $0$ and extend the convergence of $(S_{\floor{nt}}/\sigma_n:t>0)$ to $D([0,\infty))$.
 \begin{lemma}
    \label{lem: equi unif cont crit}
    Let $\{S_n\}_{n\geq 0}$ be the SRRW-RVM with zero mean, unit variance innovation sequence $\{\xi_n\}_{n\ge 1}$. 
Assume $p=p_c$ and $\{v_n\}$ is unbounded. Then, 
\[\adjustlimits\lim_{t\rightarrow 0} \limsup_{n\rightarrow\infty} \frac1{\sigma_n^2} \mathbb{E} \left( S^*_{\lfloor nt\rfloor}\right)^2 =0 \,\, \text{and, hence for all $\epsilon>0$,} \,\, \adjustlimits\lim_{t\rightarrow 0} \limsup_{n\rightarrow\infty} \mathbb{P} \left( S^*_{\lfloor nt\rfloor} > \epsilon \sigma_n \right) = 0.\]
\end{lemma}

We now provide the proofs of the above results. As in the subcritical regime in Section~\ref{subcrit reg}, we restrict some of the proofs for SRRW-R-RVM to bring out the main ideas of the proof, while deferring the general proof until Section~\ref{sec: SRRW}. Again, the proofs which do not use this additional assumption will be clearly indicated.
\begin{proof}[Proof of Lemma~\ref{prop: L^2 in prob lim of mg M}]
    Using Lemmas~\ref{Variance of the increments} and~\ref{lemma: trunc mart conv ERW}, we have 
    \begin{equation*} 
    \bb{E}(\Delta L_n)^2=\bb{E}X_n^2-p^2\mathbb{E}\l({Y_{n-1}}/{\nu_{n-1}}\r)^2\rightarrow1. 
    \end{equation*}
    Further, in the critical case, we also have $1 - p a_n \mu_n \eta_n \to - 2\gamma$. 
    Then, for Cesaro average, we get
    $\bb{E}N_n^2 \sim 4{\gamma^2}n$.
    The result then follows, as $\sigma_n^2/n\to\infty$ due to Karamata's theorem, and applying Doob's $L^2$ inequality.
\end{proof}
 \begin{remark} \label{rem: N neg crit}
    The proof holds for any mean-zero, unit-variance innovation sequence.
\end{remark}

The next result discusses the limit of the quadratic variation of the process defined by the martingale difference sequence~\eqref{eq: mg array in critical}.
\begin{proof}[Proof of Lemma~\ref{prop: QV proc in the critical regime} for SRRW-R-RVM]
    Using $\pm 1$ steps, we have $X_n^2=1$. Then, for all $t\ge0$, the quadratic variation process is given by:
    \begin{align}
    \ang{\,\widecap{M}\,}_{n,\floor{nt}} &= \frac1{v_n^2} \sum_{k=1}^\ntfloor a_k^2 \mu_k^2 \bb{E}\l(X_k^2|\mathcal{F}_{k-1}\r) -\frac{p^2}{v_n^2} \sum_{k=1}^{\ntfloor-1} a_{k+1}^2 \mu_{k+1}^2  \frac{M_{k}^2}{a_{k}^2\nu_{k}^2}\nonumber \\
    &= \frac{v_{\floor{nt}}^2}{v_n^2} -\frac{p^2}{v_n^2} \sum_{k=1}^{\ntfloor-1} a_{k+1}^2 \mu_{k+1}^2 \frac{M_{k}^2}{a_{k}^2\nu_{k}^2}. \label{eq: qv cap M}
    \end{align}
Since $\{v_n\}$ is slowly varying for the critical regime, the first term converges to $1$. 
From Lemma~\ref{lemma: trunc mart conv ERW}, $\mathbb E (M_n^2/(a_n^2 \nu_n^2)) \to 0$. Since $\{a_n \nu_n\}$ is bounded and $\{v_n\}$ is unbounded, using Toeplitz lemma, the second term is negligible in $L^1$. Combining, the result follows.
\end{proof}

\begin{remark}
    \label{rem: qv crit ERW}
    The above proof again uses simple steps of Rademacher distribution only to apply $\bb E(X_k^2 | \mathcal F_{k-1}) = 1$ in~\eqref{eq: qv cap M}.
\end{remark}

We obtain the Lindeberg conditions next.
\begin{proof}[Proof of Lemma~\ref{lemma: cond lind cond in critical case} for SRRW-R-RVM]
From~\eqref{eq: mg array in critical}, we have
\begin{equation} \label{eq: Lind cap M}
\sum_{k=1}^{\ntfloor} \bb{E} \left( (\Delta\widecap{M}_{n,k})^2\mathbbm{1}_{\{|\Delta \widecap{M}_{n,k}| >\epsilon \}} | \mathcal{F}_{k-1} \right) = \sum_{k=1}^{\ntfloor} \frac{a_k^2 \mu_k^2}{v_n^2} \bb{E} \l( (\Delta L_{k})^2\mathbbm{1}_{\left\{a_k \mu_k|\Delta L_{k}| >\epsilon v_n \right\}} | \mathcal{F}_{k-1} \r).
\end{equation}
For Rademacher innovation, we have $|\Delta L_{k}| \le 2$, which gives $$\bb{E} \big( (\Delta L_{k})^2\mathbbm{1}_{\left\{a_k \mu_k|\Delta L_{k}| >\epsilon v_n \right\}} | \mathcal{F}_{k-1} \big)\to_{\mathrm{L^1}} 0.$$ Again using Toeplitz lemma, we get the result.
\end{proof}
\begin{remark}
    \label{rem: Lind crit}
    We use $\pm 1$ steps of the Rademacher innovations to bound the martingale differences.
\end{remark}

Finally we check equicontinuity at $0$, which we prove for any SRRW-RVM with zero mean, unit variance innovation sequence.
\begin{proof}[Proof of Lemma~\ref{lem: equi unif cont crit}]
Using~\eqref{eq: N} and $\{\eta_n\}$ is nondecreasing, we have
$$\mathbb{E} \Big( S_\ntfloor^* \Big)^2 / \sigma_n^2  \le 2 \mathbb{E} \left( \max_{1 \le k\le \ntfloor} N_k^2 \right) / {\sigma_n^2} + 2 \eta_{\floor{nt}}^2 \mathbb{E} \left( \max_{1 \le k \le \floor{nt}} M_{k}^2 \right) / {\sigma_n^2}.$$ Lemma~\ref{prop: L^2 in prob lim of mg M} gives negligibility of the first term. The second term becomes negligible using Doob's $L^2$ inequality, then plugging in the asymptotics of $\mathbb{E} M_{\floor{nt}}^2$ for $p= p_c$ from Proposition~\ref{Asymptotics of martingale M_n} 
and letting $t\to0$. We finally use Chebyshev inequality.
\end{proof}
Now we prove the process weak limit for $p=p_c$ with unbounded $\{v_n\}$. The proof is restricted to SRRW-R-RVM only due to the use of the assumption in the proofs of some of the above lemmas. The proof of the theorem does not require any additional assumption. See Section~\ref{sec: SRRW} for the general case.
 \begin{proof}[Proof of Theorem \ref{thm: crit unbdd} for SRRW-R-RVM]
Using~\eqref{eq: N}, we have the decomposition:
    \begin{equation}\label{eq: critical split}
        \frac{S_{\lfloor nt\rfloor}}{\sigma_n}=\frac{N_{\lfloor nt\rfloor}}{\sigma_n} +  \left(\frac{\eta_{\lfloor nt\rfloor}}{\eta_n}-\sqrt{t} \right) p a_n \mu_n \eta_n \frac{M_{\lfloor nt\rfloor}}{v_n} +\sqrt{t} pa_n \mu_n \eta_n \frac{M_{\lfloor nt\rfloor}}{v_n}.
    \end{equation}
    Fix $T>0$. Lemma~\ref{prop: L^2 in prob lim of mg M} make the first term of the right side of~\eqref{eq: critical split} negligible in probability uniformly on $[0,T]$ and hence in $D([0,T])$ under the Skorohod topology. For $p=p_c$, Karamata's theorem gives $p a_n \mu_n \eta_n \to 2\gamma+1$. By Proposition~0.5 of~\cite{Resnick1987}, $(\eta_\ntfloor/\eta_n - \sqrt{t})$ converges uniformly to $0$ on $[0,T]$. Arguing as in Theorem~\ref{thm: subcrit}, using an application of the continuity theorem to Proposition~\ref{lemma: time deformed M_n FCLT} with supremum function over $[1/T,T]$ (note here that the process convergence in Proposition~\ref{lemma: time deformed M_n FCLT} exclude $t=0$) and further using Slutsky's theorem, the process corresponding to the second term uniformly over $[1/T, T]$ is negligible in probability. 
    The function $t\mapsto\sqrt{t}$ is continuous and bounded on $[0,T]$. Again, considering the pointwise multiplication map as in the proof of Theorem~\ref{thm: subcrit}, an application of the continuous mapping theorem and Proposition~\ref{lemma: time deformed M_n FCLT} show that the process corresponding to the third term of~\eqref{eq: critical split} converges weakly to the required limiting process in $D([1/T,T])$. The process convergence is extended to $D((0,\infty))$ using Proposition~2.3 of~\cite{Janson1994} and then to $D([0,\infty))$ using Lemma~\ref{lem: equi unif cont crit} and Proposition~2.4 of~\cite{Janson1994}.
\end{proof}

\section{Continuity of the limiting process} \label{sec: cont}
In this section, we study the continuity of the limiting process with respect to the memory parameter and the index of regular variation. {The proof does not use the Rademacher structure of the innovation sequence.}
\begin{proof}[Proof of Theorem~\ref{thm: tightness for G at crit}]
\begin{enumerate}[label=(\alph*)]
\item It is straightforward to check that the covariance kernel of the process $\mathcal G^{(p\bullet)}$ given in~\eqref{eq: cov kernel} is continuous in $p\in [0, p_c)$, including at $p=\hp$. The covariance kernel of the process $\mathcal L^{(p\bullet)}$ is given by 
\begin{equation}
    \label{eq: cov G L}
\mathbb{E} \left(\mathcal L^{(p\bullet)}(s) L^{(p\bullet)}(t) \right) = (2(1-p)(\gamma+1)-1) \mathbb{E} \left(\mathcal G^{(p\bullet)}(s) G^{(p\bullet)}(t) \right). 
\end{equation}
Multiplication by $(2(1-p)(\gamma+1)-1)$ removes the discontinuity of the covariance kernel of the process $\mathcal G^{(p\bullet)}$ at $p=p_c$ and the covariance kernel of the process $\mathcal L^{(p\bullet)}$ becomes continuous in $p\in[0, p_c]$. Thus, we have finite dimensional convergence of $\mathcal G^{(p\bullet)}$ for $p\in [0,p_c)$ and of $\mathcal L^{p\bullet}$ for $p\in [0,p_c]$.

To establish the continuity of $\mathcal G^{(p\bullet)}$ for $p\in[0,p_c)$, it is enough to show the tightness of $\{\mathcal G^{(p\bullet)}: p\le p'\}$ for any $p'<p_c$. We use Kolmogorov's continuity criterion (see Theorem~1.8 of~\cite{Revuz1999}) to establish the tightness. Since the processes are mean zero Gaussian processes starting at $0$, we need to check the moment condition only. 

 Now, for all $p\in[0, p_c)$ and for $0\le s\le t$, we have,
$\mathbb{E} \l(\mathcal G^{(p\bullet)}(s) \mathcal G^{(p\bullet)}(t) \r) \ge \mathbb{E} \l(\mathcal G^{(p\bullet)}(s)^2\r)$.
Then, further using Remark~\ref{CLT in diffusive regime}, we get
$\bb E(\mathcal G^{(p\bullet)}(t)-\mathcal G^{(p\bullet)}(s))^4 =3(\bb E(\mathcal G^{(p\bullet)}(t)-\mathcal G^{(p\bullet)}(s))^2)^2$
$\le 3 (\bb E ( \mathcal G^{(p\bullet)}(t) )^2-E ( \mathcal G^{(p\bullet)}(s) )^2)^2= 3 ( t-s )^2 \varsigma^4$.
 As $\varsigma^2 \equiv \varsigma^2(p)$ is increasing in $p\in [0,p_c)$ and $\varsigma^2(p')<\infty$ for every $p'<p_c$, $\{\bb E(\mathcal G^{(p\bullet)}(t)-\mathcal G^{(p\bullet)}(s))^4: p\le p'\}$ is bounded.

 We also observe that $\bb E (\mathcal G^{(p\bullet)}(1))^2 = \zeta^2(p) \to \infty$ as $p\uparrow p_c$. Thus, $\{\mathcal G^{(p\bullet)}(1)\}$ and, hence the collection $\{\mathcal G^{(p\bullet)}\}$, is not tight in any interval to the left of $p_c$.

To show the continuity of $\mathcal L^{(p\bullet)}$ for $p\in[0,p_c]$, we again establish the tightness of $\{\mathcal L^{(p\bullet)}: p\in[0,p_c]\}$. Arguing similarly as for the process $\mathcal G^{(p\bullet)}$ and using~\eqref{eq: cov G L}. we get $\bb E(\mathcal L^{(p\bullet)}(t)-\mathcal L^{(p\bullet)}(s))^4 \le 3 ( t-s )^2 (2\gamma+1-p)/(1-p) \le 3 ( t-s )^2 (2\gamma+1-p_c)/(1-p_c) < \infty$ for $p\le p_c$ and verify the relevant Kolmogorov's continuity criterion.

     \item As $\gamma\uparrow\infty$, we have $p_c\to1$. Thus, for any $p<1$, we have $p<p_c$ for any large enough $\gamma$. So the walk will eventually be in the subcritical regime. From~\eqref{eq: cov kernel}, it is immediate that $\bb E(\mathcal G^{(\bullet\gamma)}(s) \mathcal{G}^{(\bullet\gamma)}(t)) \to s/(1-p)$ for $0\le s\le t$ as $\gamma\to\infty$.
\end{enumerate}
\end{proof}

 \section{Barycenter Process} \label{sec: barycenter}
We apply the functional laws of large numbers and fluctuations for SRRW-RVM to a continuous map, developed in~\cite{Lo2018}, to obtain the asymptotic behavior of the barycenter process. The transform at the center of the analysis in~\cite{Lo2018} is given in Section~7.3 of the paper. The transform $\mathfrak g: D_0([0,\infty)) \to C_0([0,\infty))$ is given, for any $f\in D_0([0,\infty))$, as
\begin{equation}
    \label{eq: scaled int}
    \mathfrak g(f)(t) =
    \begin{cases}
        \frac1t \int_0^t f(u) du, &\text{when $t>0$},\\
        0, &\text{when $t=0$}.
    \end{cases}
\end{equation}
We summarize the properties of the transform in the following lemma.
\begin{lemma}
    \label{lem: bary transf}
    The transform $\mathfrak g: D_0([0,\infty)) \to C_0([0,\infty))$ given in~\eqref{eq: scaled int} is linear and continuous.
    Furthermore, the transform $\mathfrak g$ is a contraction under the the norm $\|\cdot\|_\infty$.
\end{lemma}
\begin{proof}
    We refer to Lemma~7.10 of~\cite{Lo2018} for a proof of the facts that, for any $f\in D_0([0,\infty))$, the function $\mathfrak g(f) \in C_0([0,\infty))$ and the map $\mathfrak g: D_0([0,\infty)) \to C_0([0,\infty))$ is continuous.

    To prove that $\mathfrak g$ is a contraction under $\|\cdot\|_\infty$, take any $T>0$. Then, for any $0\le t\le T$, we have, $|\mathfrak g(f)(t)| \le \int_0^t |f(s)| ds /t \le \sup_{0\le s\le t} |f(s)|$. Thus, we have
    \begin{equation}
        \label{eq: g contract}
        \sup_{0\le t\le T} |\mathfrak g(f)(t)| \le \sup_{0\le t\le T} |f(t)|,
    \end{equation}
    which gives the contraction.
\end{proof}

The following lemma studies the error in approximating the barycenter process with the transform of the original sequence.
\begin{lemma}
    \label{lem: error}
We consider a sequence of random variables $\{U_n\}_{n\ge0}$ with $U_0=0$ and the corresponding barycenter sequence given by $\overline U_0=0$ and, for $n\ge1$, $\overline U_n = \sum_{k=1}^n U_k/n$. Then, for any $T>0$, we have 
\begin{equation} \label{eq: contraction}
\sup_{[0,T]} |\overline U_{\floor{n\cdot}}| \le \sup_{[0,T]} | U_{\floor{n\cdot}}|. 
\end{equation}

The error process $\overline U_{\floor{n\cdot}} - \mathfrak g(U_{\floor{n\cdot}})$ satisfies, for any $0<\epsilon<T$,
\begin{equation}
    \label{eq: error bd}
    \sup_{[\epsilon, T]} \l| \overline U_{\floor{n\cdot}} - \mathfrak g(U_{\floor{n\cdot}}) \r| \le \frac{2}{n\epsilon} \sup_{[0,T]} |U_{\floor{n\cdot}}|.
\end{equation}

Furthermore, if the sequence of random variables $\{U_n\}$ grows sublinearly, that is $U_n/n\to0$ almost surely, then, for any sequence $\{\lambda_n\}$ going to $\infty$, the error process satisfies
\begin{equation} \label{eq: negli error}
\frac1{\lambda_n} \l\| \overline U_{\floor{n\cdot}} - \mathfrak g(U_{\floor{n\cdot}}) \r\|_\infty \to 0 \quad \text{almost surely}.
\end{equation}
\end{lemma}
\begin{proof}
For any $T>0$ and $0\le t\le T$, we have $|\overline U_{\floor{nt}}| \le \sum_{i=1}^{\floor{nt}} |U_i|/\floor{nt} \le \sup_{0\le t\le T} |U_{\floor{nt}}|$, which gives~\eqref{eq: contraction}.

From the proof of Lemma~7.4, p.~42 of~\cite{Lo2018}, we obtain, for $t>0$, that
\begin{align*}
    \overline U_{\floor{nt}} - \mathfrak g \l(U_{\floor{n\cdot}}\r)(t) 
    &= \frac{nt-\floor{nt}}{nt} \overline{U}_{\floor{nt}} - \frac{nt-\floor{nt}-1}{nt} {U}_{\floor{nt}},
\intertext{and hence, using~\eqref{eq: contraction}, we have}
    \l| \overline U_{\floor{nt}} - \mathfrak g \l(U_{\floor{n\cdot}}\r)(t) \r| 
    &\le \frac1{nt} \l( \overline U_{\floor{nt}} + U_{\floor{nt}} \r) \le \frac{2}{nt} \sup_{0\le s\le t} U_{\floor{ns}},
\end{align*}
which gives~\eqref{eq: error bd}.

Finally, the proof of Lemma~7.4 of~\cite{Lo2018} actually shows that, for $0\le t\le T$, we have 
$$\sup_n\sup_{0\le t\le T} |\overline{U}_{\floor{nt}} - \mathfrak g (U_{\floor{n t}})| <\infty$$ almost surely and~\eqref{eq: negli error} follows.
\end{proof}

This transform allows the asymptotic almost sure or weak behavior of a sequence to be transferred to its barycenter process.
\begin{theorem}[Transfer theorem for fluctuations of barycenter process]
    \label{thm: transfer}
We consider a sequence of random variables $\{U_n\}_{n\ge0}$ with $U_0=0$ and the corresponding barycenter sequence given by $\overline U_0=0$ and, for $n\ge1$, $\overline U_n = \sum_{k=1}^n U_k/n$. Further, assume that the sequence of random variables $\{U_n\}$ grows sublinearly, that is $U_n/n\to0$ almost surely.

If the sequence $\{U_n\}_{n\ge0}$ satisfies an invariance principle with respect to an appropriate normalization $\Xi_n$ going to $\infty$:
\begin{equation}
    \label{eq: inv prin T}
    \frac1{\Xi_n} U_{\floor{n\cdot}} \to U_\infty \quad \text{almost surely (weakly respectively) in $(D_0([0,\infty)), \mathcal D)$.}
\end{equation}
Then the barycenter sequence $\{\overline U_n\}_{n\ge0}$ also satisfies an invariance principle with respect to the same normalization $\Xi_n$:
\begin{equation*}
    \frac1{\Xi_n} \overline{U}_{\floor{n\cdot}} \to \overline{U}_\infty := \mathfrak g(U_\infty) \quad \text{almost surely (weakly respectively) in $(D_0([0,\infty)), \mathcal D)$.}
\end{equation*}
\end{theorem}
\begin{proof}
Since $\mathfrak g$ is a continuous transform, using the continuous mapping theorem for the almost sure or weak convergence accordingly, from~\eqref{eq: inv prin T}, we have:
\begin{equation} \label{eq: error}
    \frac1{\Xi_n} \mathfrak g\l(U_{\floor{n\cdot}}\r) \to \overline U_\infty = \mathfrak g\l(U_\infty\r) \quad \text{almost surely (weakly respectively) in $(D_0([0,\infty)), \mathcal D)$.}
\end{equation}
Since, $\Xi_n\to\infty$, using~\eqref{eq: negli error}, we have
\begin{equation}
    \label{eq: error est}
    \frac1{\Xi_n} \left\| \overline U_{\floor{n\cdot}} - \mathfrak g\left(U_{\floor{n\cdot}}\right)\right\|_\infty \to 0 \quad \text{almost surely}.
\end{equation}
Combining~\eqref{eq: error} and~\eqref{eq: error est}, the result follows.
\end{proof}

We are now ready to prove the limiting behavior of the barycenter process of SRRW-RVM, given in Theorem~\ref{thm: barycenter}. The law of large numbers follow from a contraction behavior. For the fluctuations, the weak and almost sure behavior follow from Theorem~\ref{thm: transfer}. However, $L^2$ convergence will require careful analysis, especially near $0$, using the power function structure of the limit process.
\begin{proof}[Proof of Theorem~\ref{thm: barycenter}]
For $p=1$, we have $S_n = n\xi_1$ for all $n\ge1$, giving $\overline S_n = (n+1)\xi_1/2$ for all $n\ge1$ and the law of large numbers results for the case $p=1$ follow immediately. The law of large numbers for $p\in [0,1)$ follows from~\eqref{eq: contraction}.

For the fluctuation of the barycenter process, note the limiting processes in Corollary~\ref{cor: phase} start at the origin. Thus, the weak and almost sure limiting behaviors of the normalized barycenter process follows from Theorem~\ref{thm: transfer} using Corollary~\ref{cor: phase}. The limiting processes in the critical and the supercritical cases follow directly. We only need to check the covariance kernel in the subcritical case. However, we first prove the $L^2$ convergence for bounded $\{v_n\}$.

For $L^2$ convergence, due to~\eqref{eq: sigma Karamata}, it is enough to prove 
\begin{equation} \label{eq: L2 target}
\l( a_n \mu_n \overline{S}_{\floor{nt}}: t\ge0 \r) \to_{\rm L^2} \l( \frac{\gamma+b}{b(1+b)} M_\infty t^b: t\ge0 \r) \quad \text{in $(D([0, \infty), \|\cdot\|_\infty)$}. 
\end{equation}
Fix $0<\epsilon<T$. Then, we have, using~\eqref{eq: g contract} and~\eqref{eq: error bd},
\begin{align}
    &\sup_{\epsilon\le t\le T} \left| a_n \mu_n \overline S_{\floor{nt}} - \frac{\gamma+b}{b(1+b)} M_\infty t^b \right| \nonumber\\ 
    &\qquad \le \sup_{\epsilon\le t\le T} \left| \mathfrak g \left( a_n \mu_n S_{\floor{nt}} - \frac{\gamma+b}{b} M_\infty t^b \right) \right| + a_n \mu_n \sup_{\epsilon\le t\le T} \left| \overline{S}_{\floor{n\cdot}} - \mathfrak g\l(S_{\floor{n\cdot}}\r) \right| \nonumber\\
    &\qquad \le \sup_{0\le t\le T} \left| a_n \mu_n S_{\floor{nt}} - \frac{\gamma+b}{b} M_\infty t^b \right| + \frac{2a_n \mu_n}{\epsilon} \sup_{0\le t\le T} \left| \frac1n {S}_{\floor{nt}} \right|. \label{eq: L2 upper}
\end{align}
The first term on the right side of~\eqref{eq: L2 upper} is negligible in $L^2$ by Theorem~\ref{thm: bdd quad var}~\eqref{supercrit proc} and the second term on the right side of~\eqref{eq: L2 upper} is negligible in $L^2$ using Theorem~\ref{thm: SLLN} and the fact that $a_n \mu_n \to 0$, whenever $\{v_n\}$ is bounded.

Furthermore, using~\eqref{eq: contraction}, we have
\begin{align}
    &\sup_{0\le t\le\epsilon} \left| a_n \mu_n \overline S_{\floor{nt}} - \frac{\gamma+b}{b(1+b)} M_\infty t^b \right| \nonumber\\ 
    &\qquad \le \sup_{0\le t\le\epsilon} \left| a_n \mu_n S_{\floor{nt}} \right| + \frac{\gamma+b}{b(1+b)} M_\infty \epsilon^b \nonumber\\
    &\qquad \le \sup_{0\le t\le T} \left| a_n \mu_n S_{\floor{nt}} - \frac{\gamma+b}{b} M_\infty t^b \right| + \frac{(\gamma+b)(2+b)}{b(1+b)} M_\infty \epsilon^b. \label{eq: L2 lower}
\end{align}
Again, the first term on the right side of~\eqref{eq: L2 lower} is negligible in $L^2$ by Theorem~\ref{thm: bdd quad var}~\eqref{supercrit proc} and the second term on the right side of~\eqref{eq: L2 upper} is negligible in $L^2$ as $\epsilon\to 0$, since $M_\infty$ has finite second moment. Combining~\eqref{eq: L2 upper} and~\eqref{eq: L2 lower}, we have, for any $T>0$,
\[
\sup_{0\le t\le T} \left| a_n \mu_n \overline S_{\floor{nt}} - \frac{\gamma+b}{b(1+b)} M_\infty t^b \right| \to_{\rm L^2} 0,
\]
which proves~\eqref{eq: L2 target}.

Finally, we consider the limiting covariance kernel in the subcritical case. Applying Corollary~\ref{cor: phase} to Theorem~\ref{thm: transfer}, we get $\mathcal K {=}_{\rm d} {\sqrt{1-2b}}\, \mathfrak g( \mathcal G)$, where $\mathcal G$ is the limiting process mentioned in Theorem~\ref{thm: subcrit} and Corollary~\ref{cor: phase}. 
Note that, for all $0\le s\le t$, we have
    \begin{equation}{\label{eq: cov G split}}
    \bb E\l(\int_0^t \int_0^s \mathcal{G}(u) \mathcal{G}(v) du\, dv\r)= 2 \int_{0}^s\int_{0}^v \bb E(\mathcal{G}(u) \mathcal{G}(v))du\, dv + \int_{s}^t \int_{0}^s \bb E(\mathcal{G}(u) \mathcal{G}(v))du\, dv.
     \end{equation}
Plugging in~\eqref{eq: cov kernel} for $0\le u\le v\le s$ in the first term on the right side of~\eqref{eq: cov G split}, we have,
\begin{equation}
    \int_{0}^s\int_{0}^v \bb E(\mathcal{G}(u) \mathcal{G}(v))du\,dv =
    \begin{cases}
        \frac{1}{b(1-b)}\l[\frac{(\gamma+b)(\gamma + 1-b)}{3(2-b)(1-2b)} -\frac{\gamma(\gamma + 1)}{6}\r]s^3, & \text{for $b\ne 0$},\\
        \frac{5\gamma(\gamma+1)+2}{12}s^3, & \text{for $b=0$}. \label{eq: split term 1}
    \end{cases}
\end{equation}
Again, plugging in~\eqref{eq: cov kernel} for $0\le u\le s\le v\le t$ in the second term on the right side of~\eqref{eq: cov G split}, we have,
\begin{align}
    &\int_s^t \int_0^s \bb E(\mathcal{G}(u) \mathcal{G}(v))du\, dv\nonumber\\ 
    &\qquad \qquad \qquad= 
    \begin{cases}
        \frac{1}{b(1-b)}\l[\frac{(\gamma + b) (\gamma +1-b)}{(1-2b)} \frac{s^{2-b}}{2-b} \frac{t^{1+b}-s^{1+b}}{1+b} -\frac{\gamma(\gamma +1)}{2} s^2(t-s)\r], & \text{for $b\ne 0$},\\
        \frac{3\gamma(\gamma+1)+2}{4} s^2(t-s) - \frac{\gamma(\gamma+1)}2 s^2t \log \l(\frac{s}{t}\r), & \text{for $b=0$}.
    \end{cases} \label{eq: split term 2}
\end{align}
Plugging~\eqref{eq: split term 1} and~\eqref{eq: split term 2} back in~\eqref{eq: cov G split} and simplifying, we find that the limiting covariance $(1-2b) \bb E\l(\int_0^s \int_0^t \mathcal{G}(u) \mathcal{G}(v) du\, dv\r)/(st)$ for $0\le s\le t$ is given by~\eqref{eq: lim cov}.
\end{proof}

\section{Some Illustrative Examples for the Critical Regime}{\label{sec: examples}}
In this section, we study the spatial normalization $\sigma_n$ used in the critical regime $p=p_c$. We provide explicit rates for $\sigma_n$ corresponding to two wide classes of the memory sequence $\{\mu_n\}$. All the rates, except $\sqrt{n \log n}$, are novel for SRRW-RVM, to the best of our knowledge. The spatial normalizations for the weak process limit in the critical regime with unbounded $\{v_n\}$ can be smaller than, larger than or of the same order as $\sqrt{n \log n}$. However the normalizations for the almost sure and $L^2$ limit in the critical regime with bounded $\{v_n\}$ are always heavier than $\sqrt{n \log n}$. Observe that, for bounded $\{v_n\}$ case in the critical regime, we must have $\sum a_n^2 \mu_n^2 < \infty$. If the normalization $\sigma_n\sim v_\infty/(a_n \mu_n)$ is of order smaller than or equal to $\sqrt{n \log n}$, then $\sum a_n^2 \mu_n^2 = \infty$. So, the spatial normalization must necessarily be heavier than $\sqrt{n \log n}$ in the bounded $\{v_n\}$ case in the critical regime.
The normalization depends on the slowly varying asymptotic variance scale $\{v_n\}$, which is a function of the slowly varying part of the memory sequence $\{\mu_n\}$. We summarize different examples given in this section in Table~\ref{tab: example}.
\begin{table}[ht]
    \centering
    \begin{tabular}{|l|c|c|c|}
        \hline
        \textbf{Example number}  &$\boldsymbol{\{v_n\}}$
        &\textbf{Mode of convergence}
        &\textbf{Normalization}  \\
        \hline  \hline 
        \ref{ex: contd prod},~\ref{ex: zeta alpha zero pos}~-~\ref{ex: log log pos} &unbounded &weak &$\sim\sqrt{n \log n}$ \\
        \hline
        \ref{ex: zeta zero},~\ref{ex: zeta power zero},~\ref{ex: log log zero},~\ref{ex: log log log zero} &unbounded &weak &$\gg\sqrt{n \log n}$ \\
        \hline
        \ref{ex: lighter than n log n} &unbounded &weak &$\ll\sqrt{n \log n}$ \\
        \hline
        \ref{ex: zeta neg}~-~\ref{ex: log log neg},~\ref{ex: zeta power kappa neg zero},~\ref{ex: log log as} &bounded &almost sure and $L^2$ &$\gg\sqrt{n \log n}$ \\
        \hline
    \end{tabular}
    \caption{List of examples with different spatial normalizations in the critical regime}
    \label{tab: example}
\end{table}

\begin{example}
    \label{ex: contd prod}
    We start with the memory sequence given by~\eqref{contd prod} and used in~\cite{Laulin2022, Roy2024}. For that choice of memory sequence, it follows using induction that $\nu_n = n\mu_{n+1}/(\gamma+1)$ for all $n\ge 1$. As a result, for $p=p_c$, we have $a_n \sim \Gamma(\gamma+3/2) n^{-(\gamma+1/2)}$ and $\sigma_n \sim \sqrt{n \log n}$, which is the spatial normalization used for $S_n$ under the critical regime in~\cite{Laulin2022, Roy2024}. 
\end{example}

In general, for the explicit rates of $\sigma_n$, we compute the rate of $a_n$ first, which is determined by $\sum_{k=1}^{n-1}{\mu_{k+1}}/{\nu_k}$.
\begin{lemma}{\label{log a_n}}
    For $p=p_c$, 
    $\left\{\log a_n + (\gamma+1/2)(\gamma+1)^{-1} \sum_{k=1}^{n-1} {\mu_{k+1}}/{\nu_k}\right\}$
    is convergent.
\end{lemma}
\begin{proof}
    From \eqref{contd prod}, we have for $p=p_c$,
    $$-\log a_n = \frac{\gamma +\frac{1}{2}}{\gamma+1} \sum_{k=1}^{n-1} \frac{\mu_{k+1}}{\nu_k} + \sum_{k=1}^{n-1} \Oh\l(\frac{\mu_{k+1}^2}{\nu_k^2}\r).$$
    The result then follows, as $\mu_{k+1}/\nu_k \sim (\gamma+1)/k$ from Karamata's theorem.
\end{proof}

We first assume the memory sequence to be embedded into a regularly varying function. 
\begin{assum}
    \label{assmp: gen}
    There exists a function $\mu:[1,\infty)\to(0,\infty)$, which is eventually monotone and regularly varying of index $\gamma$, such that $\mu_n = \mu(n)$.
\end{assum}
We also consider the integrated version of the function $\mu$, given by 
$\nu(x) = \int_1^x \mu(s) ds$,
to compare with the sequence $\{\nu_n\}$. By Karamata's theorem, $\nu$ is regularly varying of index $(\gamma+1)>0$ and thus diverges to $\infty$. The next lemma compares the sequences $\{\mu_{n+1}/\nu_n\}$, appearing in Lemma~\ref{log a_n}, and $\{\mu(n+1)/\nu(n+1)\}$.
\begin{lemma}{\label{err rate mu_n/nu_n}}
For $\mu$ in Assumption~\ref{assmp: gen} and its integrated version $\nu$, the sequence
$$\left\{\l|\frac{\mu_{n+1}}{\nu_n}- \frac{\mu(n+1)}{\nu(n+1)}\r|\right\}$$ is summable.
\end{lemma}
\begin{proof}
  Using $\mu_{n+1}= \mu(n+1)$ and Karamata's theorem, first note that 
  \begin{equation} \label{eq: nu bd}
    \l|\frac{\mu_{n+1}}{\nu_n}-\frac{\mu(n+1)}{\nu(n+1)}\r| = \frac{\mu(n+1)}{\nu_n\nu(n+1)}\l|\nu(n+1)-\nu_n\r| = \Oh \l( \frac{1}{n^2 \mu_n}\r) \l|\nu(n+1)-\nu_n\r|
  \end{equation}
  For eventually monotone increasing $\mu$, choose $N_0\in \bb{N}$ such that $\mu$ is monotone increasing on $[N_0,\infty)$. Then for all $n\geq N_0$,
  \begin{align}
    \nu(n+1) - \nu_n &= \sum_{k=1}^n \int_k^{k+1} (\mu(x) - \mu_k) dx \ge \nu(N_0) - \nu_{N_0-1}, \label{eq: nu lower}
  \quad \text{and}\\
        \nu(n+1) - \nu_n &= \sum_{k=1}^n \int_k^{k+1} (\mu(x) - \mu_k) dx\nonumber\\
        &\leq \nu(N_0) - \nu_{N_0-1} + \sum_{k=N_0}^n(\mu_{k+1}-\mu_k) \le \mu_{n+1} + \nu(N_0). \label{eq: nu upper}
    \end{align}
    Using $\nu_n\to\infty$ and $\nu(n+1)\to\infty$, we have $|\nu(n+1)-\nu_n| = \Oh (\mu_n)$ and the result follows from~\eqref{eq: nu bd}. For eventually monotone decreasing $\mu$, choose $N_0\in \bb{N}$ such that $\mu$ is monotone decreasing on $[N_0,\infty)$. Then for all $n\geq N_0$, the inequalities~\eqref{eq: nu lower} and~\eqref{eq: nu upper} are reversed and $|\nu(n+1)-\nu_n|$ is bounded. The result then follows from~\eqref{eq: nu bd} as $\gamma>-1$.
\end{proof}
Lemma~\ref{err rate mu_n/nu_n} allows us to rewrite Lemma~\ref{log a_n} as follows.
\begin{corollary}
    \label{cor: log a_n}
    For $p=p_c$, $\left\{\log a_n + (\gamma+\frac12)(\gamma+1)^{-1} \sum_{k=1}^{n-1} {\mu(k+1)}/{\nu(k+1)}\right\}$ is convergent.
\end{corollary}

The explicit rate of the sequence $\{a_n\}$ and the normalization $\sigma_n$ depends on $\ell(x)$, the slowly varying part of $\mu$, where we write $\mu(x) = x^\gamma \ell(x)$. We consider two broad subclasses of the function $\ell$ and analyze $\sigma_n$ in the next two subsections.

 \subsection{ {Regularly} varying function of \texorpdfstring{$\log n$} {log n}}{\label{logn RV examples}} 
Some common choices of slowly varying functions are $\log x$, $(\log x)^{-1}$, $\log \log x$, $\exp((\log \log x)^\alpha)$ with $0<\alpha<1$ etc., all of which are monotone, differentiable and regularly varying functions of $\log x$ with the derivative being regularly varying again. We consider such functions as the slowly varying factor of $\mu$.

\begin{assum}{\label{eg 1}}
Suppose $\mu(x)=x^{\gamma}\ell(x)$ satisfies:
\begin{enumerate}
 \item  For some $ B\ge 1$, the slowly varying function $\ell(x)=f(\log x),  x \ge e^B$, where, for some $\alpha \in \mathbb R$, the function $f:  [B,\infty) \rightarrow (0,\infty)$ is regularly varying of index $\alpha$ having the form
     \begin{equation}{\label{eq: RV f}}
    f(x) = x^{\alpha} \exp\l(\int_{ \log B}^{\log x} \zeta(s) ds\r),
    \end{equation} 
    where the function $\zeta$ is integrable on $[ \log B,x)$ for all $x>0$ and $\zeta(x)\to0$ as $x\to\infty$. \label{assmp: zeta}
    \item We further assume the function $\zeta$ to be one of the following three types:
    \begin{enumerate}
        \item The function $\zeta$ is the identically zero function. \label{zeta zero}
        \item The function $\zeta$ is nonincreasing  (and hence positive), regularly varying of index $-\rho$ with $\rho \in (0,1]$, such that $\int_{\log B}^\infty \zeta(x) dx = \infty$, and is also eventually differentiable with monotone derivative. \label{zeta pos}
        \item The function $\zeta = - \widetilde \zeta$, where $\widetilde \zeta$ satisfies the conditions in~(b) above. \label{zeta neg}
    \end{enumerate} \label{assmp: zeta ii}
\end{enumerate}
\end{assum}
\begin{remark} \label{rem: assmp}
The form~\eqref{eq: RV f} of $f$ is motivated by the Karamata's representation of a regularly varying function. We ignore possible oscillations of the slowly varying factor. Three choices of $\zeta$ in Assumption~\ref{eg 1}~\eqref{assmp: zeta ii} correspond to the cases where the slowly varying factor of $f$ is constant, diverges to $\infty$ or converges to $0$. 
Further, observe that for $\rho>1$, the integral in~\eqref{eq: RV f} converges, and can be absorbed in the constant. For simplicity, we assume the convergent factor in Karamata representation of $f$ to be a constant. Further, since we consider ratios of the form $\mu(n)/\nu(n)$ in Corollary~\ref{cor: log a_n}, we take the constant to be $1$. We also rule out $\rho<0$ to ensure $\zeta(x)\to 0$, as required in Karamata's representation. 

Since $\zeta'$ is assumed to be eventually monotone in Assumptions~\ref{eg 1}~\eqref{zeta pos} and~\ref{eg 1}~\eqref{zeta neg}, by the monotone density theorem (see Theorem~1.7.2 of~\cite{Bingham1987}), we get $x  |\zeta'(x)| \sim \rho \zeta(x)$, as $x\to\infty$. Furthermore, we disallow $\rho=0$ to guarantee $|\zeta'|$ to be regularly varying of index $-(\rho+1)$.
\end{remark}

As $\zeta$ is differentiable, $f$ is also twice differentiable. We gather the formulas for $f'$ and $f''$.
\begin{lemma}{\label{prop: f' and f''}}
    Let $f$ be the regularly varying function defined in \eqref{eq: RV f}. Then we have, for all large enough $x$,
     \begin{align*}
     f'(x)&={(\alpha+ \zeta(\log x))} \frac{f(x)}{x} \quad \text{ and }\\ f''(x)&=\l\{(\alpha+ \zeta(\log x))(\alpha-1+ \zeta(\log x))+ \zeta'(\log x) \r\} \frac{f(x)}{x^2}.
     \end{align*}
    Further, the functions $|f'|$ and $|f''|$ are regularly varying of indices $\alpha-1$ and $\alpha-2$ respectively, except when $\alpha\in\{0,1\}$ and $\zeta$ is identically zero. If $\alpha=0$ and $\zeta$ is the zero function, then $f'$ and $f''$ both are identically zero functions. If $\alpha=1$ and $\zeta$ is the zero function, then $f''$ is the identically zero function.
\end{lemma}

We are now ready to obtain explicit rates for the normalization $\sigma_n$ under Assumption~\ref{eg 1}.
\begin{theorem}{\label{rates l_n eg 1}}
    For $\mu(x)=x^\gamma \ell(x)$ satisfying Assumption~\ref{eg 1} and for $p=p_c$, we have
    $a_n^2\mu_n^2\sim {C_\mu} {\ell_n^{{1}/{(\gamma+1)}}}/n$,
    where $C_\mu$ is a positive constant, possibly depending on the memory sequence $\{\mu_n\}$. The sequence $\{v_n\}$ is unbounded {\rm (}respectively, bounded{\rm )} when $\alpha+\gamma+1$ is positive {\rm (}respectively, negative{\rm )}. Further, when $\alpha+\gamma+1>0$, we have
    \[
    v_n^2 \sim \frac{\gamma+1}{\alpha+\gamma+1} a_n^2 \mu_n^2 n \log n \quad \text{and} \quad \sigma_n^2 \sim 
        \frac{\gamma+1}{\alpha+\gamma+1}n\log n. 
    \]
\end{theorem}
\begin{proof}
Applying integration by parts twice in succession, the expression
    \[
    \nu(x) - \frac{x^{\gamma+1}f(\log x)}{\gamma+1} + \frac{x^{\gamma+1}f'(\log x)}{(\gamma+1)^2} - \frac{1}{(\gamma+1)^2} \int_1^x y^{\gamma}f''(\log y)dy\] is a constant.
    If $\alpha\in\{0,1\}$ and $\zeta\equiv 0$, then $f''\equiv 0$ and the integral above disappears; else, from Lemma~\ref{prop: f' and f''}, $f''$ is regularly varying and hence $y\mapsto f''(\log y)$ is slowly varying. Karamata's theorem gives
    $\int_1^x y^{\gamma}f''(\log y)dy \sim x^{\gamma+1} f''(\log x) / (\gamma+1) = \Oh \l( x^{\gamma+1} f(\log x) (\log x)^{-2} \r)$.
    Combining and using Lemma~\ref{prop: f' and f''}, we get
    \begin{align*}
    \frac{\mu(n)}{\nu(n)} &= \frac{\gamma+1}{n} \l[1-\frac{\alpha+ \zeta(\log\log n)}{(\gamma+1)\log n}+ \Oh \l( (\log n)^{-2} \r)\r]^{-1}\\
        &=\frac{\gamma+1}{n} +\frac{\alpha+ \zeta(\log\log n)}{n\log n}+\Oh\l(\frac{1}{n(\log n)^2}\r).
    \end{align*}
Thus, from Corollary~\ref{cor: log a_n}, we obtain
    $$\log a_n +\l(\gamma+ \frac{1}{2}\r)\sum_{k=1}^{n-1}{k}^{-1} +\l(\gamma+ \frac{1}{2}\r)\sum_{k=2}^{n-1}\frac{\alpha+ \zeta(\log\log k)}{(\gamma+1)k\log k}$$
    is convergent. The remaining sums can be replaced by the corresponding integral and a convergent error sequence using the integral test -- see Theorem~8.23 of ~\cite{Apostol1974}. Combining,
    \[
    -\log a_n=\l(\gamma+\frac{1}{2}\r)\log n +\frac{\gamma+\frac12}{\gamma+1}\l(\alpha\log\log n+\int_{ \log B}^{\log\log n}\zeta(x)dx\r) - \frac12 \log C_\mu,
    \]
    which then gives
    $a_n \sim \sqrt{C_\mu} n^{-(\gamma+1/2)} \ell_n^{-(\gamma+1/2)/(\gamma+1)}$,
    where $C_\mu$ is a positive constant depending on the function $\mu$. Plugging in $\{\mu_n\}$, we get the rate for $a_n^2 \mu_n^2$.

    To study boundedness of $\{v_n\}$, we check the summability of $\{a_n^2 \mu_n^2\}$. We again use the integral test, for which we need to check the function $x\mapsto x^{-1} f(\log x)^{1/(\gamma+1)}$, or equivalently $x\mapsto {x^{-(\gamma+1)}} f(\log x)$ to be eventually decreasing (since $\gamma+1$ is positive). The derivative of the latter function is eventually negative iff (using Lemma~\ref{prop: f' and f''})
    ${f'(x)}/{f(x)} = (\alpha + \zeta(\log x))/{x} < \gamma+1$
    eventually. This is true as the left side goes to $0$ as $x\to\infty$, by Assumption~\ref{eg 1}.

    By the integral test, we have $\{a_n^2 \mu_n^2\}$ to be summable iff $\int_1^x s^{-1} f(\log s)^{1/(\gamma+1)} ds$ converges iff $\int_0^{\log x} f(s)^{1/(\gamma+1)} ds$ converges iff $\int_0^{x} f(s)^{1/(\gamma+1)} ds$ converges. The condition for the boundedness of $\{v_n\}$ then follows from the fact that the function $f$ is a regularly varying function of index $\alpha$. Remaining results follow from Karamata's theorem.
\end{proof}

We conclude this subsection with the convergence results for the normalized SRRW-RVM in the critical case corresponding to Theorem~\ref{rates l_n eg 1}, which depends on the sign of $\alpha + \gamma + 1$.
We begin with the following special case of Theorem~\ref{thm: crit unbdd} for $p=p_c$ and $\{\mu_{n}\}$ satisfying Assumption \ref{eg 1} with $\alpha+\gamma+1>0$.
\begin{corollary}{\label{corr: nlogn scale}}
    Let $\{S_n\}_{n\geq 0}$ be the SRRW-RVM with zero mean, unit variance innovation sequence $\{\xi_n\}$. Then for $p=p_c$ and $\{\mu_n\}$ satisfying Assumption \ref{eg 1} with $\alpha+\gamma+1>0$,  
    \[
        \l(\frac{S_{\floor{nt}}}{\sqrt{n\log n}}:t\geq 0\r)\to_{\rm{w}}\l(\sqrt{t} N\l(0,\frac{(2\gamma+1)^2(\gamma+1)}{\alpha+\gamma+1}\r): t\geq 0\r) \quad \text{in $(D([0,\infty)), \mathcal D)$}.
        \]
\end{corollary}
\begin{remark}
    Taking $\mu\equiv 1$, for $\alpha=\gamma=0$ and $ \zeta\equiv 0$ in Corollary~\ref{corr: nlogn scale}, we get the usual SRRW with $\sqrt{n \log n}$ normalization.
\end{remark}

The following examples give memory sequences $\{\mu_n\}$ such that Corollary~\ref{corr: nlogn scale} hold. We assume $B=1$ in~\eqref{eq: RV f}.
\begin{example}[$\zeta\equiv0$, $\alpha=0$, $\gamma+1/2>0$] \label{ex: zeta alpha zero pos}
Here we have $\mu(x)=x^{\gamma}$.
\end{example}
\begin{example}[$\zeta\equiv0$, $\alpha+\gamma+1>0$, $\gamma+1/2>0$] \label{ex: zeta pos}
Here we have $\mu(x)=x^{\gamma} (\log x)^\alpha$.
\end{example}
\begin{example}[$\zeta(x)={\kappa(1-\rho)}/{x^{\rho}}$ with $\kappa\ne0$ and $0<\rho<1$, $\alpha+\gamma+1>0$, $\gamma+1/2>0$] \label{ex: zeta power pos}
    Here we have $\mu(x)=x^\gamma (\log x)^\alpha \exp{\l( {\kappa} (\log\log x)^{1-\rho} \r)}$.
\end{example}
\begin{example}[$\zeta(x)=\kappa({x}^{-1} \wedge 1)$ with $\kappa\ne0$, $\alpha+\gamma+1>0$, $\gamma+1/2>0$] \label{ex: log log pos}
    In this case $\rho=1$ and we have $\mu(x)=x^\gamma (\log x)^\alpha ( e \log \log x)^\kappa$.
\end{example}
These examples can be easily modified so that the slowly varying factor of the memory sequence $\{\mu_n\}$ is some power of the iterates of logarithm or its exponential.

Next, we obtain a special case of Theorem~\ref{thm: bdd quad var} for $p=p_c$, when $\mu$ satisfies Assumption~\ref{eg 1} with $\alpha+\gamma+1<0$. Theorem~\ref{rates l_n eg 1} gives bounded $\{v_n\}$ and almost sure and in $L^2$ convergence. The following corollary considers a large class of such models. We then give specific examples to show that such class is not vacuous. These examples are completely novel in the literature, to the best of our knowledge.
\begin{corollary}{\label{corr: superdiffusive a.s. conv eg 1}}
    Let $\{S_n\}_{n\geq 0}$ be the SRRW-RVM with zero mean, unit variance innovation sequence $\{\xi_n\}$. Then for $p=p_c$ and $\{\mu_n\}$ satisfying Assumption~\ref{eg 1} with $\alpha+\gamma+1<0$,    
    \[\l(\sqrt{{\ell_n^{{1}/{(\gamma+1)}}}} \frac{S_{\floor{nt}}}{\sqrt{n}}: t\geq 0\r) \to \l(\sqrt{t}K_{\infty} : t\geq 0\r), \quad \text{almost surely and in $L^2$ in $D([0,\infty))$,}
    \]
     where $K_{\infty}$ is a nonrandom multiple (possibly depending on $\{\mu_n\}$) of $M_\infty$, the almost sure and $L^2$ limit of $M_n$.
\end{corollary}
We now discuss examples where Corollary~\ref{corr: superdiffusive a.s. conv eg 1} holds. We consider examples analogous to those corresponding to Corollary~\ref{corr: nlogn scale}. Since we require $\alpha < -(\gamma+1) <0$,  we cannot have analogue to Example~\ref{ex: zeta alpha zero pos}, where $\alpha=0$. Also, in the following examples, we record the corresponding memory sequence $\{\mu_n\}$ and the spatial normalization ${1}/({a_n\mu_n})$ only. As noted earlier, all the normalizations below are of larger order than $\sqrt{n \log n}$. Again, we assume $B=1$ in~\eqref{eq: RV f}.
\begin{example}[$\zeta\equiv 0$, $\alpha+\gamma+1<0$, $\gamma+1/2>0$] \label{ex: zeta neg}
    In this case $\mu(x)=x^{\gamma}(\log x)^\alpha$ and the normalization is $\sqrt{n(\log n)^{-{\alpha}/{(\gamma+1)}}}$.
\end{example}
\begin{example}[$\zeta(x)={\kappa(1-\rho)}/{x^{\rho}}$ with $\kappa\ne0$ and $0<\rho<1$, $\alpha+\gamma+1<0$, $\gamma+1/2>0$] \label{ex: zeta power neg}
    In this case $\mu(x)=x^{\gamma}(\log x)^\alpha \exp(\kappa(\log \log x)^{1-\rho})$ and the spatial normalization is 
    $$\left[n(\log n)^{-{\alpha}/{(\gamma+1)}}\exp{\l(-{\kappa}(\gamma+1)^{-1}(\log\log n)^{1-\rho}\r)}\right]^{1/2}.$$
\end{example}
\begin{example}[$\zeta(x)=\kappa({x}^{-1} \wedge 1)$ with $\kappa\ne0$, $\alpha+\gamma+1<0$, $\gamma+1/2>0$] \label{ex: log log neg}
    Here we have $\rho=1$. In this case $\mu(x)=x^{\gamma}( \log x)^{\alpha}(e \log\log x)^{\kappa}$ and the spatial normalization is 
    \[\left[n(\log n)^{-{\alpha}/({\gamma+1})}(\log\log n)^{-{\kappa}/({\gamma+1})}\right]^{1/2}.\]
\end{example}

We end this subsection with examples where $\alpha+\gamma+1=0$. As $\alpha = -(\gamma+1) <0$, again, there is no analogue of Example~\ref{ex: zeta alpha zero pos}, where $\alpha=0$. The sequence $\{v_n\}$ may be bounded or unbounded and the convergence may be almost sure (also, $L^2$) or in distribution. In all cases, the spatial normalizations are heavier than $\sqrt{n \log n}$, the usual spatial normalization for the critical SRRW. The normalizations have powers of additional iterates of logarithms.  In Examples~\ref{ex: zeta zero}~-~\ref{ex: log log as}, we obtain the normalizations from Theorem~\ref{rates l_n eg 1} and then obtain the limit using Theorem~\ref{thm: crit unbdd} for unbounded $\{v_n\}$, or using Theorem~\ref{thm: bdd quad var} for bounded $\{v_n\}$. We provide only the function $\mu$ and the limit in each case. We continue to use $B=1$ in~\eqref{eq: RV f}.
\begin{example}[$\zeta\equiv0$, $\alpha+\gamma+1=0$, $\gamma+1/2>0$] \label{ex: zeta zero}
    In this case $\mu(x) = x^{\gamma} (\log x)^{-(\gamma+1)}$ and $\{v_n\}$ is unbounded. Then,  
    $\l({S_{\floor{nt}}}/{\sqrt{n\log n\log\log n}}: t\ge 0\r)\to_{\rm{w}} \l((2\gamma+1)\sqrt{t}Z: t\geq 0\r)$ in $(D([0,\infty)), \mathcal D)$.
\end{example}
\begin{example}[$\zeta(x)={\kappa{x^{-\rho}}(1-\rho)}$ with $\kappa> 0$ and $0<\rho<1$, $\alpha+\gamma+1=0$, $\gamma+1/2>0$] \label{ex: zeta power zero}
    In this case, we have
    $\mu(x)=x^{\gamma}(\log x)^{-(\gamma+1)} \exp(\kappa(\log\log x)^{1-\rho})$.
    Then $\{v_n\}$ is unbounded, giving
    $({S_{\floor{nt}}}/{\sqrt{n\log n(\log\log n)^\rho}} : t\geq 0)\to_{\rm{w}} \l((2\gamma+1)\sqrt{(\gamma+1)t/(\kappa(1-\rho)})Z: t\geq 0\r)$ in $(D([0,\infty)), \mathcal D)$.
\end{example}
\begin{example}[$\zeta(x)= {\kappa {x^{-\rho}} (1-\rho)}$ with $\kappa < 0$ and $0<\rho<1$, $\alpha+\gamma+1=0$, $\gamma+1/2>0$] \label{ex: zeta power kappa neg zero}  The function $\mu$ is as in Example~\ref{ex: zeta power zero}, but as $\kappa<0$, $\{v_n\}$ is bounded. Then, with $K_\infty$ being a nonrandom multiple of $M_\infty$, we get
\[  \l( {\frac{1}{\sqrt{n\log n}} } \exp \l( \frac{\kappa}{2(\gamma+1)} (\log \log n)^{1-\rho} \r) S_{\floor{nt}}: t\geq 0 \r) \to \l(\sqrt{t} K_\infty: t\geq 0\r),\]
almost surely and in $L^2$ in $D([0,\infty))$
\end{example}
The following three examples correspond to the case $\rho=1$.
\begin{example}[$\zeta(x)=\kappa(x^{-1} \wedge 1)$ with $\kappa+\gamma+1>0$, $\alpha+\gamma+1=0$, $\gamma+1/2>0$] \label{ex: log log zero}
    In this case, $\mu(x)=x^{\gamma}(\log x)^{-(\gamma+1)} (e \log\log x)^{\kappa}$ and $\{v_n\}$ is unbounded. Then, we have,
    $({S_{\floor{nt}}}/{\sqrt{n\log n\log\log n}}: t\geq 0)\to_{\rm{w}} ((2\gamma+1)\sqrt{{(\gamma+1)t}/(\kappa+\gamma+1)}Z: t\geq 0)$ {in $(D([0,\infty)), \mathcal D)$}.
\end{example}
\begin{example}[$\zeta(x)=\kappa({x^{-1}} \wedge 1)$ with $\kappa+\gamma+1=0$, $\alpha+\gamma+1=0$, $\gamma+1/2>0$] \label{ex: log log log zero}
    Here $f(x) = (e x \log x)^{-(\gamma+1)}$ and $\{v_n\}$ is unbounded. Then,  $({S_{\floor{nt}}}/{\sqrt{n\log n\log\log n\log\log\log n}}: t\geq 0)\to_{\rm{w}} ((2\gamma+1)\sqrt{t}Z: t\geq 0)$ {in $(D([0,\infty)), \mathcal D)$}. The product of the iterates of logarithm in the spatial normalization is interesting.
\end{example}
\begin{example}[$\zeta(x)=\kappa({x^{-1}} \wedge 1)$ with $\kappa < \alpha = -(\gamma+1)<0$] \label{ex: log log as}
Here $\mu$ is as in Example~\ref{ex: log log zero}. With $\kappa<0$, $\{v_n\}$ is bounded. With $\widetilde K_\infty$ being a nonrandom multiple of $M_\infty$, we have
    $({S_{\floor{nt}}}/{\sqrt{n\log n(\log\log n)^{-{\kappa}/{(\gamma+1)}}}}: t\geq 0) \to \l(\sqrt{t} \widetilde K_{\infty}: t\geq 0\r)$
{almost surely and in $L^2$ in $D([0,\infty))$}. 
\end{example}

\subsection{Slower than \texorpdfstring{$n\log n$}{nlog n} growth}{\label{General RV with rho neq 1}}
We next consider a large class of slowly varying functions, so that the resulting SRRW-RVM grows at a rate slower than $\sqrt{n \log n}$ in the critical regime. As in Section~\ref{logn RV examples}, we consider a regularly varying function $\mu$ of index $\gamma$, so that $\mu_n = \mu(n)$ and $\mu(x) = x^\gamma \ell(x)$. We make the following assumptions on the slowly varying $\ell$.
\begin{assum}
    \label{assum: eg 2}
    We assume that the function $\mu(x) = x^\gamma \ell(x)$ satisfies:
    \begin{enumerate}
        \item  For some $B>1$, the slowly varying function $\ell(x)$ has the form,  for $x>B$,
$\ell(x) = \exp \l( \int^{\log x}_{ \log B} \delta(s) ds \r)$.
\item The function $\delta$ is monotone decreasing to $0$ as $x\to\infty$ and integrable on $(\log B, x)$ for all $x>0$. \label{item: dec}
\item We assume that $\int^\infty_{ \log B} \delta(x)=\infty$ and that $\delta$ is regularly varying of index $-\rho$, for some $\rho\in(0,1]$. Also $\delta$ is assumed to be differentiable at least $m$ times where $m\rho>1$, with $\delta^{(k)}$ being the derivative of the $k$-th order, for $1\le k\le m$, and that $|\delta^{(k)}|$ is regularly varying of index $-(\rho+k)$ for all $1\le k\le m$. (Further, denote the function $\delta$ as  $\delta^{(0)}$.) \label{item: RV derivatives}
    \end{enumerate}
\end{assum}
\begin{remark}
    As in Remark~\ref{rem: assmp}, the slowly varying function $\ell$ is motivated as a special case of Karamata's representation. In particular, we assume $\rho>0$ to ensure the existence of the integer $m$, which is crucial for estimating the order of $\{a_n\}$ in Theorem~\ref{prop: eg 2}. 
    
    The slowly varying function is still given by $\ell(x) = f(\log x)$, where 
    \begin{equation}
        \label{eq: f slow}
        f(x) = \exp \l(\int^x_{ \log B} \delta(s) ds\r).
    \end{equation}
    Compare~\eqref{eq: f slow} with~\eqref{eq: RV f}.
    Examples of such functions $\delta$ include ${x^{-\rho}}$, $x^{-1}(\log x)^{-\alpha}$ with $0<\alpha<1$.
\end{remark}

We show that the functions $\mu$ satisfying Assumption~\ref{assum: eg 2} lead to another different class of normalizations for the corresponding SRRW-RVM in the critical regime. We estimate $\{a_n^2\mu_n^2\}$, $\{v_n^2\}$ and $\{\sigma_n^2\}$ under Assumption~\ref{assum: eg 2}. Since, $\delta$ is differentiable $m$ times, so is the function $f$, with the $k$-th derivative denoted by $f^{(k)}$, for $1\le k\le m$.
\begin{lemma}{\label{lem: eg 2}}
    Suppose the function $\mu$ satisfies Assumption~\ref{assum: eg 2}. For $1\le k\le m$,
    $g_k(x):={f^{(k)}(x)}/{f(x)} -(\delta(x))^k$
    is a polynomial in $\delta^{(l)}(x)$, $0\le l<k$, and no term is solely a power of $\delta(x)$. In particular, we also have $g_k(x)=\Oh(|\delta'(x)|)$, for all $1\le k\le m$.
\end{lemma}
\begin{proof}
    We use induction. For $k=1$,~\eqref{eq: f slow} gives $f'(x)= \delta(x) f(x)$, leading to $g_1\equiv 0$.
    
    We have $f^{(k)}(x) = f(x) \l(g_k(x) + \delta(x)^k \r)$. Differentiating and using the expression for $f'(x)$, we get 
    $f^{(k+1)}(x) = f(x) \l(g_k'(x) + k \delta(x)^{k-1} \delta'(x) + g_k(x) \delta(x) + \delta(x)^{k+1} \r)$,
    giving
    $g_{k+1}(x) = g_k'(x) + k \delta(x)^{k-1} \delta'(x) + g_k(x) \delta(x)$.

    Assume the result holds for some $k\ge 1$. By the induction hypothesis, the last two terms for $g_{k+1}$ are polynomials with the required properties. Since $\delta(x)\to0$, they are also $\Oh(|\delta'(x)|)$. The first term of $g_{k+1}$ has the same properties due to the properties of $g_k$. Comparing the indices of regular variation, $g_k'(x) = \Oh(|\delta'(x)|)$ proving the induction step.
\end{proof}
\begin{theorem}{\label{prop: eg 2}}
    Let the function $\mu$ satisfy  Assumption~\ref{assum: eg 2}. Then for $p=p_c$, we have $a_n^2\mu_n^2 \sim C_{\mu}  \ell_n^{{1}/{(\gamma+1)}}/n$, with $C_{\mu}$ a positive constant, depending on the memory sequence $\{\mu_n\}$. Further, if $\int_1^\infty f(s)^{1/(\gamma+1)} ds = \infty$, then we also have
    \[
    v_n^2 \sim C_\mu \int_1^{\log n} f(s)^{\frac{1}{\gamma+1}} ds \qquad \text{and} \qquad \sigma_n^2\sim {n}{\ell_n^{-\frac{1}{\gamma+1}}} \int_1^{\log n} f(s)^{\frac{1}{\gamma+1}} ds.
    \]
\end{theorem}
\begin{proof}
    Successive applications of integration by parts gives us,
    \begin{equation} \label{eq: int by parts m}
    \nu(x)=\frac{x^{\gamma+1}}{\gamma+1} \l[ f(\log x) + \sum_{k=2}^m (-1)^{k-1} \frac{f^{(k-1)}(\log x)}{(\gamma+1)^{k-1}} \r] + \frac{(-1)^m}{(\gamma+1)^m} \int_{0}^{x} s^\gamma f^{(m)}(\log s) ds + \Oh(1).
    \end{equation}
    Using $|\delta'(\log x)|$ and $\delta(\log x)^m$ slowly varying, and Lemma~\ref{lem: eg 2}, the integral on the right side is of the order of
    $\int_0^x s^\gamma ( |\delta'(\log s)| + \delta(\log s)^m) f(\log s) ds 
    \sim x^{\gamma+1} ( |\delta'(\log x)| + \delta(\log x)^m) \ell(x)/ (\gamma+1)$.
    
    Then, from~\eqref{eq: int by parts m}, using Lemma~\ref{lem: eg 2}, we get
    \begin{align*}
    \nu(x) 
    &= \frac{x^{\gamma+1}\ell(x)}{\gamma+1} \l[ 1 + \sum_{k=1}^{m-1} (-1)^{k} \l( \frac{\delta(\log x)}{\gamma+1} \r)^{k} + \Oh \l( \delta'(\log x) \r) + \Oh \l( \delta(\log x)^m \r) \r]\\
    &=\frac{x\mu(x)}{\gamma+1} \l( 1+\frac{\delta(\log x)}{\gamma+1} \r)^{-1} \l[ 1 + \Oh \l( \delta'(\log x) \r) + \Oh \l( \delta(\log x)^m \r) \r].
    \end{align*}
    Using $|\delta'|$ is regularly varying of index $-(1+\rho)$ and thus $\delta'(x) =\Oh (x^{-(1+\rho/2)})$, we have
    \begin{equation}
    \frac{\mu(n)}{\nu(n)} = 
    \frac{\gamma+1}{n}+\frac{\delta(\log n)}{n}+ \Oh \l(\frac{(\log n)^{-(1+{\rho}/{2})}}{n}\r)+ \Oh \l(\frac{\delta(\log n)^m}{n}\r). \label{eq: delta log n}
    \end{equation}
    
    Since $\delta(x)^m$ is monotone and regularly varying of index $-\rho m<-1$, the last two quantities are summable, giving
    $\sum_{k=1}^{n-1} {\mu(k+1)}/{\nu(k+1)} = (\gamma+1)\log n+\int_{0}^{\log n}\delta(s)ds - \log C_\mu/2$.
    Then Corollary~\ref{cor: log a_n} gives the required rate for the sequence $\{a_n^2 \mu_n^2\}$. As in the proof of Theorem~\ref{rates l_n eg 1} the function $x \mapsto x^{-1} f(\log x)^{1/(\gamma+1)}$ is eventually decreasing. Approximating $v_n^2 = \sum_{k=1}^n a_k^2 \mu_k^2$ by the corresponding integral, the rates for the sequences $\{v_n^2\}$ and $\{\sigma_n^2\}$ follow.
\end{proof}

We now provide an illustration of the rates obtained in Theorem~\ref{prop: eg 2}, applied to Theorem~\ref{thm: crit unbdd}. The spatial normalization here is lighter than the traditional  $\sqrt{n \log n}$, in contrast to the examples considered in Section~\ref{logn RV examples}.
\begin{example}[$\delta(x)= (1-\rho)x^{-\rho}$ with $0<\rho<1$, $\gamma>-1/2$, $B=1$] \label{ex: lighter than n log n}
    In this case, we have $\mu(x)=x^{\gamma} \exp \l( (\log x)^{1-\rho} \r)$ and
    \[
    \l(\frac{S_{\floor{nt}}}{\sqrt{n(\log n)^\rho}}: t\geq 0\r)\to_{\rm{w}}\left((2\gamma+1)\sqrt{\frac{(\gamma+1) t}{1-\rho} } Z: t\geq 0\right) \quad \text{in $(D([0,\infty)), \mathcal D)$.}
    \]
\end{example}

\begin{remark}
    In summary, heavier memory sequence begets lighter spatial normalization, as illustrated in the examples of Sections~\ref{logn RV examples} and~\ref{General RV with rho neq 1}. However, Theorem~\ref{prop: eg 2} does not analyze one of the heaviest memory sequences given by $\rho=0$ in Example~\ref{ex: lighter than n log n}. The main difficulty is the summability of $\delta(\log n)^m/n$ in~\eqref{eq: delta log n}. One example of such memory sequence is $\mu_n = n^\gamma \exp \l( {\log n}/{\log \log n} \r)$. 
\end{remark}
    
    This suggests the open problem:
    \begin{open}
        For $\gamma>-1/2$, the memory sequence
        $\mu_n = n^\gamma \exp \l( {\log n}/{\log \log n} \r)$ and $p=p_c$, prove the weak convergence of the normalized SRRW-RVM $S_n/{\sqrt{n \log \log n}}$ to a centered Gaussian random variable.
    \end{open}
\subsection{Nonlinear timescale and time-dependent spatial normalization for process weak limit in the critical regime}{\label{subsec: non lin time scale}}
It was proved in Theorem \ref{thm: crit unbdd} that the normalized SRRW-RVM, with linearly scaled time, converges weakly to a random element in $D([0,\infty))$, whose paths are a Gaussian multiple of the square root function, under the critical regime $p=p_c$ and unbounded $\{v_n\}$. The spatial normalization is free of time. The result holds for the usual SRRW and the model considered in~\cite{Laulin2022}. However, for these models, generally an exponential timescale $\floor{n^t}$ is considered, along with a time-dependent spatial normalization $\sqrt{n^t \log n}$, to obtain a Brownian motion limit. In fact, Theorem~1.5 of~\cite{Bertenghi2022} shows that ${S_{\floor{n^t}}}/{\sqrt{n^t\log n}}$, converges weakly to the standard Brownian motion in $(D([0,\infty)),\mathcal{D})$.  Using examples in Sections~\ref{logn RV examples} and~\ref{General RV with rho neq 1}, we show such an exponential timescale and time-dependent spatial normalization need not lead to the Brownian motion limit, or any nondegenerate limit at all. There can still be appropriate nonlinear timescale and time-dependent spatial normalization to obtain Brownian motion limit. We provide such timescales and spatial normalizations in each example.

{
\renewcommand{\thetheorem}{\ref{ex: zeta zero}}
\begin{example}[Continued]
\addtocounter{theorem}{-1}
Here $\mu(x) = x^\gamma (\log x)^{-(\gamma+1)}$ satisfies Assumption~\ref{eg 1} with $\alpha+\gamma+1=0$ and $\zeta\equiv 0$. The martingale difference array, defined in \eqref{eq: mg array in critical}, will be critical for the analysis. We study the functional central limit theorems for the martingale sampled along the timescales $\floor{n^t}$ and $\exp \l( (\log n)^t \r)$, where the latter gives us the required Brownian motion limit. We start with the quadratic variation process along the required time sequences. 
\end{example}
}

\begin{lemma}{\label{lem: QV of M_n for general time scale}}
    Let $\{S_n\}_{n\geq 0}$ be an SRRW-RVM with zero mean, unit variance innovation sequence $\{\xi_n\}$, the memory sequence $\mu_n = n^\gamma (\log n)^{-(\gamma+1)}$ and the memory parameter $p=p_c$. Then, the martingale difference array defined in~\eqref{eq: mg array in critical} satisfies, for $t>0$,    
    $\ang{\,\widecap{M}\,}_{n,\floor{\exp \l( (\log n)^t \r)}} \to t$, and 
    $\ang{\, \widecap{M} \,}_{n,\floor{n^t}} \to 1$ {in probability}.   
\end{lemma}

\begin{proof}
Theorem~\ref{rates l_n eg 1} gives $v_n^2 \sim C_\mu \log \log n$. Then, for every $t>0$,  as $n\rightarrow \infty$, ${v_{\floor{n^t}}^2} \sim {v_n^2}$ and
${v_{\floor{ \exp\l( (\log n)^t \r)}}^2} \sim t {v_n^2}$.
The results follow as in Lemma~\ref{prop: QV proc in the critical regime} with $\floor{n^t}$ and $\floor{\exp\l((\log n)^t\r)}$ replacing $\floor{nt}$.
\end{proof}
Similarly, the Lindeberg conditions hold as in Lemma~\ref{lemma: cond lind cond in critical case} with $\floor{nt}$ replaced by $\floor{n^t}$ and $\floor{\exp\l((\log n)^t\r)}$.
\begin{lemma}{\label{lem: cond Lind for general time scale}}
    Let $\{S_n\}$ be an SRRW-RVM with zero mean, unit variance innovation sequence $\{\xi_n\}$, $\mu_n = n^\gamma (\log n)^{-(\gamma+1)}$ and $p=p_c$. Then, the martingale difference array defined in \eqref{eq: mg array in critical} satisfies, for $t>0$,
    \begin{align*}
    \sum_{k=1}^{\floor{n^t}} \bb{E} \l( (\Delta\widecap{M}_{n,k})^2\mathbbm{1}_{\{|\Delta \widecap{M}_{n,k}| >\epsilon \}} | \mathcal{F}_{k-1} \r) &\to_{\rm{P}} 0 \\
    \sum_{k=1}^{\floor{\exp\l((\log n)^t\r)}} \bb{E} \l( (\Delta\widecap{M}_{n,k})^2\mathbbm{1}_{\{|\Delta \widecap{M}_{n,k}| >\epsilon \}} | \mathcal{F}_{k-1} \r) &\to_{\rm{P}} 0.
    \end{align*}
\end{lemma}
Theorem~2.5 of~\cite{Durrett1978}, Lemmas~\ref{lem: QV of M_n for general time scale} and~\ref{lem: cond Lind for general time scale} give the process convergence for $\{M_n\}$ along these time sequences.
\begin{lemma}{\label{Martingale M_n general FCLT}}
    Let $\{S_n\}_{n\geq 0}$ be an SRRW-RVM with zero mean, unit variance innovation sequence $\{\xi_n\}_{n\ge 1}$, the memory sequence $\mu_n = n^\gamma (\log n)^{-(\gamma+1)}$ and the memory parameter $p=p_c$. Then, we have, on $D((0,\infty))$,
    \begin{align*}
        \l(\frac1{\sqrt{n^t \log n \log \log n}}{p \eta_{\floor{n^t}} M_{\floor{n^t}}}: t>0\r) &\to_{\rm{w}} \l( (2\gamma +1)\sqrt{t} \mathcal B(1): t>0\r), \quad \text{and} \\
        \l(\frac{p \eta_{\floor{\exp \l((\log n)^t\r)}} M_{\floor{\exp\l( (\log n)^t\r)}}}{\sqrt{\exp\l( (\log n)^t \r) (\log n)^t \log \log n}} : t>0\r) & \to_{\rm{w}} ((2\gamma+1) \mathcal B(t): t> 0).
    \end{align*}
\end{lemma}

Finally, using Lemma~\ref{prop: L^2 in prob lim of mg M}, we show the negligibility of the normalized process of
the martingale $N_n$.
\begin{lemma}
    \label{lem: N gen neg}
    Let $\{S_n\}_{n\geq 0}$ be an SRRW-RVM with zero mean, unit variance innovation sequence $\{\xi_n\}$, the memory sequence $\mu_n = n^\gamma (\log n)^{-(\gamma+1)}$ and the memory parameter $p=p_c$. Then, for the martingale $N_n$, defined in~\eqref{eq: N}, we have, for all $t>0$,
    \[
    \frac{N_{\floor{n^t}}}{\sqrt{n^t \log n\log\log n}} \to_{\rm{P}} 0 \quad \text{and} \quad \frac{N_{\floor{\exp\l( (\log n)^t\r)}}}{\sqrt{\exp\l( (\log n)^t \r) (\log n)^t \log \log n}} \to_{\rm{P}} 0.
    \]
\end{lemma}

Lemmas~\ref{Martingale M_n general FCLT} and~\ref{lem: N gen neg} with~\eqref{eq: S N split} give the finite dimensional convergence of the SRRW-RVM.
\begin{theorem} \label{thm: rescale fdd}
    Let $\{S_n\}_{n\geq 0}$ be an SRRW-RVM with zero mean, unit variance innovation sequence $\{\xi_n\}$, the memory sequence $\mu_n = n^\gamma (\log n)^{-(\gamma+1)}$ and the memory parameter $p=p_c$. Then, we have
    \[
    \l(\frac{S_{\floor{n^{t}}}}{\sqrt{n^{t}\log n\log\log n}} : t> 0 \r) \to_{\rm{fdd}} \l( (2\gamma+1) \sqrt{t} \mathcal{B}(1): t> 0 \r), \quad \text{and} \]
    \[
    \l(\frac{S_{\floor{\exp\l( (\log n)^t\r)}}}{\sqrt{\exp\l( (\log n)^t \r) (\log n)^t \log \log n}}  : t> 0 \r) \to_{\rm{fdd}} \l( (2\gamma+1) \mathcal B(t): t> 0 \r).
    \]
\end{theorem}

\begin{remark}{\label{rem: fdd}}
    We do not get the process convergence in Theorem~\ref{thm: rescale fdd}, as we fail to show the negligibility in Lemma~\ref{lem: N gen neg} as a process in $D([0, \infty))$. However, a different limit under finite dimensional convergence with the exponential timescale $n^t$ shows that the limit cannot be a Brownian motion, as is true for the SRRW with constant memory sequence $\{\mu_n\}$.
\end{remark}

This leads us to the following open problem:
\begin{open}
    Obtain process convergence in $D([0,\infty), \mathcal D)$ for Theorem~\ref{thm: rescale fdd}.
\end{open}

We next study Examples~\ref{ex: zeta power zero}~--~\ref{ex: log log zero} to determine the limits of the normalized process under the exponential timescale and to identify the space-time scales, giving the Brownian motion limit. The argument is similar to that in Example~\ref{ex: zeta zero} and is omitted. As before, all the convergences are given for the finite dimensional distributions, because of the limitations mentioned in Remark~\ref{rem: fdd}.
{
\renewcommand{\thetheorem}{\ref{ex: zeta power zero}}
\addtocounter{theorem}{-1}
\begin{example}[Continued]
Under exponential timescaling, we get 
\[
\l(\frac{S_{\floor{n^t}}}{\sqrt{n^t\log n(\log \log n)^\rho}}:t>0\r) \to_{\rm{fdd}} \l((2\gamma+1)\sqrt{\frac{(\gamma+1)t}{\kappa(1-\rho)}}\mathcal{B}(1) : t>0\r).
\]
For Brownian motion limit, consider
$$\tau_n(t) = \l\lfloor \exp \l( \exp \l( \l( \frac{\gamma +1}{\kappa} \log t + (\log\log n)^{1-\rho}\r)^{{1}/(1-\rho)} \r) \r)  \r\rfloor,$$ giving
\[
\l(\sqrt{\frac{t}{\tau_n(t)\log \tau_n(t)(\log \log n)^\rho}} S_{\tau_n(t)}:t>0\r) \to_{\rm{fdd}}\l((2\gamma+1)\sqrt{\frac{\gamma+1}{\kappa(1-\rho)}}\mathcal B(t): t>0\r).
\]
\end{example}
}

{
\renewcommand{\thetheorem}{\ref{ex: log log zero}}
\addtocounter{theorem}{-1}
\begin{example}[Continued]
Under the exponential timescaling, we have
\[\l(\frac{S_{\floor{n^t}}}{\sqrt{n^t\log n\log\log n}}: t>0\r)\to_{\rm{fdd}}\l((2\gamma+1) \sqrt{\frac{(\gamma+1)t}{\kappa+\gamma+1}} \mathcal{B}(1) : t>0\r).
\]
For Brownian motion limit, we need the timescale 
\[\tau_n(t) = \l\lfloor\exp \l( \exp \l( t^{(\gamma+1)/(\kappa+\gamma+1)} \log \log n \r) \r) \r\rfloor.\]
Then, we have
\[
\l(\sqrt{ \frac{t^{{\kappa}/{(\kappa+\gamma+1)}}}{\tau_n(t) \log \tau_n(t) \log\log n}} S_{\tau_n(t)}: t>0 \r) \to_{\rm{fdd}} \l((2\gamma+1)\sqrt{\frac{\gamma+1}{\kappa+\gamma+1}} \mathcal B(t): t>0\r).
\]
\end{example}
}

{
\renewcommand{\thetheorem}{\ref{ex: log log log zero}}
\addtocounter{theorem}{-1}
\begin{example}[Continued]
Under the exponential timescale, we get
\[
\l(\frac{S_{\floor{n^{t}}}}{\sqrt{n^{t}\log n\log\log n\log\log\log n}} : t> 0 \r)\to_{\rm{fdd}}\l((2\gamma+1)\sqrt{t} \mathcal{B}(1): t>0\r).
\]
For Brownian motion limit, consider the timescale $\tau_n(t) = \floor{\exp(\exp((\log\log n)^t))}$. Then, we have \[
\l(\frac{S_{\tau_n(t)}}{\sqrt{\tau_n(t) \log \tau_n(t) ( \log \log n)^t \log\log\log n}} : t> 0 \r) \to_{\rm{fdd}} ((2\gamma+1) \mathcal B(t): t>0).
\]
\end{example}
}

Next consider Corollary~\ref{corr: nlogn scale} with the exponential timescale and a power function multiple of $\sqrt{n^t \log n}$ as the spatial normalization. The f.d.d.\ limit turns out to be the multiple of a time changed Brownian motion, with the time change given by a power law function and the multiple depending on $\alpha$. We also provide the timescale for the Brownian motion limit. The proof is similar and is omitted. Surely, the following Corollary applies to Examples~\ref{ex: zeta alpha zero pos}~-~\ref{ex: log log pos}.
\begin{corollary}[Continued from Corollary~\ref{corr: nlogn scale}]
    \label{cor: nlogn scale nonlinear}
    Let $\{S_n\}_{n\geq 0}$ be the SRRW-RVM with zero mean, unit variance innovation sequence $\{\xi_n\}_{n\ge 1}$. Then for $p=p_c$ and $\{\mu_n\}$ satisfying Assumption \ref{eg 1} with $\alpha+\gamma+1>0$, we have,
    \[
    \l( \sqrt{\frac{t^{{\alpha}/{(\gamma+1)}}}{ n^t \log n}} S_{\floor{n^t}}: t>0 \r) \to_{\rm{fdd}} \l( (2\gamma+1) \sqrt{\frac{\gamma+1}{\alpha+\gamma+1}} \mathcal B\l( t^{\frac{\alpha+\gamma+1}{\gamma+1}} \r) : t>0 \r).
    \]
    For the Brownian motion limit, consider the timescale $\tau_n(t) = \l\lfloor \exp \l( t^{\frac{\gamma+1}{\alpha+\gamma+1}} \log n\r) \r\rfloor$ to get
    \[
    \l( \sqrt{\frac{t^{{\alpha}/{(\alpha+\gamma+1)}}}{ \tau_n(t) \log n}} S_{\tau_n(t)}: t>0 \r) \to_{\rm{fdd}} \l( (2\gamma+1) \sqrt{\frac{\gamma+1}{\alpha+\gamma+1}} \mathcal B(t) : t>0 \r).
    \]
\end{corollary}

We conclude by considering Example~\ref{ex: lighter than n log n} in nonlinear timescales.
{
\renewcommand{\thetheorem}{\ref{ex: lighter than n log n}}
\addtocounter{theorem}{-1}
\begin{example}[Continued]
    Using Theorem~ \ref{rates l_n eg 1},we have, as $n\to\infty$,
    \[
    \frac{v_{\floor{n^t}}^2}{v_n^2} \sim t^\rho \exp \l( \frac1{\gamma+1} (\log n)^{1-\rho} (t^{1-\rho}-1) \r) \sim 
    \begin{cases}
        0, &\text{for $t<1$,}\\
        1, &\text{for $t=1$,}\\
        \infty, &\text{for $t>1$.}
    \end{cases}
    \]
    Thus, for the martingale difference sequence $\{\Delta \widecap{M}_{n,k}\}$, no nontrivial limit is possible in Proposition~\ref{lemma: time deformed M_n FCLT} with the exponential timescale $\floor{n^t}$. This establishes the limitation of the exponential timescale.

    However, for the timescale
    $\tau_n(t) = \l\lfloor \exp \l( \l( (\gamma+1) \log t + (\log n)^{1-\rho} \r)^{{1}/(1-\rho)} \r) \r\rfloor$, we get the Brownian motion limit:  
    \[
    \l(\sqrt{\frac{t}{\tau_n(t) (\log n)^\rho} } S_{\tau_n(t)}: t>0\r) \to_{\rm{fdd}} \l( (2\gamma+1) \sqrt{\frac{\gamma+1}{1-\rho} } \mathcal B(t) : t>0\r).
    \]
\end{example}
}

\section{Proof of Main Results for SRRW-RVM} \label{sec: SRRW}
The proofs of the main results given in Sections~\ref{sec: SLLN} and~\ref{ERW} make additional assumptions on the innovation sequence. We now extend these proofs to general innovation sequences under the appropriate moment assumptions. We indicate only the necessary modifications needed in the proofs of the results in Sections~\ref{sec: SLLN} and~\ref{ERW}. 
Note that the proofs of the results in Sections~\ref{as fluc} and~\ref{sec: cont} do not require any additional assumptions. 
\subsection{Law of large numbers for SRRW-RVM} \label{subsec: LLN general}
The results in Section~\ref{sec: SLLN} are proved under the finite variance assumption on the innovation sequence. To treat general zero mean innovations, we consider the centered and truncated step sizes given by
$\widetilde X_n := X_n \mathbbm{1}_{[|X_n|\le n]} - \mathbb{E} ( X_n \mathbbm{1}_{[|X_n|\le n]} | \mathcal F_{n-1} )$. Then the corresponding truncated versions of the martingales $M_n$ and $L_n$ are given by
$\widetilde M_n := \sum_{k=1}^n a_k \mu_k \widetilde X_k$ and
$\widetilde L_n := \sum_{k=1}^n \widetilde X_k$.

We first extend the proof of Lemma~\ref{lemma: trunc mart conv ERW}. It suffices to prove almost sure and $L^1$ convergence under the assumption of zero mean innovation sequence.
\begin{proof}[Proof of Lemma~\ref{lemma: trunc mart conv ERW}]
Note the sequence $\{a_n \nu_n\}$ is regularly varying of index $(1-p)(\gamma+1)$, is increasing and diverges to infinity. A careful restatement of Theorem~2.19 of~\cite{HallHeyde} gives 
\begin{equation} \label{eq: trunc mg}
\frac1{a_n \nu_n}{\widetilde M_n} \to 0 \quad \text{and} \quad \frac1{n}{\widetilde L_n} \to 0 \quad \text{almost surely and in $L^2$.}
\end{equation}

Due to the Borel-Cantelli lemma, only finitely many terms in the tail errors contribute: \begin{equation} \label{eq: tail error}
\sum_{n=1}^\infty a_n \mu_n |X_n| \mathbbm{1}_{[|X_n|>n]} < \infty \quad \text{and} \quad \sum_{n=1}^\infty |X_n| \mathbbm{1}_{[|X_n|>n]} < \infty \quad \text{almost surely}.
\end{equation}
Next, note that
\begin{align*}
\mathbb{E}(|X_n| \mathbbm{1}_{[|X_n|>n]} | \mathcal{F}_{n-1}) &= (1-p) \mathbb{E} (|\xi_1| \mathbbm{1}_{[|\xi_1|>n]}) + \frac{p}{\nu_{n-1}} \sum_{l=1}^{n-1} \mu_l |X_l| \mathbbm{1}_{[|X_l|>n]}\\
&= (1-p) \mathbb{E} (|\xi_1| \mathbbm{1}_{[|\xi_1|>n]}) + \frac{p}{\nu_{n-1}} \sum_{l=1}^{n-1} \mu_l |X_l| \mathbbm{1}_{[|X_l|>l]}.
\end{align*}
The first term goes to $0$ as $\mathbb{E}(|\xi_1|) < \infty$. By the Borel-Cantelli lemma, only finitely many terms contribute in the second sum, which is therefore negligible almost surely. Then, using Toeplitz lemma and  Cesaro averaging, we have
\begin{equation} \label{eq: cond exp cesaro}
    \frac1{a_n \nu_n} \sum_{k=1}^n a_k \mu_k \mathbb{E}(|X_k| \mathbbm{1}_{[|X_k|>k]} | \mathcal{F}_{k-1}) \to 0 \quad \text{and} \quad \frac1n \sum_{k=1}^n \mathbb{E}(|X_k| \mathbbm{1}_{[|X_k|>k]} | \mathcal{F}_{k-1}) \to 0
\end{equation}
almost surely.

Moreover, since $\mathbb{E} ( |X_n| \mathbbm{1}_{[|X_n|>n]}) = \mathbb{E} ( |\xi_1| \mathbbm{1}_{[|\xi_1|>n]} ) \to 0$, Toeplitz lemma and Cesaro averaging give 
\begin{equation} \label{eq: L1 cesaro}
\frac1{a_n\nu_n} \sum_{k=1}^n a_k \mu_k \mathbb{E} ( |X_k| \mathbbm{1}_{[|X_k|>k]} ) \to 0 \quad \text{and} \quad \frac1n \sum_{k=1}^n \mathbb{E} ( |X_k| \mathbbm{1}_{[|X_k|>k]}) \to 0.
\end{equation}

Finally, consider the decompositions
\begin{align*}
\frac{L_n}{n} &= \frac{ \widetilde L_n}{n} + \frac1n \sum_{k=1}^n X_k \mathbbm{1}_{[|X_k|>k]} - \frac1n \sum_{k=1}^n \mathbb{E} ( X_k \mathbbm{1}_{[|X_k|>k]} | \mathcal{F}_{k-1})\\
\frac{M_n}{a_n\nu_n} &= \frac{\widetilde{M}_n }{a_n\nu_n} + \frac{1}{a_n\nu_n}\sum_{k=1}^n a_k\mu_k X_k \mathbbm{1}_{[|X_k|>k]} -\frac{1}{a_n\nu_n} \sum_{k=1}^n a_k\mu_k\bb E (X_k \mathbbm{1}_{[|X_k|>k]}| \mathcal{F}_{k-1}) .
\end{align*}
Applying~\eqref{eq: trunc mg}~-~\eqref{eq: L1 cesaro}, we get the required almost sure and $L^1$ convergence for $L_n/n$ and $M_n/(a_n\nu_n)$.
\end{proof}

Due to Remark~\ref{rem: as ERW}, we need to consider Theorem~\ref{thm: SLLN} for process $L^1$ convergence only.
\begin{proof}[Proof of Theorem~\ref{thm: SLLN} (Process $L^1$ convergence for SRRW-RVM)]
For process $L^1$ convergence, we note that, using Proposition~2.3 of~\cite{Janson1994}, it is enough to show for any $T>0$, 
$S^*_{\lfloor nT \rfloor} / n
    \to_{L^1} 0$.
From~\eqref{conn L S} and truncation of the martingale $L_n$, we have
\begin{align*}
    \frac1n \mathbb{E} \left( S^*_{\lfloor nT \rfloor} \right)
    &\le \frac1n \mathbb{E} \left( \sup_{1\le l\le \lfloor nT \rfloor} \left| \widetilde L_l \right| \right) 
    + \frac2n \sum_{k=1}^{\lfloor nT \rfloor} \mathbb{E} \left( |X_k| \mathbbm{1}_{[|X_k|>k]} \right)
    + \frac1n \sum_{k=1}^{\lfloor nT \rfloor-1} \frac{\mathbb{E}(|M_k|)}{a_k \nu_k}\\
    &\le 2 \sqrt{\frac1{n^2} \mathbb{E} \widetilde L_{\lfloor nT \rfloor}^2}
    + \frac2n \sum_{k=1}^{\lfloor nT \rfloor} \mathbb{E} \left( |X_k| \mathbbm{1}_{[|X_k|>k]} \right)
    + \frac1n \sum_{k=1}^{\lfloor nT \rfloor-1} \frac{\mathbb{E}(|M_k|)}{a_k \nu_k},
\end{align*}
using Cauchy-Schwarz and Doob's $L^2$ inequalities. Then,~\eqref{eq: trunc mg},~\eqref{eq: L1 cesaro} and  Lemma~\ref{lemma: trunc mart conv ERW} make three terms on the right side negligible.
\end{proof}
\subsection{Subcritical Weak Convergence}
As noted in Remark~\ref{rem: subcrit ERW}, for Theorem~\ref{thm: subcrit}, we prove Lemmas~\ref{lem: QV} and~\ref{lem: Lindeberg ERW} for any zero mean, unit variance innovation sequence.
The following lemma will be useful towards that end.
\begin{lemma}{\label{lem: sq cond exp}}
     Let $\{S_n\}_{n\geq 0}$ be an SRRW-RVM with the zero mean innovation sequence $\{\xi_n\}$ satisfying $\bb E\xi_1^2=1$. Then $\mathbb{E} X_n^2 | \mathcal F_{n-1} \to 1$ almost surely and in $L^1$.
\end{lemma}
\begin{proof}
    Consider an SRRW-RVM with innovation sequence $\{\xi_n^2-1\}$ having finite first moment. Using induction, the steps are given by $\{X_n^2 - 1\}$ and the martingale sequence in~\eqref{eq: def M} by
    $a_n \sum_{k=1}^n \mu_k (X_k^2 -1) = a_n \sum_{k=1}^n \mu_k X_k^2 - a_n \nu_n$. Then, Lemma~\ref{lemma: trunc mart conv ERW} (under first moment assumption for the innovations only, as proved in Section~\ref{subsec: LLN general}) gives $\sum_{k=1}^n \mu_k X_k^2/\nu_n \to 1$ almost surely and in $L^1$. The result then follows by observing $\mathbb{E} X_{n+1}^2 | \mathcal{F}_{n} = p \sum_{k=1}^n \mu_k X_k^2 / \nu_n + (1-p)$.
\end{proof}

 We next obtain the limiting quadratic variation process for any zero mean, unit variance innovation sequence.
\begin{proof}[Proof of Lemma~\ref{lem: QV} for SRRW-RVM]
As noted in Remark~\ref{rem: QV sub ERW}, only the limits of the first terms of~\eqref{eq: qv omT}~--~\eqref{eq: qv tmT} are obtained for general zero mean, unit variance innovation sequence, and through Lemma~\ref{lem: sq cond exp} and Toeplitz lemma.
\end{proof}

To check the Lindeberg conditions for general innovations with finite second moment, we require the following lemma.
\begin{lemma} \label{lem: key Lind}
    Let $\{S_n\}_{n\geq 0}$ be an SRRW-RVM with the zero mean innovation sequence $\{\xi_n\}$ satisfying $\bb E\xi_1^2=1$. Then, for any sequence $\{\lambda_n\}$ increasing to $\infty$ and $\epsilon>0$, we have
    $$\mathbb{E} ( (\Delta L_n)^2 \mathbbm{1}_{\{ |\Delta L_n| > \epsilon \lambda_n \}} \mid \mathcal{F}_{n-1} ) \to_{L^1} 0.$$
\end{lemma}
\begin{proof}
    From~\eqref{Conditional Expectation} and~\eqref{conn L S}, we have
    $\Delta L_n = X_n - p {Y_{n-1}}/{\nu_{n-1}}$. Hence
\begin{align*}
    (\Delta L_n)^2 \mathbbm{1}_{\{ |\Delta L_n| > \epsilon \lambda_n \}} &\le 2X_n^2 \mathbbm{1}_{\left\{ |\Delta L_n| > \epsilon \lambda_n \right\}} + 2\frac{Y_{n-1}^2}{\nu_{n-1}^2} \\
    &\le 2X_n^2 \mathbbm{1}_{\left\{ |X_n| > \frac{\epsilon}2 \lambda_n \right\}} + 2X_n^2 \mathbbm{1}_{\left\{ \l|\frac{Y_{n-1}}{\nu_{n-1}}\r| > \frac{\epsilon}2 \lambda_n \right\}} + 2\frac{Y_{n-1}^2}{\nu_{n-1}^2},
\intertext{giving}
    \mathbb{E} \left( (\Delta L_n)^2 \mathbbm{1}_{\{ |\Delta L_n| > \epsilon \lambda_n \}} \right) 
    &\le  2\mathbb{E} \left( \xi_1^2 \mathbbm{1}_{\l\{ |\xi_1| >\frac{\epsilon}2 \lambda_n\r\} }\right) + 2\mathbb{E} \left(  |\mathbb{E}X_n^2 | \mathcal{F}_{n-1} -1 |\right) \\
    &\qquad \qquad+ 2 \mathbb{P} \left( \l| \frac{Y_{n-1}}{\nu_{n-1} } \r|  > \frac{\epsilon}2 \lambda_n \right) + 2\mathbb{E} \left( \frac{Y_{n-1}^2}{\nu_{n-1}^2} \right).
\end{align*}
Finite variance of $\xi_1$ and Lemma~\ref{lem: sq cond exp} respectively make the first two terms negligible, while the last two are so due to Lemma~\ref{lemma: trunc mart conv ERW}.
\end{proof}

We are now ready to obtain the Lindeberg conditions for the innovations with finite second moments.
 \begin{proof}[Proof of Lemma~\ref{lem: Lindeberg ERW} for SRRW-RVM]
Fix an integer $t>0$. From Lemma~\ref{lem: Linden bd ERW}, choose $\widetilde C_t>0$ and further choose $0<\rho < \gamma+1/2 - p(\gamma+1)$, when $p\in(\hp, p_c)$. We also use $\|\Delta\overline{\boldsymbol{T}}_{n,k}\|^2 = \tr(\omQ_{n,k}) (\Delta L_k)^2$ for $p\in [0, \hp)$,  
$\|\Delta{\boldsymbol{T}}_{n,k}\|^2 = \tr(\mQ_{n,k}) (\Delta L_k)^2$ for $p\in (\hp, p_c)$ and
$\|\Delta \tmT_{n,k}\|^2 = \tr(\tmQ_{n,k}) (\Delta L_k)^2$ for $p = \hp$. 
Then, we have
    \begin{align*}
    \intertext{for $p\in [0,\widehat{p})$,}
    \sum_{k=1}^{\lfloor nt\rfloor}\mathbb{E}\left(\|\Delta\overline{\boldsymbol{T}}_{n,k}\|^2\mathbbm{1}_{\{\|\Delta\overline{\boldsymbol{T}}_{n,k}\|>\epsilon\}}\middle|\mathcal{F}_{k-1}\right) 
    &\le  
    \sum_{k=1}^{\floor{nt}} \tr(\omQ_{n,k}) \bb E\l((\Delta L_k)^2 \mathbbm{1}_{\{|\Delta L_k|>\epsilon \sqrt{k}/\widetilde C_t\}} \middle| \mathcal{F}_{k-1} \r), \\
    \intertext{for $p\in (\hp, p_c)$,}
    \sum_{k=1}^{\floor{nt}} \bb E\l(\|\Delta \mT_{n,k}\|^2 \mathbbm{1}_{\{\|\Delta \mT_{n,k} \| >\epsilon \}} \middle| \mathcal F_{k-1}\r) 
    &\le  
    \sum_{k=1}^{\floor{nt} } \bb \tr(\mQ_{n,k}) E\l((\Delta L_k)^2 \mathbbm{1}_{ \{ |\Delta L_k| > \epsilon k^\rho/\widetilde C_t\} }\middle| \mathcal F_{k-1}\r), \\
    \intertext{and for $p=\hp$,}
    \sum_{k=1}^{\floor{nt}} \bb E\l(\|\Delta \tmT_{n,k}\|^2 \mathbbm{1}_{\{\|\Delta \tmT_{n,k} \| >\epsilon \}} \middle| \mathcal F_{k-1}\r) 
    &\le  
    \sum_{k=1}^{\floor{nt} } \tr(\tmQ_{n,k}) \bb E\l((\Delta L_k)^2 \mathbbm{1}_{ \{ |\Delta L_k| > \epsilon \sqrt[4]{k}/\widetilde C_t\} }\middle| \mathcal F_{k-1}\r).
    \end{align*}

    As in the proof of Lemma~\ref{lem: QV} for SRRW-R-RVM on p.~\pageref{eq: qv omT}, we have $\sum_{k=1}^{\floor{nt}}\tr(\omQ_{n,k})\to \tr(\omW(t))$ for $p\in [0, \hp)$, and, for each fixed $k$, $\tr(\omQ_{n,k})$ is negligible. Similar results hold for $\tr(\mQ_{n,k})$ and $\tr(\tmQ_{n,k})$, accordingly as $p \in (\hp, p_c)$ and $p=\hp$ respectively. We then conclude the proof, using Lemma~\ref{lem: key Lind} and Toeplitz lemma.
\end{proof}
\subsection{Weak limit in the critical regime with unbounded \texorpdfstring{$\{v_n\}$}{vn} for SRRW-RVM}
In the proof of Theorem~\ref{thm: crit unbdd} given in Section~\ref{subsec: critical weak limit}, the Rademacher distribution of the innovations is used only in Lemmas~\ref{prop: QV proc in the critical regime} and~\ref{lemma: cond lind cond in critical case}. We prove the Lemmas generally.

\begin{proof}[Proof of Lemma~\ref{prop: QV proc in the critical regime} for SRRW-RVM]
As noted in Remark~\ref{rem: qv crit ERW}, it is enough to update~\eqref{eq: qv cap M}. Using Lemma~\ref{lem: sq cond exp}, Toeplitz lemma and Karamata's theorem in the first term of~\eqref{eq: qv cap M}, we have, for $p=p_c$,

\[
\frac1{v_n^2} \sum_{k=1}^\ntfloor a_k^2 \mu_k^2 \bb{E}\l(X_k^2|\mathcal{F}_{k-1}\r) \sim \frac{v_{\floor{nt}}^2}{v_n^2} \to 1.
\]
\end{proof}
\begin{proof}[Proof of Lemma~\ref{lemma: cond lind cond in critical case} for SRRW-RVM]
Lemma~\ref{lemma: cond lind cond in critical case} follows by applying Lemma~\ref{lem: key Lind} and Toeplitz lemma to~\eqref{eq: Lind cap M}.   
\end{proof}

\section*{Acknowledgment}
The authors are grateful to Rahul Roy for providing valuable comments and insightful suggestions which greatly improved the presentation and the readability of the article.


\appendix

\section{}
In this appendix, we provide two auxiliary results, related to the regularly varying sequences $\{a_n\}$, $\{\mu_n\}$ and $\{\eta_n\}$. The first lemma considers the growth rate of certain sums.
\begin{lemma} \label{lem: rate sum}
For $p = \widehat{p}$, we have
\begin{align*}
\sum_{k=1}^n a_k \mu_k (\eta_n - {\eta}_k) &\sim (\gamma+1) n,\\ 
\sum_{k=1}^n a_k^2 \mu_k^2 (\eta_n - {\eta}_k) &\sim (\gamma+1)^2 a_n \nu_n,\\  
\text{and} \quad \sum_{k=1}^n a_k^2 \mu_k^2 (\eta_n - {\eta}_k)^2 &\sim 2 (\gamma+1)^2 n.
\end{align*}
\end{lemma}
\begin{proof}
    In the case $p=\hp$, we rewrite the factor $(\eta_n-\eta_k)$ as $\sum_{j=k}^{n-1} 1/{(a_j \nu_j)}$ and interchange the order of the summation, followed by applying Karamata's theorem twice in succession. The first two assertions then follow easily. We provide the details of the third assertion for illustration. In this case, we need extra care to handle the square of $(\eta_n-\eta_k)$. Note that
\begin{align}
    \sum_{k=1}^n a_k^2 \mu_k^2 (\eta_n-\eta_k)^2 &= \sum_{k=1}^{n-1} a_k^2 \mu_k^2 \sum_{j=k}^{n-1} \frac{1}{a_j^2\nu_j^2} + 2 \sum_{k=1}^{n-1} \sum_{i=k}^{n-1}\sum_{j=i+1}^{n-1} \frac{a_k^2\mu_k^2}{a_i\nu_ia_j\nu_j} \nonumber\\
    &= \sum_{j=1}^{n-1} \frac{1}{a_j^2\nu_j^2} v_j^2 +  2 \sum_{j=2}^{n-1} \sum_{i=1}^{j-1} \frac{v_i^2}{a_i\nu_ia_j\nu_j}. \label{eq: diff sq}
\end{align}
We get the required result as  the first term of~\eqref{eq: diff sq} satisfies
\[\sum_{j=1}^{n-1} v_j^2 / (a_j^2\nu_j^2) \sim (\gamma+1)^2 \sum_{j=1}^{n-1} j^{-1} \sim (\gamma+1)^2 \log n = \oh(n)\] using Karamata's theorem, while the second term of~\eqref{eq: diff sq} satisfies
\[
   2 \sum_{j=2}^{n-1} \sum_{i=1}^{j-1} \frac{v_i^2}{a_i\nu_ia_j\nu_j} \sim 2 (\gamma+1) \sum_{j=2}^{n-1} \sum_{i=1}^{j-1} \frac{a_i\mu_i}{a_j\nu_j} \sim 2 (\gamma+1)^2 \sum_{j=2}^{n-1} 1 \sim 2 (\gamma+1)^2 n,
\]
using Karamata's theorem twice in succession.
\end{proof}

The next result uses the uniform convergence of regularly varying functions on compact intervals.
\begin{lemma}
    \label{lem: tricky conv}
    For $p=\hp= {\gamma}/({\gamma+1})$, and for all $0<s<t<\infty$, we have
    $$\l(\eta_{\floor{nt}}-\eta_{\floor{ns}}\r) a_{\floor{ns}} \mu_{\floor{ns}} \to (\gamma+1) \log ({t}/{s}).$$
\end{lemma}
\begin{proof}
    For $p=\hp$, $\{a_n \mu_n\}$ is slowly varying. Then, by Proposition~0.5 of~\cite{Resnick1987}, for any $0<s<t<\infty$, we have
    \[
    \sup_{\floor{ns} \le k \le \floor{nt}} \left| \frac{a_{\floor{ns}} \mu_{\floor{ns}}}{a_{k} \mu_{k}} - 1\right| = \sup_{u\in[s,t]} \left| \frac{a_{\floor{ns}} \mu_{\floor{ns}}}{a_{\floor{nu}} \mu_{\floor{nu}}} - 1\right| \to 0.
    \]
    Using this uniform convergence, we get
    \[(\eta_{\floor{nt}}-\eta_{\floor{ns}}) a_{\floor{ns}} \mu_{\floor{ns}} = \sum_{k=\floor{ns}}^{\floor{nt}-1} (a_{\floor{ns}} \mu_{\floor{ns}})  / (a_k \nu_k) \sim (\gamma+1) \sum_{k=\floor{ns}}^{\floor{nt}-1} 1/k \sim (\gamma+1) \log ({t}/{s}).\]
\end{proof}

\end{document}